\pdfoutput=1
\documentclass[11pt]{amsart}
\usepackage[backend=biber,
	    url=false,
	    doi=false,
	    isbn=false,
	    date=year,
	    maxbibnames=99,
	    giveninits,
	    sortcites=true,
	    sorting=nty,
	    style=numeric]{biblatex}
\renewbibmacro{in:}{}
            
\usepackage[utf8]{inputenc}
\usepackage[T1]{fontenc}
\usepackage[text={445pt,660pt},centering]{geometry}
\usepackage[hidelinks,pdfusetitle,pdfauthor={Dylan Dronnier}]{hyperref}
\usepackage[activate={true,nocompatibility},final,tracking=true,kerning=true,factor=1100,stretch=10,shrink=10]{microtype}
\usepackage{braket}    
\usepackage{mathtools} 
\usepackage{amsthm}
\usepackage{amssymb}
\usepackage{mathrsfs}  
\usepackage{dsfont}    

\usepackage[shortlabels]{enumitem}
\usepackage[justification=centering]{caption}
\usepackage[margin=5pt,justification=centering,labelformat=simple,labelfont=sc]{subcaption}

\theoremstyle{plain}
\newtheorem{theorem}{Theorem}[section]
\newtheorem{corollary}[theorem]{Corollary}
\newtheorem{proposition}[theorem]{Proposition}
\newtheorem{lemma}[theorem]{Lemma}

\newtheorem{definition}[theorem]{Definition}
\newtheorem{hyp}{Assumption}

\theoremstyle{remark}
\newtheorem{remark}[theorem]{Remark}
\newtheorem{example}[theorem]{Example}

\newlist{propenum}{enumerate}{1} 
\setlist[propenum]{label=(\roman*)}

\DeclareMathOperator\spec{Spec}               
\DeclarePairedDelimiter\abs{\lvert}{\rvert}   
\DeclarePairedDelimiter\norm{\lVert}{\rVert}  

\newcommand{\ind}[1]{\mathds{1}_{#1}}

\newcommand{\kk}{\ensuremath{\mathrm{k}}}
\newcommand{\rd}{\ensuremath{\mathrm{d}}}
\newcommand{\ca}{\ensuremath{\mathscr{A}}}
\newcommand{\cl}{\ensuremath{\mathscr{L}^\infty}}
\newcommand{\cll}{\ensuremath{\mathcal{L}}}

\newcommand{\cb}{\ensuremath{\mathscr{B}}}
\newcommand{\ce}{\ensuremath{\mathscr{E}}}
\newcommand{\cf}{\ensuremath{\mathscr{F}}}
\newcommand{\cg}{\ensuremath{\mathscr{G}}}
\newcommand{\ch}{\ensuremath{\mathscr{H}}}
\newcommand{\ck}{\ensuremath{\mathscr{K}}}
\newcommand{\R}{\ensuremath{\mathbb{R}}}

\newcommand{\C}{\ensuremath{\mathbb{C}}}
\newcommand{\N}{\ensuremath{\mathbb{N}}}
\renewcommand{\P}{\ensuremath{\mathbb{P}}}

\newcommand{\E}{\ensuremath{\mathbb{E}}}

\newcommand{\cp}{\ensuremath{\mathcal{P}}}

\newcommand{\Tinf}{\ensuremath{\mathcal{T}}}

\newcommand{\etae}{\ensuremath{\eta^\mathrm{equi}}} 
\newcommand{\etauc}{\ensuremath{\eta^\mathrm{uni}_\mathrm{crit}}}
\newcommand{\etau}{\ensuremath{\eta^\mathrm{uni}}}

\newcommand{\esp}[1]{\mathbb{E}\left[#1\right]}
\newcommand{\gxx}{$[(\Omega, \cf, \mu), \kk]$}
\newcommand{\FF}{\ensuremath{\mathbf{F}}}
\newcommand{\as}{\text{ a.s.}}

\newcommand{\costu}{\ensuremath{C_\mathrm{uni}}} 

\newcommand{\optProb}[3]{%
  \begin{equation}\label{#1}%
    \begin{cases}
      \textbf{Minimize: } & #2 \\
      \textbf{subject to: } & #3
    \end{cases}
  \end{equation}%
}
\newcommand{\optProbminmax}[3]{%
  \begin{equation}\label{#1}%
    \begin{cases}
      \textbf{Minimize: } & #2 \\
      \textbf{subject to: } & #3
    \end{cases}
   \quad  \quad\text{and}\quad \quad \quad 
      \begin{cases}
      \textbf{Maximize: } & #2 \\
      \textbf{subject to: } & #3
    \end{cases}
  \end{equation}%
}

\newcommand{\optProbBis}[5]{%
  \begin{subequations}\label{#1}%
    \begin{alignat}{2}
	  &\! \textbf{Minimize: } & \quad & #4 \label{#2} \\
	  & \textbf{subject to: } & & #5 \label{#3}
        \end{alignat}
  \end{subequations}%
}

\newcommand{\optProbTer}[5]{%
  \begin{subequations}\label{#1}%
    \begin{alignat}{2}
	  &\! \textbf{Maximize: } & \quad & #4 \label{#2} \\
	  & \textbf{subject to: } & & #5 \label{#3}
        \end{alignat}
  \end{subequations}%
}

\date{\today}
\author{Jean-François Delmas}
\address{Jean-François Delmas,
  CERMICS, \'{E}cole des Ponts, France}
\email{jean-francois.delmas@enpc.fr}

\author{Dylan Dronnier}
\address{Dylan Dronnier,
  CERMICS, \'{E}cole des Ponts, France}
\email{dylan.dronnier@enpc.fr}

\author{Pierre-André Zitt}
\address{Pierre-André Zitt, LAMA, Université Gustave Eiffel, France}
\email{pierre-andre.zitt@univ-eiffel.fr}

\newsavebox{\largestimage}

\newcommand{\Cinf}{\ensuremath{C_\star}} 
\newcommand{\loss}{\ensuremath{\mathrm{L}}}
\newcommand{\CinfR}{\ensuremath{C_{\star, R_e}}} 
\newcommand{\CinfI}{\ensuremath{C_{\star, \I}}} 
\newcommand{\CinfL}{\ensuremath{C_{\star, \loss}}} 
\newcommand{\CsupL}{\ensuremath{C^{\star, \loss}}} 
\newcommand{\Csup}{\ensuremath{C^\star}} 
\newcommand{\lossup}{\ensuremath{\loss^\star}} 
\newcommand{\lossinf}{\ensuremath{\loss_\star}} 
\newcommand{\costa}{\ensuremath{C_\mathrm{aff}}} 
\newcommand{\costad}{\ensuremath{c_\mathrm{aff}}} 

\newcommand{\param}{\ensuremath{\mathrm{Param}}}
\newcommand{\maxcost}{c_{\max}}
\newcommand{\maxloss}{\ell_{\max}}
\newcommand{\gxxx}{$[(\Omega, \cf, \mu), k, \gamma]$}
\newcommand{\grR}{R_e[\kk]}
\newcommand{\grS}{\spec[\kk]}
\newcommand{\etainf}{\eta_{\text{\tiny SO}}}
\newcommand{\etasup}{\eta_{\text{\tiny NE}}}
\newcommand{\vg}{\ensuremath{v_\mathrm{g}}}
\newcommand{\vd}{\ensuremath{v_\mathrm{d}}}
\newcommand{\I}{\ensuremath{\mathfrak{I}}}
\newcommand{\un}{\ensuremath{\mathds{1}}}

\title{Targeted vaccination strategies for an infinite-dimensional SIS model}

\begin{document}

\thanks{This work is partially supported by Labex Bézout reference ANR-10-LABX-58}



\begin{abstract}
  We formalize and study the problem of optimal allocation strategies for a (perfect)
  vaccine in the infinite-dimensional SIS model. The question may be viewed as a
  bi-objective minimization problem, where one tries to minimize simultaneously the cost
  of the vaccination, and a loss that may be either the effective reproduction number, or
  the overall proportion of infected individuals in the endemic state. We prove the
  existence of Pareto optimal strategies for both loss functions.

  We also show that
  vaccinating according to the profile of the endemic state is a critical allocation, in
  the sense that, if the initial reproduction number is larger than 1, then this
  vaccination strategy yields an effective reproduction number equal to~$1$.
\end{abstract}

\maketitle

\section{Introduction}

\subsection{Motivation}

Increasing the prevalence of immunity from contagious disease in a population limits the
circulation of the infection among the individuals who lack immunity. This so-called
``herd effect'' plays a fundamental role in epidemiology as it has had a major impact in
the eradication of smallpox and rinderpest or the near eradication of poliomyelitis; see
\cite{HerdImmunityFine2011}. Targeted vaccination strategies, based on the heterogeneity
of the infection spreading in the population, are designed to increase the level of
immunity of the population with a limited quantity of vaccine. These strategies rely on
identifying groups of individuals that should be vaccinated in priority in order to slow
down or eradicate the disease.

In this article, we establish a theoretical framework to study targeted vaccination
strategies for the deterministic infinite-dimensional SIS model introduced in
\cite{delmas_infinite-dimensional_2020}, that encompasses as particular cases the SIS
model on graphs or on stochastic block models. In  companion papers, we provide a
series of general and specific examples that complete and illustrate the present work: see
Section~\ref{sec:suite} for more detail.

\subsection{Herd immunity and targeted vaccination strategies}\label{subsec:problem}

Let us start by recalling a few classical results in mathematical epidemiology; we refer
to Keeling and Rohani's monograph~\cite{keeling_modeling_2008} for an extensive introduction to
this field, including details on the various classical models (SIS, SIR, etc.)

In an homogeneous population, the basic reproduction number of an infection, denoted
by~$R_0$, is defined as the number of secondary cases one individual generates on average
over the course of its infectious period, in an otherwise uninfected (susceptible)
population. This number plays a fundamental role in epidemiology as it provides a scale to
measure how difficult an infectious disease is to control. Intuitively, the disease should
die out if~$R_0<1$ and invade the population if~$R_0>1$. For many classical mathematical
models of epidemiology, such as SIS or S(E)IR, this intuition can be made rigorous: the
quantity~$R_0$ may be computed from the parameters of the model, and the threshold
phenomenon occurs.

Assuming~$R_0>1$ in an homogeneous population, suppose now that only a proportion~$\etau$
of the population can catch the disease, the rest being immunized. An infected
individual will now only generate~$\etau R_0$ new cases, since a proportion~$(1-\etau)$ of
previously successful infections will be prevented. Therefore, the new \emph{effective
reproduction number} is equal to~$R_e(\etau) = \etau R_0$. This fact led to the recognition
by Smith in 1970 \cite{smith_prospects_1970} and Dietz in 1975
\cite{dietz1975transmission} of a simple threshold theorem: the incidence of an infection
declines if the proportion of non-immune individuals is reduced below~$\etauc = 1/R_0$.
This effect is called \emph{herd immunity}, and the corresponding percentage~$1-\etauc$ of
people that have to be vaccinated is called \emph{herd immunity threshold}; see for
instance~\cite{smith_concepts_2010, somerville_public_2016}.

It is of course unrealistic to depict human populations as homogeneous, and many
generalizations of the homogeneous model have been studied; see \cite[Chapter
3]{keeling_modeling_2008} for examples and further references. For most of these
generalizations, it is still possible to define a meaningful reproduction number~$R_0$, as
the number of secondary cases generated by a \textit{typical} infectious individual when
all other individuals are uninfected; see~\cite{Diekmann1990}. After a vaccination
campaign, let the vaccination strategy $\eta$ denote the (non necessarily homogeneous)
proportion of the \textbf{non-vaccinated} population, and let the effective reproduction
number $R_e(\eta)$ denote the corresponding reproduction number of the non-vaccinated
population. The vaccination strategy $\eta$ is \emph{critical} if
$R_e(\eta) =  1$.
The
possible choices of $\eta$ naturally raises a question that may be expressed as the
following informal optimization problem:
\optProb{eq:informal_optim1}{%
  \text{the quantity of vaccine to administrate}}{\text{herd immunity is reached, that
is,~$R_e\leq 1$.}}
If the quantity of available vaccine is limited, then one is also
interested in:
\optProb{eq:informal_optim2}{%
  \text{the effective reproduction number~$R_e$}}{\text{a given quantity of available
vaccine.}}
Interestingly enough, the strategy~$\etauc$, which consists in delivering the
vaccine \emph{uniformly} to the population, without taking inhomogeneity into account,
leaves a proportion~$\etauc= 1/R_0$ of the population unprotected, and
is therefore critical since~$R_e(\etauc) =1$. In particular it is  admissible for the optimization
problem~\eqref{eq:informal_optim1}.

However, herd immunity may be achieved even if the proportion of unprotected people is
\emph{greater} than~$1/R_0$, by targeting certain group(s) within the population; see
Figure~3.3 in \cite{keeling_modeling_2008}. For example, the discussion of vaccination
control of gonorrhea in~\cite[Section~4.5]{hethcote} suggests that it may be better to
prioritize the vaccination of people that have already caught the
disease: this lead us to consider a vaccination strategy
guided by the equilibrium state. This strategy denoted by~$\etae$ will be defined formally
below. Let us mention here an observation in the same vein made by Britton, Ball and
Trapman in \cite{AMathematicalBritto2020}. Recall that in the S(E)IR model, immunity can
be obtained through infection. Using parameters from real-world data, these authors
noticed that the disease-induced herd immunity level can, for some models, be
substantially lower than the classical herd immunity threshold~$1 - 1/R_0$. This can be
reformulated in term of targeted vaccination strategies: prioritizing the individuals that
are more likely to get infected in a S(E)IR epidemic may be more efficient than
distributing uniformly the vaccine in the population.

\medskip

The main goal of this paper is two-fold: formalize the optimization problems
\eqref{eq:informal_optim1} and~\eqref{eq:informal_optim2} for a particular infinite
dimensional SIS model, recasting them more generally as a bi-objective optimization
problem; and give existence and properties of solutions to this bi-objective problem. We
will also consider a closely related problem, where one wishes to minimize the size of the
epidemic rather than the reproduction number. We will in passing provide insight on the
efficiency of classical vaccination strategies such as~$\etauc$ or~$\etae$.

\subsection{Literature on targeted vaccination strategies}

Targeted vaccination problems have mainly been studied using two different mathematical
frameworks.

\subsubsection{On meta-populations models}\label{sec:metapop}

Problems~\eqref{eq:informal_optim1} and~\eqref{eq:informal_optim2} have been examined in
depth for deterministic \emph{meta-population} models, that is, models in which an
heterogeneous population is stratified into a finite number of homogeneous sub-populations
(by age group, gender, \ldots). Such models are specified by choosing the sizes of the
subpopulations and quantifying the degree of interactions between them, in terms of
various mixing parameters. In this setting,~$R_0$ can often be identified as the spectral
radius of a \emph{next-generation matrix} whose coefficients depend on the subpopulation
sizes, and the mixing parameters. It turns out that the next generation matrices take
similar forms for many dynamics (SIS, SIR, SEIR,...); see the discussion in~\cite[Section
10]{hill-longini-2003}. Vaccination strategies are defined as the levels at which each
sub-population is immunized. After vaccination, the next-generation matrix is changed and
its new spectral radius corresponds to the effective reproduction number~$R_e$.

Problem~\eqref{eq:informal_optim1} has been studied in this setting by Hill and Longini
\cite{hill-longini-2003}. These authors study the geometric properties of the so-called
threshold hypersurface, that is the vaccination allocations for which~$R_e = 1$. They also
compute the vaccination belonging to this surface with minimal cost for an Influenza A
model. Making structural assumptions on the mixing parameters, Poghotayan, Feng, Glasser
and Hill derive  in \cite{poghotanyan_constrained_2018}  an analytical formula for the
solutions of Problem~\eqref{eq:informal_optim2}, for populations
divided in two groups. Many papers also
contain numerical studies of the optimization problems~\eqref{eq:informal_optim1}
and~\eqref{eq:informal_optim2} on real-world data using gradient techniques or similar
methods; see for example \cite{DistributionOfGoldst2010, feng_elaboration_2015,
TheMostEfficiDuijze2016, feng_evaluating_2017, IdentifyingOptZhao2019}.

\medskip

Finally, the effective reproduction number is not the only reasonable way of quantifying a
population's vulnerability to an infection. For an SIR infection for example, the
proportion of individuals that eventually catch (and recover from) the disease, often
referred to as the \emph{attack rate}, is broadly used. We refer
to~\cite{TheMostEfficiDuijze2016, DoseOptimalVaDuijze2018} for further discussion on this
topic.

\subsubsection{On networks}

Whereas the previously cited works typically consider a small number of subpopulations,
often with a ``dense'' structure of interaction (every subpopulation may directly infect
all the others), other research communities have looked into a similar problem for graphs.
Indeed, given a (large), possibly random graph, with epidemic dynamics on it, and
supposing that we are able to suppress vertices by vaccinating, one may ask for the best
way to choose the vertices to remove.

The importance of the spectral radius of the network has been rapidly identified as its
value determines if the epidemic dies out quickly or survives for a long time
\cite{TheEffectOfNGaneshNone, CharacterizingRestre2006}. Since Van Mieghem \textit{et al.}
proved in~\cite{DecreasingTheVanMi2011} that the problem of minimizing spectral radius of
a graph by removing a given number of vertices is NP-complete (and therefore unfeasible in
practice), many computational heuristics have been put forward to give approximate
solutions; see for example~\cite{ApproximationASaha2015} and references therein.

\subsection{Main results}

The differential equations governing the epidemic dynamics in meta-population~SIS models
were developed by Lajmanovich and Yorke in their pioneer
paper~\cite{lajmanovich1976deterministic}. In \cite{delmas_infinite-dimensional_2020}, we
introduced a natural generalization of their equation, which can also be viewed as the
limit equation of the stochastic SIS dynamic on network, in an infinite-dimensional
space~$\Omega$, where~$x \in \Omega$ represents a feature and the probability
measure~$\mu(\mathrm{d} x)$ represents the fraction of the population with feature~$x$.

\subsubsection{Regularity of the effective reproduction function~$R_e$}

We consider the effective reproduction function in a general operator framework which we
call the \emph{kernel model}. This model is characterized by a probability space~$(\Omega,
\cf, \mu)$ and a measurable non-negative kernel~$\kk: \Omega \times \Omega \rightarrow
\R_+$. Let $T_\kk$ be the corresponding integral operator defined by:
\[
  T_\kk(h)(x) = \int_\Omega \kk(x,y) h(y) \, \mu(\mathrm{d}y).
\]
In the setting of~\cite{delmas_infinite-dimensional_2020} (see in particular Equation~(11)
therein), $T_\kk$ is the so-called \emph{next generation operator}, where the kernel~$\kk$
is defined in terms of a transmission rate kernel $k(x,y)$ and a recovery rate
function~$\gamma$ by the product~$\kk(x,y)=k(x,y)/\gamma(y)$; and the reproduction
number~$R_0$ is then the spectral radius $\rho(T_\kk)$ of~$T_\kk$. \medskip

Following~\cite[Section~5]{delmas_infinite-dimensional_2020}, we represent a vaccination
strategy by a function~$\eta: \Omega \rightarrow [0, 1]$, where~$\eta(x)$ represents the
fraction of \textbf{non-vaccinated} individuals with feature~$x$; the effective
reproduction number associated to~$\eta$ is then given by
\begin{equation}
  \label{eq:def-Re-intro}
  R_e(\eta) = \rho(T_{\kk \eta}),
\end{equation}
where~$\rho$ stands for the spectral radius and~$\kk\eta$ stands for the kernel~$ ( \kk
\eta)(x,y)=\kk(x,y) \eta(y)$. If $R_0\geq 1$, then a  vaccination
strategy~$\eta$ is called \emph{critical} if it achieves precisely the herd immunity threshold, that
is~$R_e(\eta) = 1$. 

In particular, the ``strategy'' that consists in vaccinating no one corresponds to $\eta
\equiv \un$, and of course~$R_e(\un) = R_0$. As the spectral radius is positively homogeneous,
we also get, when~$R_0\geq 1$, that the uniform strategy that corresponds to the constant
function:
\[
  \etauc\equiv\frac{1}{R_0}
\]
is critical, as~$R_e(\etauc)=1$. This
is consistent with results obtained in the homogeneous model given in
Section~\ref{subsec:problem}.

\medskip

Let~$\Delta$ be the set of strategies, that is the set of~$[0,1]$-valued functions defined
on~$\Omega$. The usual technique to obtain the existence of solutions to optimization
problems like~\eqref{eq:informal_optim1} or~\eqref{eq:informal_optim2} is to prove that
the function~$R_e$ is continuous with respect to a topology for which the set of
strategies~$\Delta$ is compact. It is natural to try and prove this continuity by
writing~$R_e$ as the composition of the spectral radius~$\rho$ and the map~$\eta \mapsto
T_{\kk\eta}$. The spectral radius is indeed continuous at compact operators (and $T_{\kk
\eta}$ is in fact compact under a technical integrability assumption on the kernel~$\kk$
formalized on page~\pageref{hyp:k} as Assumption~\ref{hyp:k}), if we endow the set of
bounded operators with the operator norm topology; see \cite{newburgh1951, burlando}.
However, this would require choosing the uniform topology on~$\Delta$,
which then is not
compact.

We instead  endow~$\Delta$ with the weak topology, see Section \ref{sec:weak}, for which
compactness holds; see Lemma~\ref{lem:D-compact}. This forces us to equip the space of
bounded operators with the strong topology, for which the spectral radius is in general
not continuous; see \cite[p.~431]{kato2013perturbation}. However, the family of
operators~$(T_{\kk\eta}, \, \eta \in \Delta)$ is \emph{collectively compact} which enables
us to recover continuity, using a serie of results obtained by Anselone~\cite{anselone}.
This leads to the following result, proved in
Theorem~\ref{th:continuity-R} below. We recall that Assumption~\ref{hyp:k}, formulated on
page~\pageref{hyp:k}, provides an integrability condition on the kernel~$\kk$.

\begin{theorem}[Continuity of the spectral radius]
  Under Assumption~\ref{hyp:k} on the kernel~$\kk$, the function~$R_e \, \colon \, \Delta
  \to \R_+$ is continuous with respect to the weak topology on~$\Delta$.
\end{theorem}

In fact, we also prove the continuity of the spectrum with respect to the Hausdorff
distance on the set of compact subsets of~$\C$. We shall write~$R_e[\kk]$ to stress the
dependence of the function~$R_e$ in the kernel~$\kk$. In Proposition \ref{prop:Re-stab},
we prove the stability of~$R_e$, by giving natural sufficient conditions on a sequence of
kernels~$(\kk_n, n\in \N)$ converging to~$\kk$ which imply that~$R_e[\kk_n]$ converges
uniformly towards~$R_e[\kk]$. This result has both theoretical and practical interest: the
next-generation operator is unknown in practice, and has to be estimated from data. Thanks
to this result, the value of $R_e$ computed from the estimated operator should converge to
the true value.

\subsubsection{On the maximal endemic equilibrium in the SIS model}

We consider the \emph{SIS model} from \cite{delmas_infinite-dimensional_2020}. This model
is characterized by a probability space $(\Omega, \cf,\mu)$, the transmission kernel~$k \,
\colon \, \Omega \times \Omega \to \R_+$ and the recovery rate~$\gamma \, \colon \, \Omega \to
\R_+^*$. We suppose in the following that the technical Assumption~\ref{hyp:k-g},
formulated on page~\pageref{hyp:k-g}, holds, so that the SIS dynamical
evolution is well defined.

This evolution is encoded as~$u=(u_t, t\in \R_+)$, where~$u_t\in \Delta$ for all~$t$ and
$u_t(x)$ represents the probability of an individual with feature~$x\in \Omega$ to be
infected at time~$t\geq 0$, and follows the equation:
\begin{equation}
  \label{eq:SIS-intro}
  \partial_t u_t = F(u_t)\quad\text{for } t\in \R_+,
  \quad\text{where}\quad
  F(g) = (1 - g) \Tinf_k (g) - \gamma g\quad\text{for } g\in \Delta,
\end{equation}
with an initial condition~$u_0 \in \Delta$ and with $\Tinf_k$ the integral operator
corresponding to the kernel $k$ acting on the set of bounded measurable
functions, see \eqref{eq:def-Tk}. It is
proved in~\cite{delmas_infinite-dimensional_2020} that such a solution~$u$ exists and is
unique under Assumption~\ref{hyp:k-g}. An \emph{equilibrium} of~\eqref{eq:SIS-intro} is a
function~$g \in \Delta$ such that $F(g) = 0$. According to
\cite{delmas_infinite-dimensional_2020}, there exists a maximal
equilibrium~$\mathfrak{g}$, \textit{i.e.}, an equilibrium such that all other equilibria
$h \in \Delta$ are dominated by~$\mathfrak{g}$:~$h \leq \mathfrak{g}$. Furthermore, we
have~$R_0\leq 1$ if and only if~$\mathfrak{g}=0$. In the connected case (for example if
$k>0$), then~$0$ and~$\mathfrak{g}$ are the only equilibria; besides~$\mathfrak{g}$ is the
long-time distribution of infected individuals in the
population:~$\lim_{t\rightarrow+\infty } u_t = \mathfrak{g}$ as soon as the initial
condition is non-zero; see \cite[Theorem~4.14]{delmas_infinite-dimensional_2020}.

\medskip

As hinted in \cite[Section~4.5]{hethcote} for vaccination control of gonorrhea, it is
interesting to consider vaccinating people with feature $x$ with
probability~$\mathfrak{g}(x)$; this corresponds to the strategy based on the maximal
equilibrium:
\[
 \etae=1 -\mathfrak{g}.
\]
The following result entails that this strategy is critical and thus achieves the herd
immunity threshold. 
Recall that
Assumption~\ref{hyp:k-g}, formulated page~\pageref{hyp:k-g}, provides technical conditions
on the parameters~$k$ and~$\gamma$ of the SIS model. The effective
reproduction number of the SIS model is the function $R_e$ defined in
\eqref{eq:def-Re-intro} with the kernel $\kk=k/\gamma$.

\begin{theorem}[The maximal equilibrium yields a critical vaccination]
 Suppose Assumption~\ref{hyp:k-g} holds. If~$R_0\geq 1$, then the
 vaccination strategy $\etae$ is critical, that is,~$R_e(\etae)=1$.
\end{theorem}
This result will be proved below as a part of Proposition~\ref{prop:caract-g}.
Let us finally describe informally another consequence of
this Proposition.
We were able to prove
in~\cite[Theorem~4.14]{delmas_infinite-dimensional_2020} that, in the connected case,
if~$R_0>1$, the disease-free equilibrium~$u=0$ is unstable.
Proposition~\ref{prop:caract-g} gives spectral information on the formal linearization of
the dynamics~\eqref{eq:SIS-intro} near any equilibrium~$h$; in particular if~$h\neq
\mathfrak{g}$ then~$h$ is linearly unstable.

\subsubsection{Regularity of the total proportion of infected population function~$\I$}

According to \cite[Section~5.3.]{delmas_infinite-dimensional_2020}, the SIS equation with
vaccination strategy~$\eta$ is given by~\eqref{eq:SIS-intro}, where~$F$ is replaced by
$F_\eta$ defined by:
\[
  F_\eta(g) = (1-g)T_{k\eta}(g) - \gamma g.
\]
and~$u_t$ now describes the proportion of infected \emph{among the non-vaccinated
population}. We denote by $\mathfrak{g}_\eta$ the corresponding maximal equilibrium (thus
considering~$\eta\equiv1$ gives $\mathfrak{g}=\mathfrak{g}_1$), so
that~$F_\eta(\mathfrak{g}_\eta)=0$. Since the probability for an individual~$x$ to be
infected in the stationary regime is~$\mathfrak{g}_\eta(x) \, \eta(x)$, the \emph{fraction
of infected individuals at equilibrium},~$\I(\eta)$, is thus given by:
\begin{equation}
  \label{eq:def-I-intro}
  \I (\eta)=\int_\Omega \mathfrak{g}_\eta
  \, \eta\, \mathrm{d}\mu=\int_\Omega \mathfrak{g}_\eta(x)
  \, \eta(x)\, \mu(\mathrm{d}x).
\end{equation}
As mentioned above, for a SIR model, distributing vaccine so as to minimize the attack
rate is at least as natural as trying to minimize the reproduction number, and this
problem has been studied for example in~\cite{TheMostEfficiDuijze2016,
DoseOptimalVaDuijze2018}. In the SIS model the quantity~$\I$ appears as a natural analogue
of the attack rate, and is therefore a natural optimization objective.

\medskip

We obtain results on~$\I$ that are very similar to the ones on~$R_e$. Recall that
Assumption~\ref{hyp:k-g} on page~\pageref{hyp:k-g} ensures that the infinite-dimensional
SIS model, given by equation \eqref{eq:SIS-intro}, is well defined. The next theorem
corresponds to Theorem~\ref{th:continuity-I}.

\begin{theorem}[Continuity of the equilibrium infection size]
 Under Assumption~\ref{hyp:k-g}, the function~$\I \, \colon \, \Delta \to \R_+$ is
 continuous with respect to the weak topology on~$\Delta$.
\end{theorem}

In Proposition~\ref{prop:I-stab}, we prove the stability of~$\I$, by giving natural
sufficient condition on a sequence of kernels and functions $((k_n, \gamma_n), n\in \N)$
converging to~$(k, \gamma)$ which imply that~$\I[k_n, \gamma_n]$ converges uniformly
towards~$\I[k,\gamma]$. We also prove that the loss functions~$\loss=R_e$ and~$\loss=\I$
are both non-decreasing ($\eta\leq \eta'$ implies~$\loss(\eta)\leq \loss(\eta')$), and
sub-homogeneous ($ \loss(\lambda \eta)\leq \lambda \loss(\eta)$ for all $\lambda\in
[0,1]$); see Propositions~\ref{prop:R_e} and~\ref{prop:I}.

\subsubsection{Optimizing the protection of the population}

Consider a cost function~$C \, \colon \, \Delta \to [0,1]$ which measures the cost for the
society of a vaccination strategy (production and diffusion). Since the vaccination
strategy~$\eta$ represents the non-vaccinated population, the cost function~$C$ should be
decreasing (roughly speaking~$\eta<\eta'$ implies $C(\eta)>C(\eta')$; see Definition
\ref{def:cont, monot}). We shall also assume that~$C$ is continuous with respect to the
weak topology on~$\Delta$, and that doing nothing costs nothing, that is,~$C(\un)=0$. A
simple and natural choice is the uniform cost~$\costu$ given by the overall proportion of
vaccinated individuals:
\[
  \costu (\eta)=\int_\Omega (1-\eta)\, \mathrm{d}\mu=1- \int_\Omega \eta\, \mathrm{d}\mu.
\]
See Remark~\ref{rem:costa-costu} for comments on other examples of cost functions.

\medskip

Our problem may now be seen as a bi-objective minimization problem: we wish to minimize
both the loss~$\loss(\eta)$ and the cost~$C(\eta)$, subject to $\eta \in \Delta$, with the
loss function~$\loss$ being either~$R_e$ or $\I$. Following classical terminology for
multi-objective optimisation problems \cite{NonlinearMultiMietti1998}, we call a
strategy~$\eta_\star$ \emph{Pareto optimal} if no other strategy is strictly better:
\[
  C(\eta)< C(\eta_\star) \implies \loss(\eta) > \loss(\eta_\star)
  \quad\text{and}\quad
  \loss(\eta)< \loss(\eta_\star) \implies C(\eta) > C(\eta_\star).
\]
The set of Pareto optimal strategies will be denoted by~$\cp_\loss$, and we define the
\emph{Pareto frontier} as the set of Pareto optimal outcomes:
\[
  \mathcal{F}_\loss =
  \{ (C(\eta_\star),\loss(\eta_\star)) \, \colon \, \eta_\star \in
  \mathcal{P}_\loss \}.
\]
Notice that, with this definition, the Pareto frontier is empty when there is no Pareto
optimal strategy.

For any strategy~$\eta$, the cost and loss of~$\eta$ vary between the following bounds:
\begin{align*}
 & 0 = C(\un) \leq C(\eta) \leq C(0) = \maxcost = \text{cost of vaccinating the whole population},\\
 & 0 =\loss(0)\leq \loss(\eta) \leq \loss(\un) = \maxloss= \text{loss incurred in the absence of
 vaccination}.
\end{align*}
Let~$\lossinf$ be the \emph{optimal loss} function and~$\CinfL$ the \emph{optimal} cost
function defined by:
\begin{align*}
  \lossinf(c) &=
  \inf \, \set{ \loss(\eta) \, \colon \, \eta \in \Delta, \, C(\eta) \leq c } \quad
  \text{for $c\in [0,\maxcost]$}, \\
  \CinfL(\ell) &=
  \inf \, \set{ C(\eta) \, \colon \, \eta \in \Delta,\, \loss(\eta) \leq \ell } \quad
  \text{for $\ell \in [0,\maxloss]$}.
\end{align*}
We simply write $\Cinf$ for $\CinfL$ when
no confusion on the loss function can arise.
Proposition~\ref{prop:f_properties} (in a more general framework in
particular for the cost function) and
Lemma~\ref{lem:c-dec+L-hom} states that the Pareto frontier is non empty and has a
continuous parametrization for the cost $C=\costu$ and the loss $\loss=R_e$ or $\loss=\I$;
see Figure~\ref{fig:pareto_frontier} below for a visualization of the Pareto frontier.

\begin{theorem}[Properties of the Pareto frontier]
  For the kernel model with loss function $\loss=R_e$ or the SIS model
  with $\loss\in \{R_e, \I\}$, and the uniform cost function $C=\costu$,
  the
  function~$\CinfL$ is continuous and decreasing on~$[0, \maxloss]$, the
  function~$\lossinf$ is continuous on~$[0, \maxcost]$ decreasing on~$[0, \CinfL(0)]$ and
  zero on~$[\CinfL(0),\maxcost]$; furthermore the Pareto frontier is
  connected  and:
  \[ \mathcal{F}_\loss = \{(c,\lossinf(c)) \, \colon \, c \in [0,\CinfL(0)]\} =
  \{(\CinfL(\ell), \ell) \, \colon \, \ell \in [0,\maxloss]\} . \]
\end{theorem}

We also establish that $\cp_\loss$ is compact in~$\Delta$ for the weak topology in
Corollary \ref{cor:P=K}; that the set of outcomes or feasible region $\FF=\{ (C(\eta),\loss(\eta)), \,
\eta\in \Delta\}$ has no holes in Proposition~\ref{prop:trou}; and that the Pareto
frontier is convex if~$C$ and~$\loss$ are convex in Proposition~\ref{prop:cvex}. We study
in Proposition~\ref{prop:F-stab} the stability of the Pareto frontier and the set of
Pareto optima when the parameters vary. \medskip

In a sense the Pareto optimal strategies are intuitively the ``best'' strategies.
Similarly, we also study the ``worst'' strategies, which we call anti-Pareto optimal
strategies, and describe the corresponding anti-Pareto frontier. Understanding the ``worst
strategies'' also helps to avoid pitfalls when one has to consider
sub-optimal strategies: for example, we prove  in \cite{ddz-Re}  that disconnecting strategies are not
the ``worst'' strategies, and we provide in \cite[Section~4]{ddz-reg}  an elementary example
where the same strategies can be ``best'' or ``worst'' according to model parameters
values. Surprisingly, proving properties of the anti-Pareto frontier
sometimes necessitates stronger assumptions than in the Pareto case:
for example, the connectedness of the anti-Pareto frontier is only
proved under a \emph{quasi-irreducibility} assumption on the kernel, 
see Lemmas \ref{lem:L**} and~\ref{lem:L**=I}.

\begin{remark}[Eradication strategies do not depend on the loss]\label{rem:CR1=CI0}%
  In \cite{delmas_infinite-dimensional_2020}, we proved that, for all~$\eta\in\Delta$, the
  equilibrium infection size~$\I(\eta)$ is non zero if and only
  if~$R_e(\eta)>1$. Consider the uniform cost   $C=\costu$. First,
  this implies that~$\cp_\I$ is a subset of~$\{ \eta\in \Delta\, \colon\, R_e(\eta)\geq
  1\}$. Secondly, a vaccination strategy~$\eta \in \Delta$ is Pareto optimal for the
  objectives~$(R_e, C)$ and satisfies~$R_e(\eta) = 1$ if and only if~$\eta$ is Pareto
  optimal for the objectives~$(\I, C)$ and satisfies~$\I(\eta) =0$:
  \begin{equation}
    \label{eq:PJ-PRe} \eta\in \cp_{R_e}\text{ and } R_e(\eta)=1 \quad \Longleftrightarrow
    \quad \eta\in \cp_{\I}\text{ and } \I(\eta)=0.
  \end{equation}
\end{remark}

\begin{remark}[Minimal cost of eradication]
  Assume~$R_0 > 1$ and the uniform cost $C=\costu$. The equivalence~\eqref{eq:PJ-PRe} implies directly that:
  \[
    \CinfR(\un)= \CinfI(0).
  \]
  Thus, this latter quantity can be seen as the minimal cost (or minimum percentage of
  people that have to be vaccinated) required to eradicate the infection. Recall the
  critical vaccination strategies~$\etauc \equiv 1/R_0$ and~$\etae=1- \mathfrak{g}$
  (as~$R_e(\etauc)=R_e(\etae)=1$).  Since~$C(\etauc)=1- 1/R_0$ and~ $C(\etae)= \int_\Omega \mathfrak{g}
  \, \mathrm{d}\mu=\I (\un)$, we obtain the following upper bounds of the minimal cost
  required to eradicate the infection:
  \[
    \CinfR(\un)=\CinfI(0)\leq \min \left(1-\frac{1}{R_0}\, , \,
    \int_\Omega \mathfrak{g} \, \mathrm{d}\mu \right).
  \]
\end{remark}

\subsubsection{Equivalence of models}
  Our last  results
address a  natural question stemming from  our choice of a  very general
framework  to modelize  the infection.   Since our  models are  infinite
dimensional  and  depend  on  the   choices  of  the  probability  space
$(\Omega,\cf, \mu)$,  the kernel  $\kk$ (for the  kernel model)  and the
kernel~$k$  and recovery  rate $\gamma$  (for  the SIS  model), the  are
different  equivalent ways  to model  the  same situation.  We study  in
Section~\ref{sec:equivalent}  a   way  to  ensure  that,   even  if  the
parameters are different, we end up with the same Pareto frontiers. This
situation  is  similar  to  random  variables having  the  same  law  in
probability  theory, or  to equivalent  graphons in  graphon theory.  In
particular  it allows  us to  treat  the same  meta-population model  in
either      a      discrete      or      a      continuous      setting,
see~Figure~\ref{fig:discrete-and-continuous}  for  an  illustration  and
Example~\ref{ex:multipartite}.

\subsubsection{An illustrative example: the multipartite graphon}

Let us illustrate some of our results on an example, which will be discussed in details in
a forthcoming companion paper \cite{ddz-reg}.

\begin{example}[Multipartite graphon]\label{ex:multipartite}
  Graphs that can be colored with~$\ell$ colors, so that no two endpoints of an edge have
  the same color are known as~$\ell$-partite graphs. In a biological setting, this
  corresponds to a population of~$\ell$ groups, such that individuals in a group can not
  contaminate individuals of the same group. Let us generalize and assume there is an
  infinity of groups,~$\ell=\infty$ of respective size~$(2^{-n}, \, n\in \N^*)$ and that
  the next generation kernel~$\kk$ is equal to the constant~$\kappa>0$ between individuals
  of different groups and equal to~$0$ between individuals of the same group (so there is
  no intra-group contamination).
  Using the equivalence of models from Section~\ref{sec:equivalent},
  we can represent this model by using a continuous state
  space~$\Omega=[0, 1]$, endowed with~$\mu$ the Lebesgue measure on~$\Omega$, the
  group~$n$ being represented by the interval~$I_n=[1-2^{-n+1},1- 2^{-n})$ for~$n\in
  \N^*$. The kernel~$\kk$ is then given by~$\kk=\kappa (1 - \sum_{n\in \N^*}
  \mathds{1}_{I_n \times I_n})$; it is represented in Figure \ref{fig:kernel}.
  \medskip
  
  Consider the  loss~$\loss =  R_e$ and  the cost~$C=\costu$  giving the
  overall proportion of vaccinated individuals.  Based on the results of
  \cite{esser_spectrum_1980,stevanovi}, we prove  in \cite{ddz-reg} that
  the     vaccination    strategies~$\mathds{1}_{[0,     1-c]}$,    with
  cost~$C(\mathds{1}_{[0,     1-c]})=c\in    [0,1/2]$,     are    Pareto
  optimal. Remembering  that the natural  definition of the degree  in a
  continuous                graph                is                given
  by~$\mathrm{deg}(x)=\int_\Omega  \kk(x,  y)\, \mu(\mathrm{d}  y)$,  we
  note that the vaccination strategy~$\mathds{1}_{[0, 1-c]}$ corresponds
  to vaccinating individuals with feature~$x\in  (1-c, 1]$, that is, the
  individuals       with       the        highest       degree.       In
  Figure~\ref{fig:pareto_frontier},  the  corresponding Pareto  frontier
  (\textit{i.e.}, the outcome of the ``best'' vaccination strategies) is
  drawn as the solid red line;  the blue-colored zone corresponds to the
  feasible    region    that    is,     all    the    possible    values
  of~$(C(\eta),  R_e(\eta))$,  where~$\eta$  ranges  over~$\Delta$;  the
  dotted  line corresponds  to the  outcome of  the uniform  vaccination
  strategy        $\eta         \equiv        c$,         that        is
  $(C(\eta), R_e(\eta))=(c, (1-c) R_0)$  where $c$ ranges over~$[0, 1]$;
  and  the red  dashed  curve corresponds  to  the anti-Pareto  frontier
  (\textit{i.e.}, the outcome of  the ``worst'' vaccination strategies),
  which  for this  model correspond  to the  uniform vaccination  of the
  nodes with the  updated lower degree; see  \cite{ddz-reg}. Notice that
  the path~$(\mathds{1}_{[0,1-c]},  \, c\in  [0,1/2])$ is  an increasing
  continuous  (for  the topology  of  the  simple convergence  and  thus
  the~$L^1(\mu)$ topology) path of Pareto  optima which gives a complete
  parametrization of the  Pareto frontier. The latter  has been computed
  numerically  using  the  power  iteration method.  In  particular,  we
  obtained the following value: $R_0 \simeq 0.697 \kappa$.
\end{example}

\begin{figure}
 \begin{subfigure}[T]{.5\textwidth} \centering
 \includegraphics[page=1]{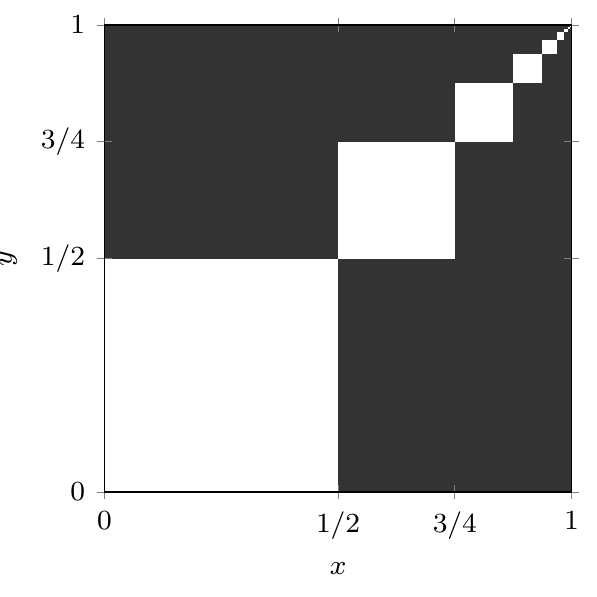} \caption{Grayplot
 of the kernel~$\kk$, with~$\Omega = [0,1]$ and~$\mu$ the Lebesgue measure
 ($\kk$ is equal to the constant $\kappa>0$
 on the black zone and to $0$ on the white zone).} \label{fig:kernel}
 \end{subfigure}%
 \begin{subfigure}[T]{.5\textwidth} \centering
 \includegraphics[page=2]{multipartite} \caption{The Pareto frontier in
 solid red line compared to the cost and loss of the uniform vaccinations in dotted line
 and the worst vaccination strategy in red dashed line.}
 \label{fig:pareto_frontier}
 \end{subfigure}%
 \caption{Example of optimization with~$\loss = R_e$.} \label{fig:optim_example}
\end{figure}

\subsection{On the companion papers}\label{sec:suite}

We detail some developments in forthcoming papers where only the uniform cost $C=\costu$
is considered. In \cite{ddz-Re}, motivated by the conjecture formulated by Hill and
Longini in finite dimension \cite[Conjecture~8.1]{hill-longini-2003}, we investigate the
convexity and concavity of the effective reproduction function $R_e$. We also prove that a
disconnecting strategy is better than the worst, \textit{i.e.}, is not anti-Pareto optimal.

In \cite{ddz-mono}, under monotonicity properties of the kernel, satisfied for example by
the configuration model, it is proven that vaccinating the individuals with the highest
(resp.\ lowest) number of contacts is Pareto (resp.\ anti-Pareto) optimal. In this case
the greedy algorithm, which performs infinitesimal locally optimal steps, is optimal as it
browses continuously the set of Pareto (resp.\ anti-Pareto) optimal strategies, providing
an increasing parametrization of the Pareto (resp.\ anti-Pareto) frontier. In this
setting, we provide some examples of SIS models where the set of Pareto optimal strategies
coincide for the losses $R_e$ and $\I$:
\begin{equation}\label{eq:PI=PR}
  \cp_\I= \cp_{R_e} \cap \{ \eta \in \Delta\, \colon\, R_e(\eta) \geq 1\}.
\end{equation}

In \cite{ddz-reg}, which includes a detailed study of the multipartite kernel of
Example~\ref{ex:multipartite}, we study the optimal vaccination when the individuals have
the same number of contacts. This provides examples where the uniform vaccination is Pareto
optimal, or anti-Pareto optimal, or not optimal for either problem. We also provide an example where
the set $\cp_{R_e}$ has a countable number of connected components (and is thus not connected). This
implies in particular that the greedy algorithm is not optimal in this case.

In \cite{ddz-2pop}, we give a comprehensive treatment of the two groups
model,~$\Omega=\{1, 2\}$, for~$\loss = R_e$, and some partial results for~$\loss=\I$.
Despite its apparent simplicity, the derivation of formulae for the Pareto optimal
strategies is non trivial, see also \cite{poghotanyan_constrained_2018}. In addition, this
model is rich enough to give examples of various interesting behaviours:
\begin{itemize}
  \item \emph{On the critical strategies~$\etauc$ and~$\etae$.} Depending on the
    parameters, the strategies~$\etauc$ and/or~$\etae$ may or may not be Pareto optimal,
    and the cost~$C(\etauc)$ may be larger than, smaller than or equal to~$C(\etae)$.
  \item \emph{Vaccinating people with highest contacts.} The intuitive idea of vaccinating
    the individuals with the highest number of contacts may or may not provide the optimal
    strategies, depending on the parameters.
  \item \emph{Dependence on the choice of the loss function.} For examples where~$R_0>1$,
    the optimal strategies for the losses~$\I$ and~$R_e$ may coincide, so that
    \eqref{eq:PI=PR} holds, or not at all, so that~$\cp_\I \cap \cp_{R_e}\cap \{ \eta\in
    \Delta\, \colon\, 1< R_e(\eta)<R_0\}=\emptyset$, depending on the parameters.
\end{itemize}

\subsection{Structure of the paper}

Section~\ref{sec:settings} is dedicated to the presentation of the vaccination model and
the various assumptions on the parameters. We also define properly the so-called loss
functions~$R_e$ and~$\I$. After recalling a few topological facts in
Section~\ref{sec:preliminaries}, we study the regularity properties of~$R_e$ and~$\I$
in~Section~\ref{sec:properties-loss}. We present the multi-objective optimization problem
in Section~\ref{sec:P-AntiP} under general condition on the loss function $\loss$ and cost
function~$C$ and prove the results on the Pareto frontier. This is completed in
Section~\ref{sec:diversCL} with miscellaneous properties of the Pareto frontier. In
Section~\ref{sec:equivalent}, we discuss the equivalent representation of models with
different parameters. Proofs of a few technical results are gathered in
Section~\ref{sec:technical_proofs}. 

\section{Setting and notation}\label{sec:settings}

\subsection{Spaces, operators, spectra}\label{sec:spaces}

All metric spaces~$(S,d)$ are endowed with their Borel~$\sigma$-field denoted by $\cb(S)$.
The set $\ck$ of compact subsets of~$\C$ endowed with the Hausdorff distance
$d_\mathrm{H}$ is a metric space, and the function~$\mathrm{rad}$ from~$\ck$ to $\R_+$
defined by~$\mathrm{rad}(K)=\max\{|\lambda|\, ,\, \lambda\in K\}$ is Lipschitz continuous
from~$(\ck,d_\mathrm{H})$ to~$\R$ endowed with its usual Euclidean distance.

Let~$(\Omega, \cf, \mu)$ be a probability space. We denote by~$\mathscr{L}^\infty$, the
Banach spaces of bounded real-valued measurable functions defined on~$\Omega$ equipped
with the~$\sup$-norm, $\mathscr{L}^\infty_+$ the subset of~$\mathscr{L}^\infty$ of
non-negative function, and $\Delta=\{f\in \mathscr{L}^\infty\,\colon\, 0\leq f\leq 1\}$
the subset of non-negative functions bounded by~$1$. For~$f$ and~$g$ real-valued functions
defined on~$\Omega$, we may write~$\langle f, g \rangle$ or $\int_\Omega f g \, \mathrm{d}
\mu$ for $\int_\Omega f(x) g(x) \,\mu( \mathrm{d} x)$ whenever the latter is meaningful.
For~$p \in [1, +\infty]$, we denote by $L^p=L^p( \mu)=L^p(\Omega, \mu)$ the space of
real-valued measurable functions~$g$ defined~$\Omega$ such that $\norm{g}_p=\left(\int
|g|^p \, \mathrm{d} \mu\right)^{1/p}$ (with the convention that~$\norm{g}_\infty$ is
the~$\mu$-essential supremum of $|g|$) is finite, where functions which agree~$\mu$-almost
surely are identified. We denote by~$L^p_+$ the subset of~$L^p$ of non-negative functions.

\medskip

Let~$(E, \norm{\cdot})$ be a Banach space. We denote by~$\norm{\cdot}_E$ the operator norm
on~$\cll(E)$ the Banach algebra of bounded operators. The spectrum~$\spec(T)$ of~$T\in
\cll(E)$ is the set of~$\lambda\in \C$ such that~$ T - \lambda \mathrm{Id}$ does not have
a bounded inverse operator, where~$\mathrm{Id}$ is the identity operator on~$E$. Recall
that~$\spec(T)$ is a compact subset of~$\C$, and that the spectral radius of~$T$ is given
by:
\begin{equation}\label{eq:def-rho}
  \rho(T)=\mathrm{rad}(\spec(T))=
  \lim_{n\rightarrow \infty } \norm{T^n}_E^{1/n}.
\end{equation}
The element $\lambda\in \spec(T)$ is an eigenvalue if there exists $x\in E$ such that
$Tx=\lambda x$ and $x\neq 0$.

If~$E$ is also a functional space, for~$g \in E$, we denote by~$M_g$ the multiplication
(possibly unbounded) operator defined by~$M_g(h)=gh$ for all~$h \in E$.

\subsection{Kernel operators}

We define a \emph{kernel} (resp.\ \emph{signed kernel}) on~$\Omega$ as a $\R_+$-valued
(resp.\ $\R$-valued) measurable function defined on~$(\Omega^2, \mathscr{F}^{\otimes 2})$.
For~$f,g$ two non-negative measurable functions defined on~$\Omega$ and~$\kk$ a kernel
on~$\Omega$, we denote by $f\kk g$ the kernel defined by:
\begin{equation}
  \label{eq:def-fkg}
  f\kk g:(x,y)\mapsto f(x)\, \kk(x,y) g(y).
\end{equation}
When~$\gamma$ is a positive measurable function defined on~$\Omega$, we write~$\kk/\gamma$
for~$\kk\gamma^{-1}$, and remark that it may differ from~$\gamma^{-1} \kk$.

For~$p \in (1, +\infty )$, we define the double norm of a signed kernel~$\kk$ by:
\begin{equation}\label{eq:Lp-integ-cond}
  \norm{\kk}_{p,q}=\left(\int_\Omega\left( \int_\Omega \abs{\kk(x,y)}^q\,
  \mu(\mathrm{d}y)\right)^{p/q} \mu(\mathrm{d}x) \right)^{1/p}
  \quad\text{with~$q$ given by}\quad \frac{1}{p}+\frac{1}{q}=1.
\end{equation}

\begin{hyp}[On the kernel model \gxx]\label{hyp:k}
  Let~$(\Omega, \cf, \mu)$ be a probability space. The kernel~$\kk$ on~$\Omega$ has a
  \emph{finite double-norm}, that is,~$\norm{\kk}_{p,q}<+\infty$ for some~$p\in (1,
  +\infty )$.
\end{hyp}

To a kernel $k$ such that $\norm{\kk}_{p,q}<+\infty$, we associate the positive integral
operator~$T_\kk$ on~$L^p$ defined by:
\begin{equation}\label{eq:def-Tkk}
  T_\kk (g) (x) = \int_\Omega \kk(x,y) g(y)\,\mu(\mathrm{d}y)
  \quad \text{for } g\in L^p \text{ and } x\in \Omega.
\end{equation}
According to~\cite[p. 293]{grobler}, operator $T_\kk$ is compact. It is well known and easy to
check that:
\begin{equation}\label{eq:double-norm-norm}
  \norm{ T_\kk }_{L^p}\leq \norm{\kk}_{p,q}.
\end{equation}
For~$\eta\in \Delta$, the kernel~$\kk \eta$ has also a finite double norm on~$L^p$ and the
operator~$M_\eta$ is bounded, so that the operator $T_{\kk \eta} = T_\kk M_\eta$ is
compact. We can define the \emph{effective spectrum} function~$\spec[\kk]$ from~$\Delta$
to~$\ck$ by:
\begin{equation}\label{eq:def-sigma_e}
  \spec[\kk](\eta)=\spec(T_{\kk\eta}),
\end{equation}
the \emph{effective reproduction number} function $R_e[\kk]=\mathrm{rad}\circ \spec[\kk]$
from~$\Delta$ to~$\R_+$ by:
\begin{equation}
  \label{eq:def-R_e}
  R_e[\kk](\eta)=\mathrm{rad}( \spec(T_{\kk \eta}))=\rho(T_{\kk\eta}),
\end{equation}
and the corresponding \emph{reproduction number}:
\begin{equation}\label{eq:def-R0}
  R_0[\kk]=R_e[\kk](\un)=\rho(T_\kk).
\end{equation}
When there is no ambiguity, we simply write $R_e$ for $R_e[\kk]$ and
$R_0$ for $R_0[\kk]$.
We say a vaccination strategy $\eta\in \Delta$ is \emph{critical} if
$R_e(\eta)=1$. 

\medskip

Following the framework of~\cite{delmas_infinite-dimensional_2020}, for $q\in (1, +\infty
)$, we also consider the following norm for the kernel $\kk$:
\[
  \norm{\kk}_{\infty,q} = \sup\limits_{x \in \Omega} \left(\int_\Omega \kk(x,y)^q\,
  \mu(\mathrm{d}y) \right)^{1/q}.
\]
Clearly, we have that $\norm{\kk}_{\infty , q}$ finite implies that $\norm{\kk}_{p,q}$ is
also finite, with $p$ such that $1/p+1/q=1$. When $\norm{\kk}_{\infty , q}<+\infty $, the
corresponding positive bounded linear integral operator~$\Tinf_\kk$
on~$\mathscr{L}^\infty$ is similarly defined by:
\begin{equation}
  \label{eq:def-Tk}
  \Tinf_{\kk} (g) (x) = \int_\Omega \kk(x,y) g(y)\,
  \mu(\mathrm{d}y)
  \quad \text{for } g\in \mathscr{L}^\infty \text{ and } x\in \Omega.
\end{equation}
Notice that the integral operators $\Tinf_\kk$ and $T_\kk$ corresponds respectively to the
operators $T_\kk$ and $\hat T_\kk$ in \cite{delmas_infinite-dimensional_2020}. According
to \cite[Lemma~3.7]{delmas_infinite-dimensional_2020}, the operator~$\Tinf_\kk^2$ on~$\cl$
is compact and~$\Tinf_\kk$ has the same spectral radius as~$ T_\kk$:
\begin{equation}\label{eq:rhoT=rhoT}
  \rho(\Tinf_\kk)=\rho(T_\kk).
\end{equation}

\subsection{Dynamics for the SIS model and equilibria}\label{sec:dyn}

In accordance with \cite{delmas_infinite-dimensional_2020}, we consider the following
assumption. Recall that $k/\gamma=k \gamma^{-1}$.

\begin{hyp}[On the SIS model \gxxx]\label{hyp:k-g}
  Let~$(\Omega, \cf, \mu)$ be a probability space. The recovery rate function~$\gamma$ is
  a function which belongs to~$\mathscr{L}^\infty_+$ and the transmission rate kernel~$k$
  on~$\Omega^2$ is such that $\norm{k/\gamma}_{\infty , q}<+\infty $ for some $q\in (1,
  +\infty )$.
\end{hyp}

Assumption~\ref{hyp:k-g} implies Assumption~\ref{hyp:k} for the kernel~$\kk=k/\gamma$. Under
Assumption~\ref{hyp:k-g}, we also consider the bounded operators~$\Tinf_{k / \gamma}$ on
$\cl$, as well as $T_{k / \gamma}$ on $L^p$, which are the so called \emph{next-generation
operator}. The SIS dynamics considered in~\cite{delmas_infinite-dimensional_2020} (under
Assumption \ref{hyp:k-g}) follows the vector field~$F$ defined on~$\mathscr{L}^\infty$ by:
\begin{equation}
  \label{eq:vec-field}
  F(g) = (1 - g) \Tinf_k (g) - \gamma g.
\end{equation}
More precisely, we consider~$u=(u_t, t\in \R)$, where~$u_t\in \Delta$ for all~$t\in\R_+$
such that:
\begin{equation}\label{eq:SIS2}
  \partial_t u_t = F(u_t)\quad\text{for } t\in \R_+,
\end{equation}
with initial condition~$u_0\in \Delta$. The value~$u_t(x)$ models the probability that an
individual of feature~$x$ is infected at time~$t$; it is proved
in~\cite{delmas_infinite-dimensional_2020} that such a solution~$u$ exists and is unique.

\medskip

An \emph{equilibrium} of~\eqref{eq:SIS2} is a function~$g \in \Delta$ such that~$F(g) =
0$. According to \cite{delmas_infinite-dimensional_2020}, there exists a maximal
equilibrium~$\mathfrak{g}$, \textit{i.e.}, an equilibrium such that all other
equilibria~$h\in \Delta$ are dominated by~$\mathfrak{g}$: $h \leq \mathfrak{g}$. The
\emph{reproduction number}~$R_0$ associated to the SIS model given by~\eqref{eq:SIS2} is
the spectral radius of the next-generation operator, so that using the definition of the
effective reproduction number~\eqref{eq:def-R_e},~\eqref{eq:def-R0} and
\eqref{eq:rhoT=rhoT}, this amounts to:
\begin{equation}\label{eq:def-R0-2}
  R_0= \rho (\Tinf_{k/\gamma})=R_0[k/\gamma]= R_e[k/\gamma](\un).
\end{equation}
If~$R_0\leq 1$ (sub-critical and critical case), then~$u_t$ converges pointwise to~$0$
when~$t\to\infty$. In particular, the maximal equilibrium~$\mathfrak{g}$ is equal to~$0$
everywhere. If~$R_0>1$ (super-critical case), then~$0$ is still an equilibrium but
different from the maximal equilibrium $\mathfrak{g}$, as~$\int_\Omega \mathfrak{g} \,
\mathrm{d}\mu > 0$.

\subsection{Vaccination strategies}\label{sec:vacc}

A \emph{vaccination strategy}~$\eta$ of a vaccine with perfect efficiency is an element
of~$\Delta$, where~$\eta(x)$ represents the proportion of \emph{\textbf{non-vaccinated}}
individuals with feature~$x$. Notice that~$\eta\, \mathrm{d} \mu$ corresponds in a sense
to the effective population.

Recall the definition of the kernel~$f\kk g$ from~\eqref{eq:def-fkg}. For~$\eta \in
\Delta$, the kernels~$k\eta/\gamma$ and~$k\eta$ have finite norm $\norm{\cdot}_{\infty ,
q}$ under Assumption \ref{hyp:k-g}, so we can consider the bounded positive
operators~$\Tinf_{k \eta / \gamma}$ and~$\Tinf_{k\eta}$ on~$\mathscr{L}^\infty$. According
to \cite[Section~5.3.]{delmas_infinite-dimensional_2020}, the SIS equation with
vaccination strategy~$\eta$ is given by~\eqref{eq:SIS2}, where~$F$ is replaced by~$F_\eta$
defined by:
\begin{equation}
  \label{eq:vec-field-vaccin}
  F_\eta(g) = (1-g) \Tinf_{k\eta}(g) - \gamma g.
\end{equation}
We denote by~$u^\eta=(u^\eta_t, t\geq 0)$ the corresponding solution with initial
condition~$u_0^\eta \in \Delta$. We recall that~$u_t^\eta(x)$ represents the probability
for an non-vaccinated individual of feature~$x$ to be infected at time $t$. Since the
effective reproduction number is the spectral radius of~$\Tinf_{k\eta/\gamma}$, we
recover~\eqref{eq:def-R_e} as~$
\rho(\Tinf_{k\eta/\gamma})=\rho(T_{k\eta/\gamma} )=
R_e[k/\gamma](\eta)$ with $\kk=k/\gamma$.
We denote by~$\mathfrak{g}_\eta$ the corresponding maximal equilibrium (so that
$\mathfrak{g}=\mathfrak{g}_1$). In particular, we have:
\begin{equation}
  \label{eq:F(g)=0}
  F_\eta(\mathfrak{g}_\eta)=0.
\end{equation}
We will denote by~$\I $ the \emph{fraction of infected individuals at equilibrium}. Since
the probability for an individual with feature~$x$ to be infected in the stationary regime
is~$\mathfrak{g}_\eta(x) \, \eta(x)$, this fraction is given by the following formula:
\begin{equation}\label{eq:asymptotic_number_endemic}
  \I (\eta)=\int_\Omega \mathfrak{g}_\eta
  \, \eta\, \mathrm{d}\mu=\int_\Omega \mathfrak{g}_\eta(x)
  \, \eta(x)\, \mu(\mathrm{d}x).
\end{equation}
We deduce from~\eqref{eq:vec-field-vaccin} and~\eqref{eq:F(g)=0}
that~$\mathfrak{g}_\eta\eta=0$ $\mu$-almost surely is equivalent to~$\mathfrak{g}_\eta=0$.
Applying the results of~\cite{delmas_infinite-dimensional_2020} to the kernel~$k \eta$, we
deduce that:
\begin{equation}
  \label{eq:gh>0}
  \I (\eta)>0 \,\Longleftrightarrow\, R_e[k/\gamma](\eta)>1.
\end{equation}

\medskip

We conclude this section with a result on the maximal equilibrium~$\mathfrak{g}$ which is
a direct consequence of Proposition~\ref{prop:caract-g} proved in
Section~\ref{sec:proof-I}. This result completes what is known
from~\cite{delmas_infinite-dimensional_2020}. Notice that, if~$R_0>1$, then Property
\ref{cor:g>0} implies that the strategy~$1-\mathfrak{g}$ is critical.

\begin{proposition}[On the maximal equilibrium]\label{cor:stab_g_star}
  Suppose Assumption \ref{hyp:k-g} holds and write $R_e$ for $R_e[k/\gamma]$.
  \begin{propenum}
  \item\label{cor:h=g} For any~$h\in\Delta$,~$h=\mathfrak{g}$ if and only if~$F(h) = 0$
    and~$R_e(1-h)\leq 1$.
  \item\label{cor:g>0} If~$\mathfrak{g}\neq 0$, then~$R_e (1-\mathfrak{g}) = 1$.
  \end{propenum}
\end{proposition}

\section{Preliminary topological results}
\label{sec:preliminaries}
\subsection{On the weak topology}\label{sec:weak}

We first recall briefly some properties we shall use frequently. We can see~$\Delta$ as a
subset of~$L^1 $, and consider the corresponding \emph{weak topology}:
a sequence~$(g_n, \, n \in \N)$ of elements of~$\Delta$ converges weakly to~$g$ if for
all~$h \in L^\infty $ we have:
\begin{equation}
  \label{eq:weak-cv}
  \lim\limits_{n \to \infty} \int_\Omega h g_n \, \mathrm{d}\mu= \int_\Omega h
  g\, \mathrm{d}\mu.
\end{equation}
Notice that \eqref{eq:weak-cv} can easily be extended to any function $h\in L^q$ for any
$q\in (1, +\infty )$; so that the weak-topology on $\Delta$, seen as a subset of $L^p$
with $1/p+1/q=1$, can be seen as the trace on $\Delta$ of the weak topology on $L^p$. The
main advantage of this topology is the following compactness result.

\begin{lemma}[Topological properties of~$\Delta$]\label{lem:D-compact}
 We have that:
 \begin{propenum}
 \item\label{item:D-compact} The set~$\Delta$ endowed with the weak topology is compact
 and sequentially compact.
 \item\label{item:f-cont} A function from $\Delta$ (endowed with the weak topology) to a
 metric space (endowed with its metric topology) is continuous if and only if it is
 sequentially continuous.
 \end{propenum}
\end{lemma}
\begin{proof}
 Let $p \in (1, +\infty )$, and consider the weak topology on $\Delta$ as the trace on
 $\Delta$ of the weak topology on $L^p$. We first prove \ref{item:D-compact}. Since~$L^p$
 is reflexive, by the Banach-Alaoglu theorem \cite[Theorem V.4.2]{Con90}, its unit ball
 is weakly compact. The set~$\Delta$ is closed and convex, therefore it is weakly closed;
 see~\cite[Corollary V.1.5]{Con90}. Thus, $\Delta$ is weakly compact as a weakly closed
 subset of the weakly compact unit ball. By the Eberlein–Šmulian theorem \cite[Theorem
 V.13.1]{Con90},~$\Delta$ is also weakly sequentially compact.

 \medskip

 We now prove \ref{item:f-cont}. A continuous function is sequentially continuous.
 Conversely, the inverse image of a closed set by a sequentially continuous function is
 sequentially closed. Besides, a sequentially closed subset of a sequentially compact set
 is sequentially compact. Using the Eberlein–Šmulian theorem, we deduce that
 the inverse images of closed sets are compact. In particular, they are closed which
 proves a sequentially continuous function is continuous.
\end{proof}

\subsection{Invariance and continuity of the spectrum for compact operators}

We recall a few facts on operators. Let~$(E, \norm{\cdot})$
be a Banach space. Let~$A\in \cll(E)$. We denote by $A^\top$ the
adjoint of $A$. A sequence~$(A_n,n \in \N)$ of elements of~$\cll(E)$
converges strongly to~$A \in \cll(E)$ if~$\lim_{n\rightarrow \infty } \norm{A_nx
-Ax}=0$ for all~$x\in E$. Following~\cite{anselone}, a set of operators~$\ca\subset
\cll(E)$ is \emph{collectively compact} if the set~$\{ A x \, \colon \, A \in \ca, \,
\norm{x}\leq 1 \}$ is relatively compact.

We collect some known results on the spectrum of to compact operators.
Recall that the spectrum of a compact operator is finite or countable
and has at most one accumulation point, which is $0$. Furthermore, $0$
belongs to the spectrum of compact operators in infinite dimension.
\begin{lemma}
 \label{lem:prop-spec-mult} Let $A, B$ be elements of $\cll(E)$.
\begin{propenum}
\item\label{item:A-B}
 If~$A$,~$B$ and~$A-B$ are positive
operators, then we have:
\begin{equation}\label{eq:r(A)r(B)}
\rho(A)\geq \rho(B).
\end{equation}
\item\label{item:spec-adjoint-mult}
 If $A$ is compact, then we have:
\begin{align}
 \label{eq:adjoint-mult}
 \spec(A)=\spec(A^\top)
 \\
 \label{eq:r(AB)-mult}
 \spec(AB)=\spec(BA)
\end{align}
and in particular:
\begin{equation}\label{eq:r(AB)}
 \rho(AB)=\rho(BA).
\end{equation}

\item \label{item:density-mult}
 Let~$(E', \norm{\cdot}')$ be a Banach space such that~$E'$ is continuously and densely
embedded in~$E$. Assume that $A(E')\subset E'$, and denote by~$A'$ the
restriction of $A$ to~$E'$ seen as an operator on~$E'$. If $A$ and $A'$
are compact, then we have:
\begin{equation}\label{eq:sT=sT'}
 \spec(A)=\spec(A').
\end{equation}

\item \label{item:collectK-cv}
 Let~$(A_n, n\in \N)$ be a collectively compact sequence which
 converges strongly to~$A$. Then, we have $\lim_{n\rightarrow \infty }
 \spec(A_n)=\spec(A)$ in~$(\ck, d_\mathrm{H})$, and $\lim_{n\rightarrow }
 \rho(T_n)=\rho(T)$.
 \end{propenum}
\end{lemma}

\begin{proof}
 Property \ref{item:A-B} can be found in \cite[Theorem~4.2]{marek}.
 Equation \eqref{eq:adjoint-mult} from Property
 \ref{item:spec-adjoint-mult} can be deduced from the
 \cite[Theorem page 20]{konig}. Using the \cite[Proposition page 25]{konig}, we get
 that
 $ \spec(AB) \cap \C^*=\spec(BA) \cap \C^*$, and thus
 \eqref{eq:r(AB)}. As~$A$ is compact we get that~$AB$ and~$BA$ are
 compact, thus~$0$ belongs to their spectrum in infinite dimension.
 Whereas in finite dimension,
 as~$\mathrm{det}(AB)=\mathrm{det}(A)\mathrm{det}(B) =\mathrm{det}(BA)$
 (where~$A$ and~$B$ denote also the matrix of the corresponding
 operator in a given base), we get that~$0$ belongs to the spectrum
 of~$AB$ if and only if it belongs to the spectrum of~$BA$. This gives
 \eqref{eq:r(AB)-mult}.
\medskip

Property \ref{item:density-mult} follows from \cite[Corollary~1
and Section~6]{halber}.
We eventually check Property \ref{item:collectK-cv}. We deduce from
\cite[Theorems~4.8 and 4.16]{anselone} (see also (d) and (e)
in~\cite[Section~3]{SpectralProperAnselo1974})
that~$\lim_{n\rightarrow \infty } \spec(T_n)=\spec(T)$. Then use that
the function~$\mathrm{rad}$ is continuous to deduce the convergence of
the spectral radius from the convergence of the spectra (see also (f) in
~\cite[Section~3]{SpectralProperAnselo1974}).
\end{proof}




\section{First properties of the functions \texorpdfstring{$R_e$}{Re} and
\texorpdfstring{$\I$}{I}}\label{sec:properties-loss}

\subsection{The effective reproduction number~\texorpdfstring{$R_e$}{Re}}

We consider the kernel model $[(\Omega, \cf\!, \mu), \kk]$ under Assumption~\ref{hyp:k},
so that~$\kk$ is a kernel on~$\Omega$ with finite double norm. Recall the effective
reproduction number function~$R_e[\kk]$ defined on $\Delta$
by~\eqref{eq:def-R_e}:~$R_e[\kk](\eta)=\rho(T_\kk M_\eta)$ and the reproduction
number~$R_0[\kk]=\rho(T_\kk)$. We
simply write $R_e$ and~$R_0$ for~$R_e[\kk]$ and~$R_0[\kk]$ respectively when no confusion
on the kernel can arise.

\begin{proposition}[Basic properties of $R_e$]\label{prop:R_e}
  Suppose Assumption \ref{hyp:k} holds. Let $\eta, \eta_1, \eta_2 \in \Delta$. The
  function~$R_e=R_e[\kk]$ satisfies the following properties:
  \begin{propenum}
  \item\label{prop:a.s.-Re}%
    $R_e(\eta_1)=R_e(\eta_2)$ if~$\eta_1=\eta_2$ $\mu$-almost surely.
  \item\label{prop:min_Re}%
    $R_e(0) = 0$ and~$R_e(\un) = R_0$.
  \item\label{prop:increase}%
    $R_e(\eta_1) \leq R_e(\eta_2)$ if $\eta_1 \leq \eta_2$ $\mu$-almost surely.
  \item\label{prop:normal}%
    $R_e(\lambda \eta) = \lambda R_e(\eta)$ for all~$\lambda \in [0,1]$.
  \end{propenum}
\end{proposition}

\begin{proof}
  If~$\eta_1=\eta_2$~$\mu$-almost surely, then we have that~$T_{\kk\eta_1} =
  T_{\kk\eta_2}$, and thus~$R_e(\eta_1)=R_e(\eta_2)$. This gives Point~\ref{prop:a.s.-Re}.
  Point~\ref{prop:min_Re} is a direct consequence of the definition of~$R_e$. Since for
  any fixed~$\lambda\in\mathbb{C}$ and any operator~$A$, the spectrum of~$\lambda A$ is
  equal to~$\{ \lambda s, \, s\in \spec(A)\}$, Point~\ref{prop:normal} is clear. Finally,
  note that if~$\eta_1 \leq \eta_2$ $\mu$-almost everywhere, then the operator~$T_{\kk
  \eta_2} - T_{\kk \eta_1 }$ is positive. According to~\eqref{eq:r(A)r(B)}, we get that $\rho(T_{\kk \eta_1 })
  \leq \rho(T_{k \eta_2 })$. This concludes the proof of Point~\ref{prop:increase}.
\end{proof}

We generalize a continuity property on the spectral radius originally stated
in~\cite{delmas_infinite-dimensional_2020} by weakening the topology.

\begin{theorem}[Continuity of~$\grR$ and~$\grS$]\label{th:continuity-R}
  Suppose Assumption \ref{hyp:k} holds. Then, the functions~$\spec[\kk]$ and~$R_e[\kk]$
  are continuous functions from $\Delta$ (endowed with the weak-topology) respectively
  to~$\ck$ (endowed with the Hausdorff distance) and to~$\R_+$ (endowed with the usual
  Euclidean distance).
\end{theorem}

Let us remark the proof holds even if~$\kk$ takes negative values.

\begin{proof}
  Let~$B$ denote the unit ball in $L^p$, with $p\in (1, +\infty )$ from Assumption
  \ref{hyp:k}. Since the operator~$T_\kk$ is compact, the set~$T_{\kk}(B)$ is relatively
  compact. For all~$\eta\in \Delta$, set~$\eta B=\{\eta g\, \colon\, g\in B\}$. As~$\eta
  B\subset B$, we deduce that~$T_{\kk \eta}(B)= T_\kk (\eta B) \subset T_\kk (B)$. This
  implies that the family~$(T_{\kk \eta}, \, \eta \in \Delta)$ is collectively compact.

  Let~$(\eta_n , \, n \in \N)$ be a sequence in~$\Delta$ converging weakly to some~$\eta
  \in \Delta$. Let~$g \in L^p$. The weak convergence of~$\eta_n$ to~$\eta$ implies
  that~$(T_{\kk \eta_n}(g), \, n \in \N)$ converges~$\mu$-almost surely to~$T_{\kk \eta}
  (g)$. Consider the function:
  \[
    K(x)=\left(\int _\Omega \kk(x,y)^q \, \mu(\mathrm{d} y) \right)^{1/q},
  \]
  which belongs to~$L^p$, thanks to~\eqref{eq:Lp-integ-cond}. Since for
  all~$x$,
  \[
    \abs{T_{\kk \eta_n}(g)(x)}
    \leq T_{\kk}( \abs {\eta_n g})(x)
    \leq K(x) \, \| \eta_n g\|_p \leq K(x)\, \norm{g}_p ,
  \]
  we deduce, by dominated convergence, that the convergence holds also in~$L^p$:
  \begin{equation}\label{eq:cv-Tkn}
    \lim_{n\rightarrow \infty } \norm{ T_{\kk \eta_n}(g) - T_{\kk \eta}(g)}_p=0,
  \end{equation}
  so that $T_{\kk\eta_n}$ converges strongly to $T_{\kk\eta}$. Using
  Lemma~\ref{lem:prop-spec-mult}~\ref{item:collectK-cv} (with~$T_n=T_{\kk \eta_n}$
  and~$T=T_{\kk \eta}$) on the continuity of the spectrum, we get
  that~$\lim_{n\rightarrow\infty } \spec[\kk](\eta_n)=\spec[\kk](\eta)$. The
  function~$\spec[\kk]$ is thus sequentially continuous, and, thanks to Lemma
  \ref{lem:D-compact}, it is continuous from~$\Delta$ endowed with the weak topology to
  the metric space~$\ck$ endowed with the Hausdorff distance. The continuity
  of~$R_e[\kk]$ then follows from its definition~\eqref{eq:def-rho} as the composition of
  the continuous functions~$\mathrm{rad}$ and~$\spec[\kk]$.
\end{proof}

We give a stability property of the spectrum and
spectral radius with respect to the kernel~$\kk$.

\begin{proposition}[Stability of~$\grR$ and~$\grS$]\label{prop:Re-stab}
  Let~$p\in (1, +\infty )$. Let~$(\kk_n, n\in \N)$ and~$\kk$ be kernels on~$\Omega$ with
  finite double norms on~$L^p$. If~$\lim_{n\rightarrow\infty} \norm{\kk_n - \kk}_{p,q}=0$,
  then we have:
  \begin{equation}
    \label{eq:Re-stab}
    \lim_{n\rightarrow\infty }\, \sup_{\eta\in \Delta} \Big|R_e[\kk_n](\eta) -
    R_e[\kk](\eta)\Big|=0
    \quad\text{and}\quad
    \lim_{n\rightarrow\infty }\, \sup_{\eta\in \Delta} d_\mathrm{H}\Big (
    \spec[\kk_n](\eta), \spec [\kk](\eta)\Big)=0.
  \end{equation}
\end{proposition}

\begin{proof}
  We first prove that~$\lim_{n\rightarrow\infty } \spec[\kk_n](\eta_n)= \spec[\kk](\eta)$,
  where the sequence~$(\eta_n, n\in \N)$ is any sequence in~$\Delta$ which converges
  weakly to~$\eta\in \Delta$.

  The operators~$\ca=\{T_\kk\} \cup \{T_{\kk_n}\, \colon\, n\in \N\}$ are compact, and we
  deduce from~\eqref{eq:double-norm-norm} that:
  \[
    \lim_{n\rightarrow\infty}\norm{T_{\kk_n} -T_\kk}_{L^p}=0.
  \]
  The   family~$\ca$   is   then   easily  seen   to   be   collectively
  compact.  (Indeed,  let  $(y_n=T_{\kk_{i_n}}(x_n),   n\in  \N)$  be  a
  sequence with  $\norm{x_n}\leq 1$, $i_n\in \N\cup\{\infty  \}$ and the
  convention  $\kk_{i_n}=\kk$   if  $i_n=\infty  $.   Up   to  taking  a
  sub-sequence, we can assume that  either the sequence $(i_n, n\in \N)$
  is  constant and  $(T_{\kk_{i_0}}(x_n),  n\in \N)$  is convergent  (as
  $T_{\kk_{i_0}}$ is compact)  or that the sequence $(i_n,  n\in \N)$ is
  increasing and  the sequence  $(T_\kk(x_n), n\in  \N)$ is convergent
  (as $T_{\kk}$ is compact)   towards a
  limit, say $y$. In the  former case, clearly
  the sequence $(y_n, n\in \N)$ converges.  In the latter case, we have:
  $\norm{T_{\kk_{i_n}}   (x_n)  -   y}_p   \leq  \norm{T_{\kk_{i_n}}   -
    T_\kk}_{L^p} +  \norm{T_{\kk} (x_n)  - y}_p$, which  readily implies
  that the sequence $(y_n, n\in  \N)$ converges towards $y$. This proves
  that  the family~$\ca$  is collectively  compact.)  This  implies, see
  \cite[Proposition~4.1(2)]{anselone}    for     details,    that    the
  family~$\ca'=\{T'M_\eta\,  \colon,  T'\in  \ca \text{  and  }  \eta\in
  \Delta\}$   is    collectively   compact.    We   deduce    that   the
  sequence~$(T_n=T_{\kk_n  \eta_n}  =T_{\kk_n}M_{\eta_n}, n\in  \N)$  of
  elements    of~$\ca'$    is     collectively    compact    and    that
  $T=T_{\kk\eta}=T_\kk M_\eta$ is compact.

  Let~$g\in L^p$. We have:
  \[
    \norm{T_n(g) - T(g)}_p
    \leq \norm{T_{\kk_n}-T_\kk}_{L^p} \, \norm{ g}_p + \norm{T_{\kk\eta_n}(g) -
    T_{\kk\eta}(g)}_p.
  \]
  Using~$\lim_{n\rightarrow\infty }\norm{T_{\kk_n} -T_\kk}_{L^p}=0$ and~\eqref{eq:cv-Tkn},
  we get that~$\lim_{n\rightarrow\infty } \norm{T_n(g) - T(g)}_p$, thus~$(T_n, n\in \N)$
  converges strongly to $T$. With Lemma~\ref{lem:prop-spec-mult}~\ref{item:collectK-cv},
  we get that~$\lim_{n\rightarrow \infty } \spec(T_n)=\spec(T)$, that is
  $\lim_{n\rightarrow\infty } \spec[\kk_n](\eta_n)= \spec[\kk](\eta)$. \medskip

  Then, as the function~$\eta \mapsto d_\mathrm{H}\Big ( \spec[\kk_n](\eta), \spec
  [\kk](\eta)\Big)$ is continuous on the compact set~$\Delta$, thanks to
  Theorem~\ref{th:continuity-R}, it reaches its maximum say at~$\eta_n\in \Delta$
  for~$n\in \N$. As~$\Delta$ is compact, consider a sub-sequence which converges weakly to
  a limit say~$\eta$. Since
  \begin{multline*}
    \sup_{\eta\in \Delta} d_\mathrm{H}\Big (
    \spec[\kk_n](\eta), \spec [\kk](\eta)\Big)\\
    \begin{aligned}
 & = d_\mathrm{H}\Big (
 \spec[\kk_n](\eta_n), \spec [\kk](\eta_n)\Big) \\
 & \leq d_\mathrm{H}\Big (
 \spec[\kk_n](\eta_n), \spec [\kk](\eta)\Big)
 + d_\mathrm{H}\Big (
 \spec[\kk](\eta_n), \spec [\kk](\eta)\Big),
    \end{aligned}
  \end{multline*}
  using the continuity of~$\spec[\kk]$, we deduce that along this sub-sequence the right
  hand side converges to 0. Since this result holds for any converging sub-sequence, we
  get the second part of~\eqref{eq:Re-stab}. The first part then follows from the
  definition~\eqref{eq:def-rho} of $R_e$ as a composition, and the Lipschitz continuity of
  the function~$\mathrm{rad}$.
\end{proof}

\subsection{The asymptotic proportion of infected individuals \texorpdfstring{$\I $}{}}

We consider the SIS model $[(\Omega, \cf, \mu), k, \gamma]$ under
Assumption~\ref{hyp:k-g}. Recall from~\eqref{eq:asymptotic_number_endemic} that the
asymptotic proportion of infected individuals~$\I $ is given on $\Delta$ by
$\I (\eta)=\int_\Omega \mathfrak{g}_\eta \, \eta\,
\mathrm{d}\mu$, where $\mathfrak{g}_\eta$ is the maximal solution in
$\Delta$ of the equation $F_\eta(h) = 0$.
We first give a preliminary result.
\begin{lemma}
  \label{lem:Fh>0} Let $\eta, g \in \Delta$. If $F_\eta(g) \geq 0$, then we have $g\leq
  \mathfrak{g}_\eta$.
\end{lemma}
\begin{proof}
  According to \cite[Proposition~2.10]{delmas_infinite-dimensional_2020}, the solution
  $u_t$ of the SIS model with vaccination $\partial_t u_t = F_\eta(u_t)$ and initial
  condition $u_0 = g$ is non-decreasing since $F_\eta(g) \geq 0$. According to
  \cite[Proposition~2.13]{delmas_infinite-dimensional_2020}, the pointwise limit of $u_t$
  is an equilibrium. As this limit is dominated by the maximal equilibrium
  $\mathfrak{g}_\eta$ and since $u_t$ is non-decreasing, this proves that
  $g\leq\mathfrak{g}_\eta$.
\end{proof}

We may now state the main properties of the function $\I$.

\begin{proposition}[Basic properties of $\I $]\label{prop:I}
  Suppose that Assumption~\ref{hyp:k-g} holds. Let $\eta, \eta_1, \eta_2 \in \Delta$. The
  function~$\I $ has the following properties:
  \begin{propenum}
  \item\label{prop:a.s.-I} $\I (\eta_1)=\I (\eta_2)$ if $\eta_1=\eta_2$ $\mu$-almost
    surely.
  \item\label{prop:min-I} $\I (\eta)=0$ if and only if $R_e[k/\gamma](\eta) \leq 1$.
  \item\label{prop:mono-I} $\I (\eta_1) \leq \I (\eta_2)$ if $\eta_1\leq \eta_2$
    $\mu$-almost surely.
  \item\label{prop:hom-I} $\I (\lambda \eta) \leq \lambda \I (\eta)$ for all $\lambda \in
    [0,1]$.
  \end{propenum}
\end{proposition}

\begin{proof}
  If $\eta_1=\eta_2$ $\mu$-almost surely, then the operators $\Tinf_{k\eta_1}$ and
  $\Tinf_{k\eta_2}$ are equal. Thus, the equilibria $\mathfrak{g}_{\eta_1}$ and
  $\mathfrak{g}_{\eta_2}$ are also equal which in turns implies that $\I (\eta_1)=\I
  (\eta_2)$. Point~\ref{prop:min-I} is already stated in Equation~\eqref{eq:gh>0}.

  \medskip

  To prove the monotonicity (Point~\ref{prop:mono-I}), consider $\eta_1\leq \eta_2$. Since
  $\Tinf_{k\eta_1} \leq \Tinf_{k\eta_2}$, we get $F_{\eta_1}(g)\leq F_{\eta_2}(g)$ for all
  $g\in\Delta$. In particular, taking $g=\mathfrak{g}_{\eta_1}$ and
  using~\eqref{eq:F(g)=0}, we get $F_{\eta_2}(\mathfrak{g}_{\eta_1})\geq 0$. By
  Lemma~\ref{lem:Fh>0} this implies $\mathfrak{g}_{\eta_1}\leq \mathfrak{g}_{\eta_2}$. To
  sum up, we get:
  \begin{equation}
    \label{eq:mon-gh}
    \eta_1 \leq \eta_2 \quad \Longrightarrow \quad \mathfrak{g}_{\eta_1}\leq
    \mathfrak{g}_{\eta_2}.
  \end{equation}
  This readily implies that $ \I (\eta_1) = \int_\Omega
  \mathfrak{g}_{\eta_1}\, \eta_1 \, \mathrm{d}\mu \leq \int_\Omega
  \mathfrak{g}_{\eta_2}\, \eta_2 \, \mathrm{d}\mu = \I (\eta_2)$. We conclude using
  Point~\ref{prop:a.s.-I}.

  \medskip

  We now consider Point~\ref{prop:hom-I}. Since $\lambda\in [0, 1]$, we deduce
  from~\eqref{eq:mon-gh} that $\mathfrak{g}_{\lambda \eta} \leq
  \mathfrak{g}_{\eta}$. This implies that $ \I (\lambda \eta) = \int_\Omega
  \mathfrak{g}_{\lambda \eta}\, \lambda \eta \, \mathrm{d}\mu \leq \lambda \int_\Omega
  \mathfrak{g}_{\eta}\, \eta \, \mathrm{d}\mu = \lambda \I (\eta)$.
\end{proof}

The proof of the following continuity results are both postponed to
Section~\ref{sec:proof-I}.

\begin{theorem}[Continuity of $\I$]\label{th:continuity-I}
  Suppose that Assumption~\ref{hyp:k-g} holds. The function $\I $ defined on $\Delta$ is
  continuous with respect to the weak topology.
\end{theorem}

We write $\I[k, \gamma]$ for $\I$ to stress the dependence on the
parameters $k, \gamma$ of the SIS model.

\begin{proposition}[Stability of $\I$]\label{prop:I-stab}
  Let $((k_n, \gamma_n), n\in \N)$ and $(k,\gamma)$ be a sequence of kernels and functions
  satisfying Assumption~\ref{hyp:k-g}. Assume furthermore that there exists $p'\in (1,
  +\infty )$ such that $\kk=\gamma^{-1} k$ and $(\kk_n=\gamma^{-1}_n k_n, n\in \N)$ have
  finite double norm in $L^{p'}$ and that $\lim_{n\rightarrow\infty } \norm{\kk_n
  -\kk}_{p',q'}=0$. Then we have:
  \begin{equation}\label{eq:I-stab}
    \lim_{n\rightarrow\infty }\, \sup_{\eta\in \Delta} \Big|\I[k_n, \gamma_n](\eta) -
    \I[k,\gamma](\eta)\Big| = 0.
  \end{equation}
\end{proposition}

\section{Pareto and anti-Pareto frontiers}\label{sec:P-AntiP}

\subsection{The setting}\label{sec:cost-loss}

To any vaccination strategy $\eta \in \Delta$, we associate a cost and a loss.
\begin{itemize}
  \item\textbf{\emph{The cost function}}. The \emph{cost} $C(\eta)$ measures all the costs
    of the vaccination strategy (production and diffusion). The cost is expected to be a
    decreasing function of $\eta$, since $\eta$ encodes the non-vaccinated population.
    Since doing nothing costs nothing, we also expect $C(\un)=0$, see
    Assumptions \ref{hyp:loss+cost} below. We shall also  consider
    natural hypothesis on $C$, see Assumptions~\ref{hyp:cost}
    and~\ref{hyp:cost**}. 
    A simple  cost model is the
    affine cost given by:
    \begin{equation}\label{eq:def-C0} \costa(\eta)=\int_\Omega (1-\eta(x))\, \costad(x) \,
      \mu(\mathrm{d} x),
    \end{equation}
    where $\costad(x)$ is the cost of vaccination of population of feature $x$, with
    $\costad\in L^1$ positive. The particular case $\costad=1$ is the uniform cost~$C =
    \costu$:
    \begin{equation}\label{eq:def-C} \costu (\eta) = \int_\Omega (1 - \eta) \,
      \mathrm{d}\mu.
    \end{equation}
    The real cost of the vaccination may be a more complicated function
    $\psi(\costa(\eta))$
    of the affine cost, for example if the marginal cost of producing a vaccine depends on
    the quantity already produced. However, as long as $\psi$ is strictly increasing, this
    will not affect the optimal strategies.

  \item\textbf{\emph{The loss function}}. The \emph{loss} $\loss(\eta)$ measures the
    (non)-efficiency of the vaccination strategy~$\eta$. Different choices are possible
    here. We prove in this section general results that only depend on a few natural
    hypothesis  for $\loss$; see Assumptions \ref{hyp:loss+cost}, \ref{hyp:loss} and
    \ref{hyp:loss**}. These hypothesis are in particular satisfied if the loss is the
    effective reproduction number $R_e$ (kernel and SIS models), or the asymptotic
    proportion of infected individuals $\I$ (SIS model); more precisely see Lemmas
    \ref{lem:c-dec+L-hom}, \ref{lem:L**} and \ref{lem:L**=I}.
\end{itemize}

We shall consider  cost and loss
functions with some regularities.

\begin{definition}\label{def:cont, monot}
  We say that a real-valued function $H$ defined on $\Delta$ endowed with the weak
  topology is:
  \begin{itemize}
    \item \textbf{\emph{Continuous}}: if $H$ is continuous with respect to the weak
      topology on $\Delta$.
    \item \textbf{\emph{Non-decreasing}}: if for any $\eta_1, \eta_2\in \Delta$ such that
      $\eta_1\leq \eta_2$, we have $H(\eta_1)\leq H(\eta_2)$.
    \item \textbf{\emph{Decreasing}}: if for any $\eta_1, \eta_2\in \Delta$ such that
      $\eta_1\leq \eta_2$ and $\int_\Omega \eta_1\, \mathrm{d} \mu <\int_\Omega \eta_2\,
      \mathrm{d} \mu$, we have $H(\eta_1)>
      H(\eta_2)$.
    \item \textbf{\emph{Sub-homogeneous}}: if $H(\lambda \eta)\leq \lambda H(\eta)$ for
      all $\eta\in \Delta$ and $\lambda\in [0, 1]$.
  \end{itemize}
\end{definition}

The definition of non-increasing function and increasing function are similar.

\begin{hyp}[On the cost function and loss function]\label{hyp:loss+cost}
  The loss function $\loss: \Delta \rightarrow \R$ is non-decreasing and continuous with
  $\loss(0)=0$. The cost function $C: \Delta\rightarrow
  \R$ is non-increasing and continuous with $C(\un) = 0$. We also have:
  \[
  \maxloss:=\max_\Delta \, \loss>0\quad\text{and}\quad \maxcost:=\max_\Delta \, C>0.
  \]
\end{hyp}
Assumption \ref{hyp:loss+cost} will always hold. In particular, the loss
and the cost functions are non-negative and non-constant.
\medskip

We will consider the multi-objective minimization and maximization problems:
\optProbminmax{eq:multi-objective-minmax}{(C(\eta),\loss(\eta))}{\eta \in \Delta}%
Before going further, let us remark that for the reproduction number optimization in the
vaccination context, one can without loss of generality consider the uniform cost instead of
the affine cost.

\begin{remark}\label{rem:costa-costu}
  Consider the kernel model $\param=[(\Omega, \cf, \mu), \kk]$ with the affine cost
  function $\costa$ and the loss $R_e$. Furthermore, if we assume that $\costad$ is
  bounded and bounded away from 0 (that is $\costad$ and $1/\costad$ belongs to $\cl_+$),
  and without loss of generality, that $\int \costad \, \mathrm{d} \mu=1$, then we can
  consider the weighted kernel model $\param_0=[(\Omega, \cf, \mu_0), \kk_0]$ with measure
  $\mu_0(\mathrm{d} x)=\costad(x) \, \mu(\mathrm{d} x)$ and kernel $\kk_0= \kk/\costad$.
  (Notice that if Assumption~\ref{hyp:k-g} holds for the model $\param$, then it also
  holds for the model $\param_0$.) Consider the loss $\loss=R_e$. Then for a strategy
  $\eta\in \Delta$, we get that $(\costa(\eta), \loss (\eta))$ for the model $\param$ is
  equal to $(\costu(\eta), \loss (\eta))$ for the model $\param_0$. Therefore, for the loss
  function $\loss=R_e$, instead of the affine cost $\costa$, one can consider without any
  real loss of generality the uniform cost. (This holds also for the SIS model.) However,
  this is no longer the case for the loss function $\loss=\I$ in the SIS model.
\end{remark}

Multi-objective problems are in a sense ill-defined because in most cases, it is
impossible to find a single solution that would be optimal to all objectives
simultaneously. Hence, we recall the concept of Pareto optimality. Since the minimization
problem is crucial for vaccination, we shall define Pareto optimality for the bi-objective
minimization problem. A strategy $\eta_\star \in \Delta$ is said to be \emph{Pareto
optimal} for the minimization problem in \eqref{eq:multi-objective-minmax} if any
improvement of one objective leads to a deterioration of the other, for $\eta\in \Delta$:
\begin{equation}\label{eq:defParetoOptimal}
  C(\eta)<C(\eta_\star) \implies \loss(\eta) >\loss(\eta_\star)
  \quad\text{and}\quad
  \loss(\eta)<\loss(\eta_\star) \implies C(\eta) > C(\eta_\star).
\end{equation}
Similarly, a strategy $\eta^\star\in \Delta$ is \emph{anti-Pareto optimal} if it is Pareto
optimal for the bi-objective maximization problem in \eqref{eq:multi-objective-minmax}.
Intuitively, the ``best'' vaccination strategies are the Pareto optima and the ``worst''
vaccination strategies are the anti-Pareto optima. \medskip

We define the \emph{feasible region} as all possible outcomes:
\[
  \FF =\{ (C(\eta),\loss(\eta)), \,\eta\in \Delta\}.
\]
Then, we first consider the minimization problem for the ``best'' strategies. The set of
Pareto optimal strategies will be denoted by~$\cp_\loss$, and the Pareto frontier is
defined as the set of Pareto optimal outcomes:
\[
  \mathcal{F}_\loss = \{ (C(\eta),\loss(\eta))\,\colon\, \eta \text{ Pareto optimal}\}.
\]
We consider the minimization problems related to the ``best'' vaccination strategies, with
$\ell \in [0, \maxloss]$ and $c\in [0, \maxcost]$:
\optProbBis{eq:Prob1}{eq:optProb1}{eq:constraint1}{\loss(\eta)}{\eta\in\Delta, \,
C(\eta)\leq c,} as well as
\optProbBis{eq:Prob2}{eq:optProb2}{eq:constraint2}{C(\eta)}{\eta\in\Delta, \,
\loss(\eta)\leq \ell.} We denote the values of Problems~\eqref{eq:Prob1} and
\eqref{eq:Prob2} by:
\begin{align*}
  \lossinf(c)
&=\inf\{ \loss(\eta)\, \colon\, \eta\in \Delta \text{ and }
C(\eta) \leq c\} \quad\text{for $c\in [0, \maxcost]$},\\
\Cinf(\ell)
&=\inf\{ C(\eta)\, \colon\, \eta\in \Delta \text{ and } \loss(\eta) \leq
\ell\}\quad\text{for $\ell\in [0, \maxloss]$}.
\end{align*}

We now consider the maximization problem related to the ``worst'' vaccination strategies,
with $\ell \in [0, \maxloss]$ and $c\in [0, \maxcost]$:
\optProbTer{eq:Prob1**}{eq:optProb1**}{eq:constraint1**}{\loss(\eta)}{\eta\in\Delta, \,
C(\eta)\geq c,} as well as
\optProbTer{eq:Prob2**}{eq:optProb2**}{eq:constraint2**}{C(\eta)}{\eta\in\Delta, \,
\loss(\eta)\geq \ell.} We denote the values of Problems~\eqref{eq:Prob1**} and
\eqref{eq:Prob2**} by:
\begin{align*}
  \lossup(c)
&=\sup\{ \loss(\eta)\, \colon\, \eta\in \Delta \text{ and }
C(\eta) \geq c\} \quad\text{for $c\in [0, \maxcost]$},\\
\Csup(\ell)
&=\sup\{ C(\eta)\, \colon\, \eta\in \Delta \text{ and } \loss(\eta) \geq
\ell\}\quad\text{for $\ell\in [0, \maxloss]$}.
\end{align*}
We denote by~$\cp_\loss^{\mathrm{Anti}}$ the set of anti-Pareto optimal strategies, and by
$\mathcal{F}_\loss^{\mathrm{Anti}}$ its frontier:
\[
  \mathcal{F}_\loss^{\mathrm{Anti}} = \{ (C(\eta), \loss(\eta))\,\colon\, \eta
  \text{ anti-Pareto optimal}\}.
\]

If necessary, we may write $\CinfL$ and $\CsupL$ to stress the dependence of the function
$\Cinf$ and $\Csup$ in the loss function $\loss$.

Under  Assumption \ref{hyp:loss+cost}, as the loss and the cost functions are
 continuous on the compact set
$\Delta$, the infima in the definitions of the value functions $\Cinf$ and $\lossinf$ are
minima; and the suprema in the definition of the value functions $\Csup$ and $\lossup$ are
maxima. Since $\Delta$ in endowed with the weak topology, we will consider the set of
Pareto and anti-Pareto optimal vaccination modulo $\mu$-almost sure equality.

\medskip

See Figure \ref{fig:typical_frontier} for a typical representation of the possible aspects
of the feasible region~$\FF$ (in light blue), the value functions and the Pareto and
anti-Pareto frontiers under the general Assumption~\ref{hyp:loss+cost}, and the connected
Pareto and anti-Pareto frontiers under further regularity on the cost and loss functions
(see Assumption~\ref{hyp:cost}-\ref{hyp:loss**} below) in Figure~\ref{fig:reg_frontier}.
In Figure~\ref{fig:pareto_frontier}, we have plotted in solid red line the Pareto frontier
and in dashed red line the anti-Pareto frontier from Example \ref{ex:multipartite}.

\subsubsection*{Outline of the section}

It turns out that the anti-Pareto optimization problem can be recast as a Pareto
optimization problem by changing signs and exchanging the cost and loss functions. In
order to make use of this property for the kernel and SIS models, we study the Pareto
problem under assumptions on the cost that are general enough to cover the choices
$\costu$ and $- \loss$, and assumptions on the loss that cover the choices $R_e$, $\I$ and
$-\costu$.

\medskip

The main result of this section states that all the solutions of the optimization
Problems~\eqref{eq:Prob1} or~\eqref{eq:Prob2} are Pareto optimal, and gives a description
of the Pareto frontier $\mathcal{F}_\loss$ as a graph in Section \ref{sec:Pareto-F}, and
similarly for the anti-Pareto frontier in Section \ref{sec:Pareto-AF}. Surprisingly, the
problem is not completely symmetric, compare Lemma \ref{lem:c-dec+L-hom} used for the
Pareto frontier and Lemmas \ref{lem:L**} and \ref{lem:L**=I} used for the anti-Pareto
frontier. In the latter lemmas, notice the kernel considered is quasi-irreducible, whereas
this condition is not needed for the Pareto frontier.

\subsection{On the Pareto frontier}\label{sec:Pareto-F}

We first check that Problems \eqref{eq:Prob1} and \eqref{eq:Prob2} have solutions.

\begin{proposition}[Optimal solutions for fixed cost or fixed
  loss]\label{prop:main-result}
  Suppose that Assumption~\ref{hyp:loss+cost} holds. For any cost $c\in[0,\maxcost]$,
  there exists a minimizer of the loss under the cost constraint $C(\cdot)\leq c$, that
  is, a solution to Problem~\eqref{eq:Prob1}. Similarly, for any loss
  $\ell\in[0,\maxloss]$, there exists a minimizer of the cost under the loss constraint
  $\loss(\cdot) \leq \ell$, that is a solution to Problem~\eqref{eq:Prob2}.
\end{proposition}

\begin{proof}
  Let $c \in [0,\maxcost]$. The set $\{\eta\in \Delta \,\colon\, C(\eta)\leq c\}$ is
  non-empty as it contains $\un$1 since $C(\un)=0$. It is also compact as $C$ is continuous on
  the compact set $\Delta$ (for the weak topology). Therefore, since the loss function
  $\loss$ is continuous (for the weak topology), we get that $\loss$ restricted to this
  compact set reaches its minimum. Thus, Problem~\eqref{eq:Prob1} has a solution. The
  proof is similar for the existence of a solution to Problem~\eqref{eq:Prob2}.
\end{proof}

We start by a general result concerning the links between the three problems.

\begin{proposition}[Single-objective and bi-objective problems]\label{thm:single-bi}
  Suppose Assumption \ref{hyp:loss+cost} holds.
  \begin{propenum}
  \item\label{single-bi} If $\eta_\star$ is Pareto optimal, then $\eta_\star$ is a
    solution of~\eqref{eq:Prob1} for the cost $c=C(\eta_\star)$, and a solution
    of~\eqref{eq:Prob2} for the loss $\ell=\loss(\eta_\star)$. Conversely, if $\eta_\star$
    is a solution to both problems \eqref{eq:Prob1} and \eqref{eq:Prob2} for some values
    $c$ and $\ell$, then $\eta_\star$ is Pareto optimal.
  \item\label{intersection} The Pareto frontier is the intersection of the graphs of
    $\Cinf$ and $\lossinf$: \[ \mathcal{F}_\loss = \{ (c,\ell)\in [0, \maxcost]\times [0,
    \maxloss]\, \colon\, c=\Cinf(\ell) \text{ and } \ell = \lossinf(c)\}. \]
  \item\label{special_points} The points $(0,\lossinf(0))$ and $(\Cinf(0),0)$ both belong
    to the Pareto frontier, and we have $\Cinf(\lossinf(0)) = \lossinf(\Cinf(0))=0$.
    Moreover, we also have $\Cinf(\ell) = 0$ for $\ell\in [\lossinf(0), \maxloss]$, and
    $\lossinf(c) = 0$ for $c\in [\Cinf(0), \maxcost]$.
  \end{propenum}
\end{proposition}

\begin{proof}
  Let us prove~\ref{single-bi}. If $\eta_\star$ is Pareto optimal, then for any strategy
  $\eta$, if $C(\eta)\leq C(\eta_\star)$ then $\loss(\eta)\geq \loss(\eta_\star)$ by
  taking the contraposition in~\eqref{eq:defParetoOptimal}, and $\eta_\star$ is indeed a
  solution of Problem~\eqref{eq:Prob1} with $c=C(\eta_\star)$. Similarly $\eta_\star$ is a
  solution of Problem~\eqref{eq:Prob2}.

  For the converse statement, let $\eta_\star$ be a solution of \eqref{eq:Prob1} for some
  $c$ and of \eqref{eq:Prob2} for some~$\ell$. It is also a solution of \eqref{eq:Prob1}
  with $c=C(\eta_\star)$. In particular, we get that for $\eta\in \Delta$, $\loss(\eta)<
  \loss(\eta_\star)$ implies that $C(\eta)>c=C(\eta_\star)$, which is the second part of
  \eqref{eq:defParetoOptimal}. Similarly, use that $\eta_\star$ is a solution to
  \eqref{eq:Prob2}, to get that the first part of \eqref{eq:defParetoOptimal} also holds.
  Thus the strategy $\eta_\star$ is Pareto optimal.

  \medskip

  To prove Point~\ref{intersection}, we first prove that $
  \mathcal{F}_\loss $ is a subset of $\{ (c,\ell)\, \colon\, c=\Cinf(\ell) \text{ and } \ell =
  \lossinf(c)\}$.
  A point in $\mathcal{F}_\loss$ may be written as $(C(\eta_\star),
  \loss(\eta_\star))$ for some Pareto optimal strategy $\eta_\star$.
  By Point~\ref{single-bi}, $\eta_\star$ solves
  Problem~\eqref{eq:Prob1}
  for the cost $C(\eta_\star)$, so $\lossinf(C(\eta_\star)) =
  \loss(\eta_\star)$. Similarly, we have
  $\Cinf(\loss(\eta_\star)) =
  C(\eta_\star)$, as claimed.

  We now prove the reverse inclusion. Assume that $c=\Cinf(\ell)$ and $\ell= \lossinf(c)$,
  and consider $\eta$ a solution of Problem~\eqref{eq:Prob2} for the loss~$\ell$:
  $\loss(\eta) \leq \ell$ and $C(\eta) = \Cinf(\ell) = c$. Then $\eta$ is admissible for
  Problem~\eqref{eq:Prob1} with cost $c=\Cinf(\ell)$, so $\loss(\eta) \geq
  \lossinf(\Cinf(\ell)) = \lossinf(c) = \ell$. Therefore, we get $\loss(\eta) =
  L^\star(c)$, and $\eta$ is also a solution of Problem~\eqref{eq:Prob1}. By
  Point~\ref{single-bi}, $\eta$ is Pareto optimal, so $(C(\eta),\loss(\eta)) = (c,\ell)\in
  \mathcal{F}_\loss$, and the reverse inclusion is proved. \medskip

  Finally we prove Point~\ref{special_points}. We have $\Cinf(0)=\min \{ C( \eta)\,
    \colon\,
    \eta\in \Delta \text{ and }
  \loss(\eta)=0 \}\in [0, \maxcost]$. Let $\eta\in \Delta$ such that $\loss(\eta)=0$ and
  $C(\eta)=\Cinf(0)$. We deduce that $\lossinf(\Cinf(0))\leq \loss(\eta)=0$ and
  thus $\lossinf(\Cinf(0))=0$ as $\loss$ is non-negative. We deduce
  from~\ref{intersection} that $(\Cinf(0), 0)$ belongs to $
  \mathcal{F}_\loss $. Since $\Cinf$ is non-increasing, we also get
  that $\Cinf=0$ on $[\Cinf(0), \maxcost]$. The other properties of
  \ref{special_points} are proved similarly.
\end{proof}

\begin{figure}
  \begin{subfigure}[T]{.5\textwidth}
    \centering
    \includegraphics[page=1]{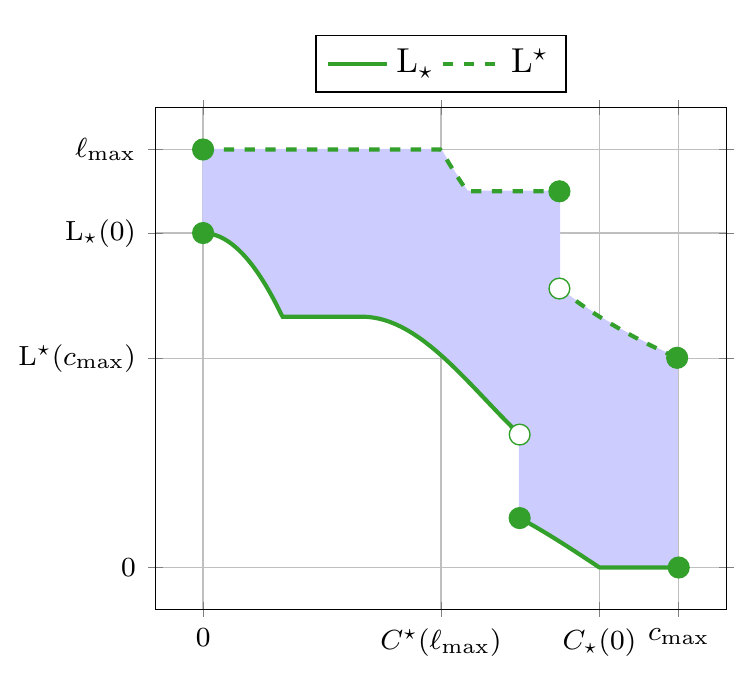}
    \caption{Value functions for Problems \eqref{eq:Prob1} and \eqref{eq:Prob1**}.}
    \label{fig:value_F}
  \end{subfigure}%
  \begin{subfigure}[T]{.5\textwidth} \centering
    \includegraphics[page=2]{frontiers}
    \caption{Value functions for Problems \eqref{eq:Prob2} and \eqref{eq:Prob2**}.}\label{fig:frontier}
  \end{subfigure}
  \begin{subfigure}[T]{.5\textwidth}
    \centering
    \includegraphics[page=3]{frontiers}
    \caption{Pareto and anti-Pareto frontier.}
  \end{subfigure}%
    \begin{subfigure}[T]{.5\textwidth}
      \centering
      \includegraphics[page=4]{frontiers}
      \caption{Pareto and anti-Pareto frontier under additional regularity
      Assumptions~\ref{hyp:cost}-\ref{hyp:loss**}.}\label{fig:reg_frontier}
  \end{subfigure}
  \caption{An example of the possible aspects of the feasible region $\FF$ (in light blue), the value
 functions $\lossinf$, $\lossup$, $\Cinf$, $\Csup$, and the Pareto and anti-Pareto frontier (in red) under
 Assumption~\ref{hyp:loss+cost}.}
  \label{fig:typical_frontier}
\end{figure}

The next two hypotheses on $C$ and $\loss$ will imply that the Pareto frontier is
connected.

\begin{hyp}\label{hyp:cost}
  If the cost $C$ has a local minimum (for the weak topology) at $\eta$, then $C(\eta)=0$
  and $\eta$ is a global minimum of $C$.
\end{hyp}

\begin{hyp}\label{hyp:loss}
  If the loss $\loss$ has a local minimum (for the weak topology) at $\eta$, then
  $\loss(\eta) = 0$ and $\eta$ is a global minimum of $\loss$.
\end{hyp}

Under these hypotheses, the picture becomes much nicer, see Figure~\ref{fig:reg_frontier},
where the only flat parts of the graphs of $\Cinf$ and $\lossinf$
occur at zero cost or zero loss. 

\begin{proposition}\label{prop:f_properties}
  Under Assumption \ref{hyp:loss+cost} and~\ref{hyp:cost} the following properties hold:
  \begin{propenum}
  \item\label{prop:c_star_decreases} The optimal cost $\Cinf$ is decreasing on $[0,
    \lossinf(0)]$.
  \item\label{prop:l_constraint_binding} If $\eta$ solves Problem~\eqref{eq:Prob2} for the
    loss $\ell\in[0,\lossinf(0)]$,
    then $\loss(\eta) = \ell$ (that is, the constraint is binding). Moreover $\eta$ is
    Pareto optimal, and:
    \begin{equation}\label{eq:LC=id}
      \lossinf(\Cinf(\ell)) = \ell.
    \end{equation}
  \item\label{prop:frontier=graph_c_star} The Pareto frontier is the graph of $\Cinf$:
    \begin{equation}\label{eq:FL=L*}
      \mathcal{F}_\loss = \{(\Cinf(\ell), \ell) \, \colon \, \ell \in [0,\lossinf(0)]\}.
    \end{equation}
  \end{propenum}
  Similarly, under Assumptions \ref{hyp:loss+cost} and~\ref{hyp:loss}, the following
  properties hold:
  \begin{propenum}[resume]
  \item\label{prop:l_star_decreases} The optimal loss $\lossinf$ is decreasing on $[0,
    \Cinf(0)]$.
  \item\label{prop:c_constraint_binding} If $\eta$ solves Problem~\eqref{eq:Prob1} for the
    cost $c\in[0,\Cinf(0)]$,
    then $C(\eta) =c$. Moreover $\eta$ is Pareto optimal, and $\Cinf(\lossinf(c)) = c$.
  \item\label{prop:frontier=graph_l_star} The Pareto frontier is the graph of $\lossinf$:
    \begin{equation}\label{eq:FL=C*} \mathcal{F}_\loss = \{(c, \lossinf(c)) \, \colon \, c
      \in [0,\Cinf(0)]\}.
    \end{equation}
  \end{propenum}

  Finally, if Assumptions \ref{hyp:cost} and \ref{hyp:loss} hold, then $\lossinf$ is a
  continuous decreasing bijection of $[0,\Cinf(0)]$ onto $[0, L^\star(0)]$ and $\Cinf$ is
  the inverse bijection, and the Pareto frontier is compact and connected.
\end{proposition}

\begin{proof}
  We prove \ref{prop:c_star_decreases}. Let $0\leq \ell<\ell' \leq \lossinf(0)$, and let
  $\eta_\star$ be a solution of Problem~\eqref{eq:Prob2}:
  \begin{equation}
    \label{eq:etaSolvesProb2}
    C(\eta_\star) = \Cinf(\ell) \quad\text{and}\quad \loss(\eta_\star)
    \leq \ell.
  \end{equation}
  The set $\mathcal{O} = \{\eta\, \colon\,  \loss(\eta) < \ell'\}$ is open and contains $\eta_\star$.
  Since $\loss(\eta_\star)<\lossinf(0)$, we get $C(\eta_\star)>0$, so $\eta_\star$ is not
  a global minimum for $C$. By Assumption~\ref{hyp:cost}, it cannot be a local minimum for
  $C$, so $\mathcal{O}$ contains at least one point $\eta'$ for which
  $C(\eta')<C(\eta_\star)$. Since $\eta'\in \mathcal{O}$, we get $\loss(\eta')\leq \ell'$,
  so that $\Cinf(\ell') \leq C(\eta') < C(\eta_\star)= \Cinf(\ell)$. Since $\ell<\ell'$
  are arbitrary, $\Cinf$ is decreasing on $[0,\lossinf(0)]$. \medskip

  We now prove~\ref{prop:l_constraint_binding}. If the inequality
  in~\eqref{eq:etaSolvesProb2} was strict, that is $\loss(\eta_\star)<\ell$, then we would
  get a contradiction as $C(\eta_\star) \geq \Cinf(\loss(\eta_\star)) > \Cinf(\ell) =
  C(\eta_\star)$. Therefore any solution $\eta_\star$ of \eqref{eq:Prob2} satisfies
  $\loss(\eta_\star) = \ell$, and in particular $\Cinf(\loss(\eta_\star))=\Cinf(\ell) =
  C(\eta_\star)$. This implies in turn that $\eta_\star$ also solves~\eqref{eq:Prob1}: if
  $\eta$ satisfies $\loss(\eta)<\loss(\eta_\star)$, then using the definition of $\Cinf$,
  the fact that it decreases, and the definition of $\eta_\star$, we get:
  \[
    C(\eta) \geq \Cinf(\loss(\eta)) > \Cinf(\loss(\eta_\star)) = C(\eta_\star).
  \]
  By contraposition, we have $\loss(\eta)\geq \loss(\eta_\star)$ for any $\eta$ such that
  $C(\eta)\leq C(\eta_\star)$, proving that $\eta_\star$ is also a solution
  of~\eqref{eq:Prob1} with $c=C(\eta_*)$. By Point~\ref{single-bi} of
  Proposition~\ref{thm:single-bi}, $\eta_\star$ is Pareto optimal. Therefore
  $(C(\eta_\star),\loss(\eta_\star)) = (\Cinf(\ell),\ell)$ belongs to the Pareto frontier.
  Using Point~\ref{intersection} of Proposition~\ref{thm:single-bi}, we deduce that
  $\ell=\lossinf(\Cinf(\ell))$. \medskip

  To prove Point~\ref{prop:frontier=graph_c_star}, note that Equation \eqref{eq:LC=id}
  shows that, if $c=\Cinf(\ell)$ for $\ell\in [0, \lossinf(0)]$, then $\ell=\lossinf(c)$.
  Use Point~\ref{intersection} and
  \ref{special_points} of Proposition~\ref{thm:single-bi}, to get that $\mathcal{F}_\loss
  = \{ (c,\ell)\,\colon\, c=\Cinf(\ell), \, \ell \in [0, \lossinf(0)]\}$. \medskip

  The claims \ref{prop:l_star_decreases}, \ref{prop:c_constraint_binding} and
  \ref{prop:frontier=graph_l_star} are proved in the same way, exchanging the roles of
  $\loss$ and~$C$.

  \medskip

  To conclude the proof, it remains to check that $\Cinf$ and $\lossinf$ are continuous
  under Assumptions \ref{hyp:loss+cost}, \ref{hyp:cost} and \ref{hyp:loss}. We deduce from
  Point~\ref{prop:l_constraint_binding} and Proposition~\ref{prop:main-result} that $[0,
  \lossinf(0)]$ is in the range of $\lossinf$. Since $\lossinf $ is decreasing, thanks to
  Point~\ref{prop:l_star_decreases} and $\lossinf(\Cinf(0))=0$, see
  Proposition~\ref{thm:single-bi}~\ref{special_points}, we get that $\lossinf$ is
  continuous and decreasing on $[0, \lossinf(0)]$, and thus one-to-one from $[0,\Cinf(0)]$
  onto $[0, \loss^\star(0)]$. Then use \eqref{eq:LC=id} to get that $\Cinf$ is its inverse
  bijection. The continuity of $\lossinf$ and \eqref{eq:FL=C*} implies that $
  \mathcal{F}_\loss$ is compact and connected.
\end{proof}

Finally, let us check that Assumptions~\ref{hyp:cost} and~\ref{hyp:loss} hold under very
simple assumptions, which are in particular satisfied by the cost functions $\costu$ and
$\costa$ and the loss functions $R_e$ and $\I$ (recall from Propositions \ref{prop:R_e}
and \ref{prop:I} that $R_e$ and $\I$ are sub-homogeneous).

\begin{lemma}\label{lem:c-dec+L-hom}
  Suppose Assumption \ref{hyp:loss+cost} holds. If the cost function $C$ is decreasing,
  then Assumption \ref{hyp:cost} holds and $\lossinf (0)=\maxloss$. If the loss function
  $\loss$ is sub-homogeneous, then Assumption \ref{hyp:loss} holds.
\end{lemma}

\begin{proof}
  Let $\eta\in\Delta$. If $C$ has a local minimum at~$\eta$, then, as
  $C$ is non-increasing,  for $\varepsilon>0$
  small enough, we get that $C(\eta) \geq C(\eta+\varepsilon(1-\eta)) \geq C(\eta)$. If $C$ is
  decreasing, this is only possible if $\eta = 1$ almost surely, so that $\eta$ is a
  global minimum of $C$. This also gives $\lossinf(0)=\maxloss$. Similarly if $\loss$ has
  a local minimum at $\eta$, then for $\varepsilon>0$ small enough $\loss(\eta) \leq
  \loss((1-\varepsilon) \eta) \leq (1-\varepsilon) \loss(\eta)$, so $\loss(\eta) = 0$ and
  $\eta$ is a global minimum of $\loss$.
\end{proof}

\begin{corollary}\label{cor:P=K}
  Suppose that Assumptions \ref{hyp:loss+cost}, \ref{hyp:cost} and \ref{hyp:loss} hold.
  The set of Pareto optimal strategies $\mathcal{P}_\loss$ is compact (for the weak
  topology).
\end{corollary}

\begin{proof}
  Since $\lossinf$ is continuous thanks to Proposition~\ref{prop:f_properties}, we deduce
  that $\mathcal{F}_\loss$, which is given by~\eqref{eq:FL=C*}, is compact and thus
  closed. Since $\mathcal{P}_\loss=f^{-1}(\mathcal{F}_\loss)$, where the function
  $f=(C,\loss)$ defined on $\Delta$ is continuous, we deduce that $\mathcal{P}_\loss$ is
  closed and thus compact as $\Delta$ is compact.
\end{proof}

\subsection{On the anti-Pareto frontier}\label{sec:Pareto-AF}

\subsubsection{The general setting}
Letting $C'(\eta) = \maxloss - \loss(\eta)$ and $\loss'(\eta) = \maxcost -
C(\eta)$, it is easy to see that:
\[
 \Cinf'(c) =\maxloss - \lossup(\maxcost-c)
 \quad\text{and}\quad
 \lossinf'(\ell) = \maxcost - \Csup(\maxloss - \ell),
\]
so that Proposition~\ref{prop:f_properties} may be applied to the cost function
$C'$ and the loss function $\loss'$ to yield
the following result.

\begin{proposition}[Single-objective and bi-objective problems for the anti-Pareto
  strategies]\label{thm:single-bi-anti}
 Suppose Assumption \ref{hyp:loss+cost} holds.
 \begin{propenum}
 \item\label{single-bi-anti}
 If $\eta ^\star$ is anti-Pareto optimal, then $\eta^\star$ is a solution of~\eqref{eq:Prob1**}
 for the cost $c=C(\eta^\star)$, and a solution of~\eqref{eq:Prob2**} for the loss
 $\ell=\loss(\eta^\star)$.
 Conversely, if $\eta^\star$ is a solution to both problems
 \eqref{eq:Prob1**}
 and \eqref{eq:Prob2**} for some values $c$ and $\ell$,
 then $\eta^\star$ is anti-Pareto optimal.
\item\label{intersection-anti}
 The anti-Pareto frontier is the intersection of the graphs of $\Csup$
 and $\lossup$:
 \[
 \mathcal{F}_\loss^\mathrm{Anti} = \{ (c,\ell)\in [0, \maxcost]\times [0,
 \maxloss]\, \colon\, c=\Csup(\ell) \text{ and } \ell =
 \lossup(c)\}.
 \]

 \item\label{special_points-anti}
 The points $(\Csup(\maxloss),\maxloss)$ and $(\maxcost,\lossup(\maxcost))$ both belong
 to the anti-Pareto frontier, and we have
 $\Csup(\lossup(\maxcost)) =\maxcost$ and $\lossup(\Csup(\maxloss)) =
 \maxloss$.
 Moreover, we also have $\Csup(\ell) = \maxcost$ for $\ell\in
 [0, \lossup(\maxcost)]$,
 and $\lossup(c) = \maxloss$ for $c\in [0,\Csup(\maxloss)]$.
 \end{propenum}
\end{proposition}

The following additional hypotheses rule out the occurrence of flat
parts in the anti-Pareto frontier.
\begin{hyp}
 \label{hyp:cost**}
 If the cost $C$ has a local maximum at $\eta$ (for the weak
 topology), then $C(\eta)=\maxcost$ and $\eta$ is a global maximum of
 $C$.
\end{hyp}

\begin{hyp}
 \label{hyp:loss**}
 If the loss $\loss$ has a local maximum at $\eta$ (for the weak
 topology), then $\loss(\eta) = \maxloss$ and $\eta$ is a global
 maximum of $\loss$.
\end{hyp}

The following result is now a consequence of
Proposition~\ref{prop:f_properties} and Corollary~\ref{cor:P=K}
applied to the loss function $\loss'$ and cost function $C'$.

\begin{proposition}\label{prop:f_properties**}
  Under Assumption \ref{hyp:loss+cost} and~\ref{hyp:cost**} the following properties hold:
  \begin{propenum}
  \item\label{prop:c_star_decreases**} The optimal cost $\Csup$ is decreasing on
    $[\Csup(\maxloss), \maxcost]$.
  \item\label{prop:l_constraint_binding**} If $\eta$ solves Problem~\eqref{eq:Prob2**} for
    the loss $\ell\in[\lossup(\maxcost),\maxloss]$, then $\loss(\eta) = \ell$ (that is,
    the constraint is binding). Moreover $\eta$ is anti-Pareto optimal, and
    $\lossup(\Csup(\ell)) = \ell$.
  \item\label{prop:frontier=graph_c_star**} The anti-Pareto frontier is the graph of
    $\Csup$:
    \begin{equation}\label{eq:FL=L**} \mathcal{F}_\loss^{\mathrm{Anti}} = \{(\Csup(\ell),
	\ell) \, \colon \, \ell \in
      [\lossup(\maxcost),\maxloss]\}.
    \end{equation}
  \end{propenum}
  Similarly, under Assumptions \ref{hyp:loss+cost} and~\ref{hyp:loss**}, the following
  properties hold:
  \begin{propenum}[resume]
  \item\label{prop:l_star_decreases**} The optimal loss $\lossup$ is decreasing on
    $[\lossup(\maxcost),\maxloss]$.
  \item\label{prop:c_constraint_binding**} If $\eta$ solves Problem~\eqref{eq:Prob1**} for
    the cost $c\in[\Csup(\maxloss),\maxcost]$, then $C(\eta) =c$. Moreover $\eta$ is
    anti-Pareto optimal, and $\Csup(\lossup(c)) = c$.
  \item\label{prop:frontier=graph_l_star**} The anti-Pareto frontier is the graph of
    $\lossup$:
    \begin{equation}\label{eq:FL=C**} \mathcal{F}_\loss^{\mathrm{Anti}} = \{(c,
      \lossup(c)) \, \colon \, c \in [\Csup(\maxloss),\maxcost]\}.
    \end{equation}
  \end{propenum}

  Finally, if Assumptions \ref{hyp:loss+cost}, \ref{hyp:cost**} and \ref{hyp:loss**} hold,
  then $\lossup$ is a continuous decreasing bijection of $[\Csup(\maxloss),\maxcost]$ onto
  $ [\lossup(\maxcost),\maxloss]$, $\Csup$ is the inverse bijection, and the anti-Pareto
  frontier is compact and connected. Furthermore,  the set of
  anti-Pareto optimal strategies $\mathcal{P}_\loss^{\mathrm{Anti}}$ is compact (for the
  weak 
  topology).
\end{proposition}

The following result is similar to the first part of Lemma \ref{lem:c-dec+L-hom}.

\begin{lemma}\label{lem:c-dec+L-hom**}
  Suppose Assumption \ref{hyp:loss+cost} holds. If the cost function $C$ is decreasing,
  then Assumption \ref{hyp:cost**} holds and $\lossup (\maxcost)=0$.
\end{lemma}

\begin{proof}
  Let $\eta\in\Delta$ and $\varepsilon\in (0,1)$. Since $C$ is decreasing,
  $C((1-\varepsilon)\eta) \geq C(\eta)$, with equality if and only if $\eta=0$
  $\mu$-almost surely. Therefore the only local maximum of $C$ is $\eta=0$, and it is a
  global maximum. Since $C(\eta)=\maxcost$ implies that $\eta=0$ $\mu$-almost surely, we
  also get that $\lossup(\maxcost)=\loss(0)=0$.
\end{proof}

\subsubsection{The particular case of the kernel and SIS models}

We show that, under an irreducibility hypothesis on the kernel,
Assumption~\ref{hyp:loss**} holds for the loss functions $R_e$ and $\mathfrak{J}$. The
reducible case is more delicate and it is studied in more details in~\cite{ddz-Re} for the
loss function $\loss=R_e$; in particular Assumption~\ref{hyp:loss**} may not hold and the
anti-Pareto frontier may not be connected.

Let us recall some notation. Let $\kk$ be a kernel with finite double norm. For $A, B\in
\cf$, we write $A\subset B$ a.s.\ if $\mu(B \cap A^c)=0$ and $A=B$ a.s.\ if  $A\subset B$
a.s.\ and $B \subset A$ a.s. For $A, B\in \cf$, $x\in \Omega$ and a kernel
$\kk$, we simply write $\kk(x,A)=\int_{ A} \kk(x,y)\, \mu(\rd y)$, $ \kk(B,x)=\int_{ B}
\kk(z,x)\, \mu(\rd z)$ and:
\[
  \kk(B, A)=
  \int_{B \times A} \kk(z,y)\, \mu(\rd z) \mu(\rd y).
\]
A set $A \subset \cf$ is \emph{$\kk$-invariant} if $\kk(A^c, A)=0$. (Notice that if $A$ is
$\kk$-invariant, then $L^p(A,\mu)$ is an invariant closed subspace for $T_\kk$, seen as an
operator on $L^p(\Omega, \mu)$.) A kernel $\kk$ is \emph{irreducible} (or connected) if
any $\kk$-invariant set $A$ is such that a.s. $A=\emptyset$ or a.s.  $A=\Omega$. Define
$\{\kk \equiv 0\}$ as $\{x\in \Omega\, \colon\, \kk(x, \Omega) + \kk( \Omega, x)=0\}$, so
that $\kk(A,\Omega)+ \kk(\Omega, A)=0$ implies that a.s.\ $A\subset\{\kk \equiv 0\}$. A
kernel $\kk$ is \emph{quasi-irreducible} if the restriction of $\kk$ to $\{\kk \equiv
  0\}^c$ is irreducible, that is if any $\kk$-invariant set $A$ is such that $A\subset \{\kk
\equiv 0\}$ a.s.\ or  $A^c \subset \{\kk \equiv 0\}$ a.s. Notice the definition of the
quasi-irreducibility from  \cite[Definition~2.11]{bjr} is slightly stronger as it uses
a topology on $\Omega$.

\begin{lemma}\label{lem:L**}
  Consider the kernel model $\param=[(\Omega, \cf, \mu), \kk]$ under
  Assumption~\ref{hyp:k}. If $\kk$ is quasi-irreducible, then Assumption
  \ref{hyp:loss**} holds for $\loss=R_e[\kk]$, and $\Csup (\maxloss)=C(\ind{\{\kk\equiv
  0\}^c})$ (which is 0 if $\kk$ is irreducible).
\end{lemma}

\begin{proof}
  The quasi-irreducible case can easily be deduced from the irreducible case, so we assume that
  $\kk$ is irreducible. In particular, we have $\kk(\Omega, y)>0$ almost surely. Let
  $\eta\in\Delta$ be a local maximum of $R_e$ on $\Delta$; we want to show that it is also
  a global maximum.

  Suppose first that $\inf\eta >0$. Then $\kk\eta$ is irreducible with finite double norm.
  According \cite[Theorem V.6.6 and Example V.6.5.b]{schaefer_banach_1974}, the eigenspace
  of $T_{\kk \eta}$ associated to $R_e(\eta)$ is one-dimensional and it is spanned by a
  vector $\vd$ such that $\vd>0$ almost surely, and the corresponding left eigenvector
  associated to $R_e(\eta)$, say $\vg$, can be chosen such that $\langle \vg,
  \vd\rangle=1$ and $\vg>0$ almost surely. According to \cite[Theorem
  2.6]{EffectivePertuBenoit}, applied to $L_0=T_{\kk\eta}$ and $L = T_{\kk(\eta +
  \varepsilon(1-\eta))}$, we have, using that $\norm{L_0 - L}=O(\varepsilon)$ thanks to
  \eqref{eq:double-norm-norm}: \[
    R_e( \eta+\varepsilon(1-\eta)) = R_e(\eta) + \varepsilon \langle \vg,T_{\kk(1-\eta)}
  \vd \rangle + O(\varepsilon^2). \] Since $R_e$ has a local maximum at $\eta$, the
  first order term on the right hand side vanishes, so $\vg(x)\kk(x,y)(1-\eta(y))\vd(y)
  = 0$ for $\mu$ almost all $x$ and $y$. Since $\vg$ and $\vd$ are positive almost surely
  and $\kk$ is irreducible, we get that $\kk(\Omega, y)(1- \eta(y))=0$ almost surely and
  thus $\eta(y) = 1$ almost surely. Therefore $\eta = \un$, which is a global maximum for
  $R_e$.

 Finally, suppose that $\inf\eta = 0$. Let $\mathcal{O}$ be an open subset of $\Delta$ on
 which $R_e\leq R_e(\eta)$ and with $\eta\in \mathcal{O}$. For $\varepsilon>0$ small enough, the strategy
 $\eta_\varepsilon = \eta +\varepsilon(1-\eta)$ belongs to $\mathcal{O}$ and satisfies
 $R_e(\eta) \leq R_e(\eta_\varepsilon) \leq R_e(\eta)$ (where the first inequality comes
 from the fact that $R_e$ is non-decreasing). Therefore $\eta_\varepsilon$ is a local
 maximum, and thus $\eta_\varepsilon=\un$ almost surely. This readily implies that $\eta=\un$
 almost surely.

  We deduce that if $\eta$ is a local maximum, then $\eta=1$ almost surely. Thus $\eta$ is
  a global maximum and $\Csup(\maxloss)=C(\un)=0$. This ends the proof.
\end{proof}

\begin{lemma}\label{lem:L**=I}
  Consider the SIS model $\param=[(\Omega, \cf, \mu), k, \gamma]$ under
  Assumption~\ref{hyp:k-g}. If $k$ is quasi-irreducible, then Assumption \ref{hyp:loss**}
  holds for $\loss=\I$ and $\Csup (\maxloss)=C(\ind{\{\kk\equiv 0\}^c})$ (which is 0 if
  $k$ is irreducible).
\end{lemma}

\begin{proof}
  The quasi-irreducible case can easily be deduced from the irreducible case, so we assume
  that $k$ is irreducible.

  Set $\kk=k/\gamma$. Suppose that $\I$ has a local maximum at some $\eta\in\Delta$. For
  $\varepsilon\in (0, 1)$, the kernel $\kk {\eta_\varepsilon}$, with $\eta_\varepsilon=
  \eta+\varepsilon(1-\eta)$, is irreducible (with finite double norm) since $k$ is
  irreducible and $\gamma$ is positive and bounded. We have that for $\varepsilon>0$ small
  enough:
  \[ \I(\eta) \geq \I(\eta_\varepsilon)= \int_\Omega \mathfrak{g}_{\eta_\varepsilon}\,
    \eta_\varepsilon \, \mathrm{d}\mu \geq \int_\Omega \mathfrak{g}_{\eta_\varepsilon}\,
    \eta\, \mathrm{d}\mu
    \geq \int_\Omega
    \mathfrak{g}_{\eta}\, \eta\, \mathrm{d}\mu
    = \I(\eta),
  \]
  where we used that $\eta\leq \eta_\varepsilon$ and $0\leq \mathfrak{g}_\eta \leq
  \mathfrak{g}_{\eta_\varepsilon}$, see \eqref{eq:mon-gh}. Therefore all these quantities
  are equal. Since the equilibrium $\mathfrak{g}_{\eta_\varepsilon}$ is $\mu$-a.e.
  positive thanks to \cite[Remark~4.11]{delmas_infinite-dimensional_2020} as $\kk
  \eta_\varepsilon$ is irreducible, we must have $\eta_\varepsilon = \eta$ a.s, which is only
  possible if $\eta = 1$ almost surely.

  Since $\I(\un)> \I(\eta)$ for any $\eta\neq \un$, we also get $\Csup(\maxloss)=0$, with
  $\maxloss=\I(\un)=\int_\Omega \mathfrak{g}\, \rd \mu$.
\end{proof}

\section{Miscellaneous properties for set of outcomes and the Pareto
frontier}\label{sec:diversCL}

We prove results concerning the feasible region, the stability of the Pareto frontier and
its geometry.

\subsection{No holes in the feasible region}

We check there is no hole in the feasible region.

\begin{proposition}\label{prop:trou}
  Suppose that Assumption \ref{hyp:loss+cost} holds. The feasible region~$\FF$ is compact,
  path connected, and its complement is connected in $\R^2$. It is the whole region
  between the graphs of the one-dimensional value functions:
  \begin{align*}
    \FF &= \{ (c,\ell) \in \R^2 \,\colon\, 0\leq c \leq \maxcost,\,
    \lossinf(c) \leq \ell \leq \lossup(c) \} \\
	&= \{ (c,\ell) \in \R^2 \,\colon\, 0\leq \ell \leq \maxloss,\,
	\Cinf(\ell) \leq c \leq \Csup(\ell)\}.
  \end{align*}
\end{proposition}

\begin{proof}
  The region $\FF$ is compact and path-connected as a continuous image by $(C,\loss)$ of
  the compact, path-connected set $\Delta$.

  By symmetry, it is enough to prove that $\FF$ is equal to
  $F_1=\{ (c,\ell) \in \R^2 \,\colon\, 0\leq c \leq \maxcost,\, \lossinf(c) \leq \ell \leq
  \lossup(c) \} $. Let $(c,\ell) \in \FF$ and $\eta\in \Delta$ be such that $(c,\ell) =
  (C(\eta),\loss(\eta))$. By definition of $\lossinf$ and $\lossup$, we have: $
  \lossinf(c) = \lossinf(C(\eta)) \leq \loss( \eta) \leq \lossup(C(\eta)) = \lossup(c)$.
  We deduce that $(c,\ell)\in F_1$.

  \medskip

  Let us now prove that $F_1\subset \FF$. Let us first consider a point of the form $(c,
  \lossinf(c))$, where $0\leq c \leq \maxcost$. By definition, there exists $\eta$ such
  that $C(\eta) \leq c$ and $\loss(\eta) = \lossinf(c)$. Let $\eta_t = t\eta$. The map
  $t\mapsto C(\eta_t)$ is continuous from $[0,1]$ to $[C(\eta),\maxcost]$, and
  $c\in[C(\eta),\maxcost]$, so there exists $s$ such that $C(\eta_s) = c$. Since $\loss$
  is non-decreasing, $\loss(\eta_s)\leq \loss(\eta)$. By definition of $\lossinf(c)$,
  $\loss(\eta_s)\geq \lossinf(c)$. Therefore $(c,\lossinf(c)) = (C(\eta_s),\loss(\eta_s))$
  belongs to $\FF$. Similarly the graphs of $\Cinf$, $\Csup$ and $\lossup$ are also
  included in $\FF$.

  So, it is enough to check that, if $A=(c, \ell)$ is in $F_1$, with $c\in(0,\maxcost)$
  and $\ell\in(\lossinf(c),\lossup(c))$, then $A$~belongs to $\FF$. We shall assume that
  $A\not\in \FF$ and derive a contradiction by building a loop in $\FF$ that encloses $A$
  and which can be continuously contracted into a point in $\FF$.

  Since $\lossinf(c) < \ell< \lossup(c)$, there exist $\etainf$ and $\etasup$ such that:
  \[
    C(\etainf) \leq c, \quad \loss(\etainf) < \ell, \quad C(\etasup) \geq c \quad
  \text{and}\quad \loss(\etasup) > \ell.
  \]
  We concatenate the four paths defined for $u\in [0, 1]$:
  \[
    u \mapsto u\etainf, \quad
    u\mapsto (1-u) \etainf + u, \quad u\mapsto (1-u) + u\etasup \quad\text{and}\quad
  u\mapsto (1-u) \etasup,
  \]
  to obtain a continuous loop $(\eta_t, t\in [0, 4])$ from $[0,4]$ to $\Delta$, such that:
  \[ \eta_0 = \eta_4 = 0, \quad \eta_1 = \etainf, \quad \eta_2 = 1 \quad\text{and}\quad
  \eta_3 = \etasup.
  \]

  We now define a continuous family of loops $(\gamma_s, s\in [0, 1])$ in $\R^2$ by
  \[
    \gamma_s(t) = (C(s \eta_t),\loss(s\eta_t), t\in [0, 4]).
  \]
  By definition, for all $s\in[0,1]$, $\gamma_s$ is a continuous loop in $F$. Since $A =
  (c,\ell) \notin \FF$, the loops $\gamma_s$ do not contain $A$, so the winding number
  $W(\gamma_s,A)$ is well-defined (see for example \cite[Definition 6.1]{HJ09}). As
  $A\not\in \FF$, we get that $\gamma_s$ is a continuous deformation in
  $\R^2\setminus\{A\}$ from $\gamma_1$ to $\gamma_0$. Thanks to \cite[Theorem 6.5]{HJ09},
  this implies that $W(\gamma_s,A)$ does not depend on $s\in [0, 1]$.

  For $s=0$, the loop degenerates to the single point $(C(0),0)$ so the winding number is
  $0$. For $s=1$, let us check that the winding number is $1$, which will provide the
  contradiction. To do this, we compare $\gamma_1$ with a simpler loop $\delta$ defined
  by:
  \[ \delta (0) = \delta (4) = (\maxcost,0), \quad \delta (1) = (0,0), \quad \delta (2) =
    (0,\maxloss) \quad\text{and}\quad
    \delta (3) = (\maxcost,\maxloss),
  \]
  and by linear interpolation for non integer values of $t$: in other words, $\delta $ runs
  around the perimeter of the axis-aligned rectangle with corners $(0,0)$ and
  $(\maxcost,\maxloss)$. Clearly, we have $W(\delta , A)=1$.

  Let $M_t$, $N_t$ denote $\gamma_1(t)$ and $\delta (t)$ respectively. For $t\in [0,1]$,
  we have $N_t = ((1-t)\maxcost,0)$, so the second coordinate of $\overrightarrow{AN_t}$
  is non-positive. On the other hand $\loss(t\etainf) \leq \loss(\etainf) < \ell$, so the
  second coordinate of $\overrightarrow{AM_t}$ is negative. Therefore the two vectors
  $\overrightarrow{AN_t}$ and $\overrightarrow{AM_t}$ cannot point in opposite directions.
  Similar considerations for the other values of $t\in [1, 4]$ show that
  $\overrightarrow{AN_t}$ and $\overrightarrow{AM_t}$ never point in opposite directions.
  By \cite[Theorem 6.1]{HJ09}, the winding numbers $W(\gamma_1, A)$ and $W(\delta, A)$ are
  equal, and thus $W(\gamma_1, A)=1$.

  This gives that $A\in \FF$ by contradiction, and thus $F_1\subset \FF$.

  \medskip

  Finally, it is easy to check that $F_1$ has a connected complement, because $F_1$ is
  bounded, and all the points in $F_1^c$ can reach infinity by a straight line: for
  example, if $\ell> \lossup(c)$, then the half-line $\{(c,\ell'), \ell' \geq \ell\}$ is
  in $F_1^c$.
\end{proof}

\subsection{Stability}
We can consider the stability of the Pareto frontier and the set of
Pareto optima. Recall that, thanks to~\eqref{eq:FL=C*}, the graph
$\{(c,\lossinf(c)) \, \colon \, c \in [0,\maxcost]\}$ of~$\lossinf$ is
the union of the Pareto frontier and the straight line joining
$(0,\Cinf(0))$ to $(0,\maxcost)$ and can thus be seen as an extended
Pareto frontier. The proof of the following proposition is immediate. It
implies in particular the convergence of the extended Pareto
frontier. This result can also easily be adapted to the anti-Pareto frontier.

\begin{proposition}\label{prop:F-stab}
 Let $C$ be a cost function and $(\loss^{(n)}, n\in \N)$ a sequence of
 loss functions converging uniformly on $\Delta$ to a loss function
 $\loss$. Assume that Assumptions \ref{hyp:loss+cost}, \ref{hyp:cost}
 and \ref{hyp:loss} hold for the cost $C$ and the loss functions
 $\loss^{(n)}$, $n\in \N$, and $\loss$. Then $\lossinf^{(n)}$ converges
 uniformly to $\lossinf$. Let $\eta\in \Delta$ be the weak limit of a
 sequence $(\eta_n, n\in \N)$ of Pareto optima, that is
 $\eta_n\in \mathcal{P}_{\loss^  {(n)}}$ for all $n\in \N$. If
 $C(\eta)\leq \Cinf(0)$, then we
 have $\eta\in \mathcal{P}_{\loss}$.
\end{proposition}

\begin{remark}[On the continuity of the Pareto Frontier]\label{rem:cont-F}
 It might happen that some elements of $\mathcal{P}_{\loss}$ are not
 weak limit of sequence of elements of $\mathcal{P}_{\loss^  {(n)}}$; see
 \cite{ddz-2pop} for such discontinuity. It might also happen that a
 sequence $(\eta_n, \, n \in \N)$ such that
 $\eta_n \in \mathcal{P}_{\loss^  {(n)}}$ and $\loss^  {(n)}(\eta_n)>0$ converges
 to some $\eta$ that does not belong to $\mathcal{P}_\loss$ if
 $\loss(\eta) = 0$. In particular, in this case, ${\Cinf}_{,\loss^  {(n)}}(0)$
 does not converge to ${\Cinf}_{,\loss}(0)$, where ${\Cinf}_{,\loss'}$
 is the value function $\Cinf$ associated to the loss $\loss'$. This situation is represented
 in Figure~\ref{fig:stability}. In Figure~\ref{fig:perturb_kernel}, we
 have plotted a perturbation
 $\kk_\varepsilon = \kk + \varepsilon \sum_{n \in \N^\star}
 \mathds{1}_{I_n \times I_n}$ of the multipartite kernel $\kk$ defined
 in Example~\ref{ex:multipartite} for $\varepsilon>0$ small. According
 to Proposition~\ref{prop:Re-stab}, $R_e[\kk_\varepsilon]$ converges
 uniformly to $R_e[\kk]$ when $\varepsilon$ vanishes. However, the
 Pareto optimal strategies for $\kk_\varepsilon$ that cost more than
 $1/2$ do not converge to some Pareto optimal strategies for
 $\kk$. This can be seen in Figure~\ref{fig:pareto_stability}, where
 the Pareto frontier of $\kk_\varepsilon$ (in blue) corresponding to
 costs larger than 1/2 does not have a counterpart in the Pareto
 frontier of $\kk$ (in red).
\end{remark}

\begin{figure}
 \begin{subfigure}[T]{.5\textwidth} \centering
 \includegraphics[page=1]{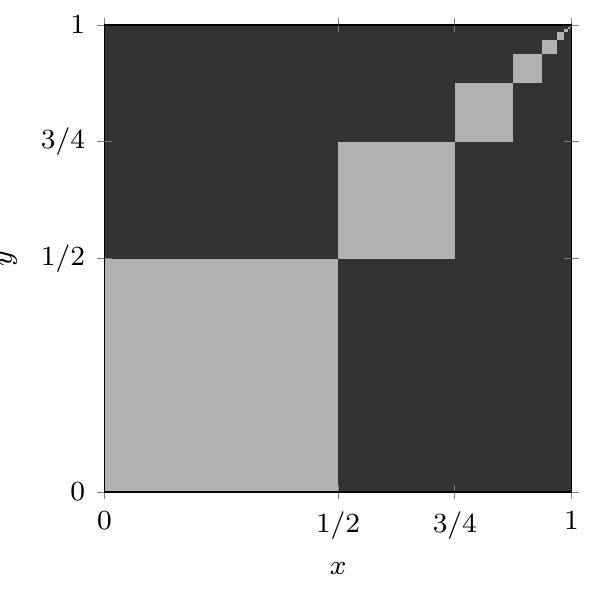}
 \caption{Grayplot of the kernel~$\kk_\varepsilon$, with~$\Omega = [0,1]$
 and~$\mu$ the Lebesgue measure ($\kk_\varepsilon$ is equal to the constant $\kappa>0$
 on the black zone and to $\varepsilon$ on the gray zone, with
 $\varepsilon>0$ small).}
 \label{fig:perturb_kernel}
 \end{subfigure}%
 \begin{subfigure}[T]{.5\textwidth} \centering
 \includegraphics[page=2]{perturbated_multipartite}
 \caption{In red, the Pareto frontier of the kernel $\kk$
 represented in Figure~\ref{fig:kernel} compared to the Pareto
 frontier of the kernel $\kk_\varepsilon$ in blue.}
 \label{fig:pareto_stability}
 \end{subfigure}%
 \caption{On the stability of the Pareto frontier} \label{fig:stability}
\end{figure}

\subsection{Geometric properties}

If the cost function is affine, then there is a nice geometric property
of the Pareto frontier.

\begin{lemma}\label{lem:affine-chord}
 Suppose that Assumption \ref{hyp:loss+cost} holds, the cost
 function is affine (\textit{i.e.}, $C=\costa$ given by
 \eqref{eq:def-C0}) and the loss function $\loss$ is sub-homogeneous.
 Then, we have $\lossinf( \theta c+ (1-\theta)\maxcost ) \leq \theta
 \lossinf(c)$ for all $c\in [0, \maxcost]$ and
 $\theta\in[0,1]$.
\end{lemma}

\begin{remark}
 Geometrically, Lemma~\ref{lem:affine-chord} means that the graph of the loss
 $\lossinf \, \colon \, [0,\maxcost] \to [0,\maxloss]$ is below its chords with end point
 $(1,\lossinf(\un)) = (1,0)$. See Figures \ref{fig:pareto_frontier} for a typical
 representation of the Pareto frontier (red solid line).
\end{remark}

\begin{proof}
 Let $c\in [0, \maxcost]$ and $\theta\in[0,1]$. Thanks to Lemma
 \ref{lem:c-dec+L-hom}, Assumption \ref{hyp:loss} holds. Thus, thanks
 to Proposition \ref{prop:f_properties}~\ref{prop:l_star_decreases},
 there exists $\eta\in \cp_\loss$ with cost $C(\eta)=c$ and thus
 $\loss(\eta)=\lossinf (c)$. Since $C$ is affine, we have:
 \[ C(\theta \eta) = \theta C(\eta)+ (1 - \theta)\maxcost \leq \theta c
 +(1- \theta)\maxcost . \] Therefore, $\theta\eta$ is admissible for
 Problem~\eqref{eq:Prob1} with cost constraint
 $C(\cdot) \leq \theta c + (1-\theta) \maxcost$. This implies that
 $\lossinf(\theta c+ (1-\theta) \maxcost) \leq \loss(\theta \eta) \leq
 \theta \lossinf(c)$, thanks to the sub-homogeneity of the loss
 function $\loss$.
\end{proof}

In some case, we shall prove that the considered loss function is convex
(which in turn implies Assumption~\ref{hyp:loss}). In this case,
choosing a convex cost function implies that Assumption \ref{hyp:cost}
holds and the Pareto frontier is convex. A similar result holds in the
concave case. We provide a
short proof of this result.

\begin{proposition}\label{prop:cvex}
  Suppose that Assumption \ref{hyp:loss+cost} holds. If the cost function $C$ and the loss
  function $\loss$ are convex, then the functions $\Cinf$ and $\lossinf$ are convex. If
  the cost function $C$ and the loss function $\loss$ are concave, then the functions
  $\Csup$ and $\lossup$ are convex.
\end{proposition}

\begin{proof}
  Let $\ell_0,\ell_1 \in [0, \maxloss]$. By Proposition \ref{prop:main-result}, there
  exist $\eta_0$, $\eta_1$ such that $\loss (\eta_i) \leq \ell_i$ and $C(\eta_i) =
  \Cinf(\ell_i)$ for $i\in \{0,1\}$. For $\theta\in [0, 1]$, let $\ell=(1-\theta) \ell_0 +
  \theta \ell_1$. Since $C$ and $\loss$ are assumed to be convex, $\eta = (1-\theta)\eta_0
  + \theta \eta_1$ satisfies:
  \[
    C(\eta) \leq (1-\theta)\Cinf(\ell_0) + \theta
    \Cinf(\ell_1) \quad\text{and}\quad \loss (\eta) \leq (1-\theta)\ell_0 + \theta \ell_1.
  \]
  Therefore, we get that $\Cinf((1-\theta)\ell_0 + \theta \ell_1) \leq C(\eta) \leq
  (1-\theta)\Cinf(\ell_0) + \theta \Cinf(\ell_1)$, and $\Cinf$ is convex. The proof of the
  convexity of $\lossinf$ is similar. The concave case is also similar.
\end{proof}

\section{Equivalence of models by coupling}\label{sec:equivalent}
Even if in full generality, the cost
function could also be
 treated  as a parameter, we shall  for simplicity consider  only the uniform  cost
$\costu$ given by \eqref{eq:def-C} in  this section.  
(The interested
reader  can use   Remark \ref{rem:costa-costu}  for  a first  generalization to  the
affine cost function given by~\eqref{eq:def-C0}.)

\subsection{Motivation}

The aim of this section is to provide examples of different set of
parameters for which two kernel or SIS models are ``equivalent'', in the
intuitive sense that their Pareto frontiers 
are the same (as subsets of $\R_+^2$), and it is possible to map nicely the Pareto optima from
one model to the another.
In Section~\ref{sec:exple-couple}, we present an example where  discrete
models can be represented as a continuous models and an example  based on
measure preserving transformation in the spirit of the graphon theory. 
We shall  consider the two families of models:
\begin{itemize}
 \item \textbf{\emph{the kernel model}} characterized by
 $\param=[(\Omega, \cf,\mu), \kk]$, with Assumption \ref{hyp:k}
 fulfilled, and loss function $\loss=R_e$;
 \item \textbf{\emph{the SIS model}} characterized by
 $\param=[(\Omega, \cf,\mu), k, \gamma]$, with Assumption \ref{hyp:k-g}
 fulfilled, and loss function $\loss\in \{R_e, \I\}$;
\end{itemize}
where $(\Omega, \cf, \mu)$ is a probability space, $\kk$ and $k$ are
non-negative kernels on $\Omega$ and $\gamma$ is a non-negative function on
$\Omega$. 

\medskip

In order to emphasize the dependence of a quantity $H$ on the parameters
$\param$ of the model, we shall write $H[\param]$ for $H$. For example
we write: $\Delta[\param]$ for the set of functions
$\{\eta\in \cl(\Omega, \cf)\,\colon\, 1\geq \eta\geq 0\}$, which clearly
depends on the parameters $\param$; and the effective reproduction
function $R_e[\param]$. For example, under Assumption \ref{hyp:k-g}, we
have the equality of the following functions:
$R_e[(\Omega, \cf, \mu), k, \gamma]=R_e[(\Omega, \cf, \mu),k/\gamma, 1]=
R_e[(\Omega, \cf, \mu), k/\gamma]$, where for the last equality the left
hand-side refers to the SIS model and the right hand-side refers to the
kernel model (where Assumption \ref{hyp:k} holds as a consequence of
Assumption \ref{hyp:k-g}). Using~\eqref{eq:r(AB)} if $\inf \gamma>0$
(see \cite[Section~3]{ddz-Re} for details and more general results), then we
also have
$ R_e[(\Omega, \cf, \mu), k/\gamma] = R_e[(\Omega, \cf, \mu),
\gamma^{-1} k]$.

\subsection{On measurability}

Let us recall some well-known facts on measurability. Let~$(E, \ce)$ and~$(E', \ce')$ be
two measurable spaces. If~$E'=\R$, then we take~$\ce'=\mathcal{B}(\R)$ the
Borel~$\sigma$-field. Let~$f$ be a function from~$E$ to~$E'$. We denote by
$\sigma(f)=\{f^{-1}(A)\, \colon \, A\in \ce'\}$ the~$\sigma$-field generated by~$f$. In
particular~$f$ is measurable from~$(E, \ce)$ to $(E', \ce')$ if and only
if~$\sigma(f)\subset \ce$. Let~$\varphi$ be a measurable function from~$(E, \ce)$ to~$(E',
\ce')$. For~$\nu$ a measure on~$(E,
\ce)$, we write~$\varphi_\# \nu$ for the for the push-forward measure on~$(E',
\ce')$ of the measure~$\nu$ by the function~$\varphi$ (that is~$\varphi_\#
\nu(A)=\nu(\varphi^{-1}(A))$ for all~$A\in \ce'$). By definition of ~$\varphi_\# \nu$, for
a non-negative  measurable function~$g$ defined from~$(E', \ce')$ to~$(\R, \mathcal{B}(\R))$, we have:
\begin{equation}\label{eq:push}
  \int_{E'} g \, \mathrm{d} \varphi_\#\nu= \int _{E} g\circ
  \varphi \, \mathrm{d} \nu.
\end{equation}

Let~$f$ be a measurable function from~$(E, \ce)$ to~$(\R,
\mathcal{B}(\R))$. We recall that:
\begin{equation}
 \label{eq:def-f-phi0}
 \sigma(f)\subset \sigma(\varphi)\,
 \Longrightarrow\, f=g\circ \varphi,
\end{equation}
for some measurable function $g$ from~$(E', \ce')$ to~$(\R, \mathcal{B}(\R))$.

The random variables we consider are defined on a probability space, say
$(\Omega_0, \cf_0, \P)$.

\subsection{Coupled models}

We  refer   the  reader  to  \cite{janson2010graphons}   for  a  similar
development in  the graphon setting.  We first define coupled  models in
the next  definition and state  in Proposition~\ref{prop:coupling-optim}
that  coupled models  have  related (anti-)Pareto  optima  and the  same
(anti-)Pareto frontiers.

In the kernel  model, we consider  the models
$\param_i=[(\Omega_i, \cf_i,  \mu_i), \kk_i]$ for $i\in  \{1,2\}$, where
Assumption  \ref{hyp:k} holds  for  each  model; in  the  SIS model,  we
consider the models $\param_i=[(\Omega_i, \cf_i, \mu_i), k_i, \gamma_i]$
for  $i\in  \{1,2\}$,  where  Assumption \ref{hyp:k-g}  holds  for  each
model. In what follows, we simply  write $\Delta_i$ the set of functions
$\Delta$ for the model $\param_i$.

A measure $\pi$ on $(\Omega_1\times \Omega_2, \cf_1\otimes \cf_2)$ is a \emph{coupling} if
its marginals are $\mu_1$ and $\mu_2$.

\begin{definition}[Coupled models]\label{def:couple}
  The models $\param_1$ and $\param_2$ are \emph{coupled} if there exists two independent
  $\Omega_1\times \Omega_2$-valued random vectors $(X_1, X_2)$ and $(Y_1, Y_2)$ (defined
  on a probability space $(\Omega_0, \cf_0, \P)$) with the same distribution given by a
  coupling (\emph{i.e.} $X_i$ and $Y_i$ have distribution $\mu_i$) such that, $\P$-almost surely:
  \begin{align*}
    \text{Kernel model:}\quad
 &\kk_1(X_1,Y_1) = \kk_2(X_2,Y_2),\\
 \text{SIS model:}\quad
 & \gamma_1(X_1) = \gamma_2(X_2) \quad\text{and}\quad k_1(X_1,Y_1) = k_2(X_2,Y_2).
  \end{align*}
  In this case, two real-valued measurable functions $v_1$ and $v_2$ defined respectively
  on $\Omega_1$ and $\Omega_2$ are \emph{coupled} (through $V$) if there exists a
  real-valued $\sigma(X_1, X_2)$-measurable integrable random variable $V$ such that
  $\P$-almost surely:
  \[ \esp{V|\, X_i}=v_i(X_i) \quad\text{for $i\in \{1, 2\}$}. \]
\end{definition}

\begin{remark}\label{rem:couplage}
  We keep notation from Definition \ref{def:couple}
  \begin{propenum}
  \item\label{item:V-X1X2} Since $V$ is real-valued and $\sigma(X_1, X_2)$-measurable, we
    deduce from \eqref{eq:def-f-phi0} that there exits a measurable function $v$ defined
    on $\Omega_1\times \Omega_2$ such that $V=v(X_1, X_2)$, thus the following equality
    holds $\P$-almost surely: 
    \[
      \esp{v(Y_1, Y_2)|\, Y_i}=v_i(Y_i) \quad\text{for $i\in \{1, 2\}$}.
    \]
  \item\label{item:V1} If $W$ is a real-valued integrable $\sigma(X_1) \cap
    \sigma(X_2)$-measurable random variable, then setting $v_i(X_i)=\esp{W|X_i} = W$, the
    equality $v_1(X_1) = v_2(X_2)$ holds almost surely, and we get that $v_1$ and $v_2$ are
    coupled (through $W$).
  \item \label{item:V2} Let $\eta_1\in \Delta_1$. According to \eqref{eq:def-f-phi0}, there
    exists $\eta_2\in\Delta_2$ such that $\esp{\eta_1(X_1)|X_2} = \eta_2(X_2)$. Thus, by
    definition $\eta_1$ and $\eta_2$ are coupled (through $V = \eta_1(X_1)$).
  \end{propenum}
\end{remark}

The main result of this section, whose proof is given in Section \ref{sec:proof-coupling},
states that coupled models have coupled Pareto optimal strategies, and thus the same
(anti-)Pareto frontier.

\begin{proposition}[Coupling and Pareto optimality]
  \label{prop:coupling-optim}
  Let $\param_1$ and $\param_2$ be two coupled (kernel or SIS) models with the uniform
  cost function $C=\costu$ and loss function $\loss$ (with $\loss=R_e$ in the kernel model
  and $\loss\in \{R_e, \I\}$ in the SIS model). If the functions $\eta_1\in \Delta_1$ and
  $\eta_2\in \Delta_2$ are coupled, then: \[ \text{$\eta_1$ is Pareto optimal (for
      $\param_1$)} \, \Longleftrightarrow\, \text{$\eta_2$ is Pareto optimal (for
  $\param_2$)}. \]
  Furthermore, if $\eta_1\in \Delta_1$ is Pareto optimal (for $\param_1$), then there
  exists a Pareto optimal (for $\param_2$) strategy $\eta_2\in \Delta_2$ such that
  $\eta_1$ and $\eta_2$ are coupled. In particular, the (anti-)Pareto
  frontiers  are the same for the two models $\param_1$ and $\param_2$.
\end{proposition}

The next Corollary  is useful for model reduction,  which corresponds to
merging individuals with identical  behavior, see the examples
in       Sections~\ref{sec:dis-cont}       and~\ref{sec:exple-esp-cond}.
Equation~\eqref{eq:P0-h-Eh} below  could also be stated  for anti-Pareto
optima; and the adaptation to the kernel  model is immediate.
\begin{corollary}\label{cor:couplage}
  Let $\param=[(\Omega,\cf,\P),k,\gamma]$ be a SIS model with the uniform cost function
  $C=\costu$ and loss function $\loss\in \{R_e, \I\}$. Let $\cg\subset\cf$ be a
  $\sigma$-field such that $\gamma$ is $\cg$-measurable and $k$ is
  $\cg\otimes\cg$-measurable. Then, for any $\eta\in \Delta[\param]$, we have:
  \begin{equation}
    \label{eq:P0-h-Eh} \text{$\eta$ is Pareto optimal} \, \Longleftrightarrow\,
    \text{$\esp{\eta|\, \cg}$ is Pareto optimal}.
  \end{equation}
\end{corollary}

\begin{proof}
  Let $\Omega_0=\Omega^2$  endowed with  the product  $\sigma$-field and
  the product  probability measure  $\P_0$, and $X$  (resp. $Y$)  be the
  projection on  the first  (resp. second)  coordinate. Thus  the random
  variables  $X$ and  $Y$  are  independent, $(\Omega,\cf)$-valued  with
  distribution $\P$.   Write $(X', Y')$  for $(X,Y)$ when  considered as
  $(\Omega,\cg)$-valued random variables. Notice  that $X'$ and $Y'$ are
  by construction  independent with  distribution $\P'$, where  $\P'$ is
  the restriction of $\P$ to  $\cg$. As $\gamma$ is $\cg$-measurable and
  $k$  is  $\cg\otimes  \cg$-measurable,   we  can  consider  the  model
  $\param'=[(\Omega, \cg, \P'), k, \gamma]$. Then $(X,X')$ and $(Y, Y')$
  are   two   trivial   couplings  such   that   $k(X,Y)=k(X',Y')$   and
  $\gamma(X)=\gamma(X')$.  Thus the  models $\param$  and $\param'$  are
  coupled.       We have that        $\eta\in         \Delta$         and
  $\eta'=\E[\eta|\, \cg]\in  \Delta'$ are coupled through  $\eta\circ X$
  since  $\E_0[\eta\circ      X|\,\sigma(X)]=\eta\circ      X$      and
  $\E_0[\eta\circ        X|\,\sigma(X')]=\eta'\circ        X'$        as
  $\sigma(X')=X^{-1}(\cg)$ and $X=X'$ can be seen as the identity map on
  $\Omega$.       The       conclusion      then       follows      from
  Proposition~\ref{prop:coupling-optim}.
\end{proof}

\subsection{Examples of couplings}
\label{sec:exple-couple}
In this section, we consider the SIS model as the kernel model can be handled in the same
way. We denote by $\mathrm{Leb}$ the Lebesgue measure.

\subsubsection{Discrete and continuous models}
\label{sec:dis-cont}

We now formalize how finite population models can be seen as particular cases of models
with a continuous population. Let $\Omega_\mathrm{d}\subset \N$,~$\cf_\mathrm{d}$ the set
of subsets of~$\Omega_\mathrm{d}$ and~$\mu_\mathrm{d}$ a probability measure
on~$\Omega_\mathrm{d}$. Without loss of generality, we can assume that
$\mu_\mathrm{d}(\{\ell\})>0$ for all $\ell\in \Omega_\mathrm{d}$. We
set~$\Omega_\mathrm{c}=[0, 1)$, with~$\cf_\mathrm{c}$ its Borel~$\sigma$-field
and~$\mu_\mathrm{c}=\mathrm{Leb}$. Let $(B_\ell, \ell\in\Omega_\mathrm{d})$ be a partition
of $[0,1)$ in measurable sets such that $\mathrm{Leb}(B_\ell) = \mu_\mathrm{d}(\{\ell\})$
for all $\ell\in \Omega_\mathrm{d}$. The measure $\pi$ on $\Omega_\mathrm{d}\times
\Omega_\mathrm{c}$ uniquely defined by:
\[
  \pi( \{\ell\} \times A) = \mathrm{Leb}(B_\ell \cap A)
\]
for all measurable $A\subset[0,1)$ and $\ell\in \Omega_\mathrm{d}$ is clearly a coupling
of $\mu_\mathrm{d}$ and $\mu_\mathrm{c}$. If the kernels $k_\mathrm{d}$ on
$\Omega_\mathrm{d}$ and $k_\mathrm{c}$ on $\Omega_\mathrm{c}$ and the functions
$\gamma_\mathrm{d}$ and $\gamma_\mathrm{c}$ are related through the formula:
\[
  \gamma_\mathrm{c}(x)=\gamma_\mathrm{d}(\ell)
  \quad\text{and}\quad
  k_\mathrm{c}(x,y) = k_\mathrm{d}(\ell,j), \quad \text{for $x\in
    B_\ell$, $y \in B_j$ and $\ell, j\in \Omega_\mathrm{d}$},
\]
then the discrete model~$\param_\mathrm{d}=[(\Omega_\mathrm{d}, \cf_\mathrm{d},
\mu_\mathrm{d}), k_\mathrm{d}, \gamma_\mathrm{d} ]$ and the continuous
model~$\param_\mathrm{c}=[([0,1), \cf_\mathrm{c},\mathrm{Leb}), k_\mathrm{c},
\gamma_\mathrm{c}]$ are coupled. Roughly speaking, we can blow up the atomic part of the
measure~$\mu_\mathrm{d}$ into a continuous part, or, conversely, merge all points that
behave similarly for $k_\mathrm{c}$ and $\gamma_\mathrm{c}$ into an atom, without altering
the Pareto frontier.

\begin{example}\label{ex:sbm2}
  We consider the  so called stochastic block model, with
  2  populations for simplicity, in  the setting  of  the SIS  model, and  give in  this
  elementary case the corresponding  discrete and  continuous models. Then, we explicit
  the  relation  with the  formalism  of  the  same model  developed  in
  \cite{lajmanovich1976deterministic}   by    Lajmanovich   and   Yorke.
  \medskip

  The discrete SIS model is defined on $\Omega_\mathrm{d}=\{1,2\}$ with the probability
  measure $\mu_\mathrm{d}$ defined by $\mu_\mathrm{d}(\{1\})=1-\mu_\mathrm{d}(\{2\})=p$
  with $p\in (0,1)$, and a kernel~$k_\mathrm{d}$ and recovery function~$\gamma_\mathrm{d}$
  given by the matrix and the vector:
  \[
    k_\mathrm{d}=\begin{pmatrix} k_{11} & k_{12}\\
    k_{21} & k_{22}
  \end{pmatrix}
  \quad\text{and}\quad
  \gamma_\mathrm{d}=\begin{pmatrix} \gamma_1\\ \gamma_2 \end{pmatrix}.
  \]
  Notice $p$  is the  relative size of  population 1.  The corresponding
  discrete                            model                           is
  $\param_\mathrm{d}                                                   =
  [(\{1,2\},\cf_\mathrm{d},\mu_\mathrm{d}),k_\mathrm{d},\gamma_\mathrm{d}]$;
  see Figure~\ref{fig:discrete-model}. \medskip

  The continuous model is defined on the state space~$\Omega_\mathrm{c}=[0, 1)$ is endowed
  with its Borel~$\sigma$-field,~$\cf_\mathrm{c}$, and the Lebesgue
  measure~$\mu_\mathrm{c}=\mathrm{Leb}$. The segment~$[0,1)$ is partitioned into two
  intervals~$B_1=[0,p)$ and~$B_2=[p, 1)$, the transmission kernel~$k_\mathrm{c}$ and
  recovery rate $\gamma_\mathrm{c}$ are given by:
  \[
    k_\mathrm{c}(x,y) = k_{ij}
    \quad\text{and}\quad
    \gamma_\mathrm{c}(x)= \gamma_i
    \quad \text{for $x\in B_i$, $y\in B_j$, and  $i,j\in \{1, 2\}$}.
  \]
  The corresponding continuous model is $\param_\mathrm{c} = [([0,
  1),\cf_\mathrm{c},\mathrm{Leb}),k_\mathrm{c},\gamma_\mathrm{c}]$; see
  Figure~\ref{fig:continuous-model}. By the general discussion above, the discrete and
  continuous models are coupled, and in particular they have the same Pareto and
  anti-Pareto frontiers.

  Furthermore, in this simple example, it is easily checked that a discrete vaccination
  $\eta_\mathrm{d}=(\eta_1,\eta_2)$ and a continuous vaccination
  $\eta_\mathrm{c}=(\eta_\mathrm{c}(x), x\in [0, 1))$ are coupled if and only if there
  exists a function $\eta$ defined on $\Omega_\mathrm{c}\times
  \Omega_\mathrm{d}=[0,1)\times\{1,2\}$ such that:
  \[
    \left\{
      \begin{aligned}
	\eta_i &= \frac{1}{\mathrm{Leb}(B_i)}\int_{B_i} \eta(x,i)
	\, \mathrm{d}x, \quad i\in \{1,2\}, \\
	\eta_c(x) &= \eta(x,1)\mathds{1}_{B_1}(x) + \eta(x,2)\mathds{1}_{B_2}(x),\quad \mathrm{Leb}\text{-a.s.,}
      \end{aligned}
    \right.
  \]
  which occurs if and only if:
  \[
    \eta_i = \frac{1}{\mathrm{Leb}(B_i)}\int \eta_c(x) \mathds{1}_{B_i}(x)
    \, \mathrm{d} x, \quad i\in \{1,2\}.
  \]
  Therefore, in this case, the optimal strategies of the continuous model are easily
  deduced from the optimal strategies of the discrete model.

  \medskip

  To conclude this example, we rewrite, using the formalism of the discrete model
$\param_\mathrm{d}$,  the next-generation matrix~$K$ in the setting of
  \cite{lajmanovich1976deterministic}, and the effective next-generation matrix~$K_e(\eta)$ when the
  vaccination strategy~$\eta$ is in force (recall~$\eta_i$ is the proportion of population
  with feature~$i$ which is not vaccinated):
  \[
    K=
    \begin{pmatrix}
      \kk_{11} \,p\, & \kk_{12} \,(1-p) \\
      \kk_{21}\, p\, & \kk_{22} \,(1-p)
    \end{pmatrix}
    \quad\text{and}\quad
    K_e(\eta)
    = \begin{pmatrix}
      \kk_{11} \,p\, \eta_1 & \kk_{12} \,(1-p)\, \eta_2 \\
      \kk_{21}\, p\, \eta_1 & \kk_{22} \,(1-p)\, \eta_2
    \end{pmatrix},
  \]
  with $p=\mu_\mathrm{d}(\{1\})$, $1-p=\mu_\mathrm{d}(\{2\})$ and
  $\kk_\mathrm{d}=k_\mathrm{d}/\gamma_\mathrm{d}$, that is:
  \[
    \kk_\mathrm{d}= \begin{pmatrix}
      \kk_{11} & \kk_{12} \\
      \kk_{21} & \kk_{22}
    \end{pmatrix}
    =
    \begin{pmatrix}
      k_{11}/\gamma_1\, & k_{12}/\gamma_2 \\
      k_{21}/\gamma_1\, & k_{22}/\gamma_2
    \end{pmatrix}.
  \]
\end{example}

\begin{figure}
 \savebox{\largestimage}{\includegraphics[page=1]{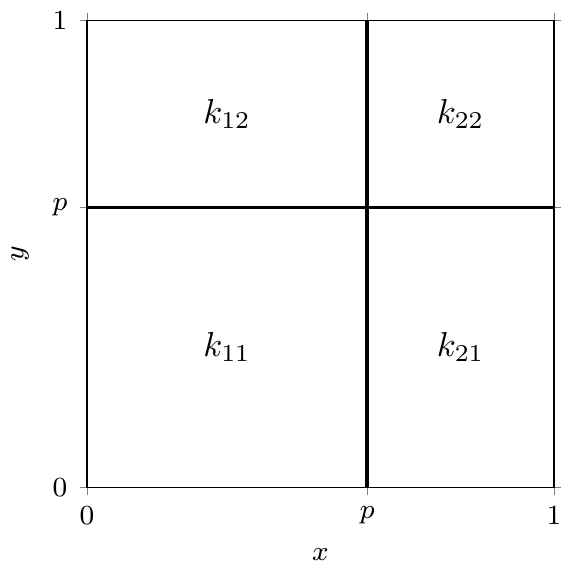}}
 \begin{subfigure}[b]{0.48\textwidth} \centering
 \usebox{\largestimage}
 \caption{Continuous model: kernel~$k_\mathrm{c}$ on~$\Omega_\mathrm{c}=[0,1)$
 with the Lebesgue measure.}
 \label{fig:continuous-model}
 \end{subfigure}
 \quad
 \begin{subfigure}[b]{0.48\textwidth}
 \centering
 \raisebox{\dimexpr.5\ht\largestimage-.5\height}{%
 \includegraphics[page=2]{equivalent}}
 \caption{Discrete model: kernel~$k_\mathrm{d}$ on
 $\Omega_\mathrm{d}=\{1, 2\}$ with the measure
 $p\delta_1+ (1-p)\delta_2$.}
 \label{fig:discrete-model}
 \end{subfigure}
 \caption{Coupled continuous model (left) and discrete
   model (right).}
 \label{fig:discrete-and-continuous}
\end{figure}

\subsubsection{Measure preserving function}

This section is motivated by the theory of graphons, which are indistinguishable by measure
preserving transformation, see \cite[Sections~7.3 and~10.7]{lovasz_large_2012}.
Let~$(\Omega, \cf, \mu)$ be a measurable space. We say a measurable function~$\varphi$
from $(\Omega, \cf)$ to itself is \emph{measure preserving} if~$\mu= \varphi _\# \mu$. For
example the function~$\varphi:x\mapsto 2x
\mod(1)$ defined on the probability space~$([0, 1], \cb([0, 1],
\mathrm{Leb})$ is measure
preserving.

Let~$\varphi$ be measure preserving function on $\Omega$. Let $k_1$ be a kernel and
$\gamma_1$ a function on
$\Omega$ such that the model $\param_1=[(\Omega, \cf, \mu), k_1,
\gamma_1]$ satisfies Assumption \ref{hyp:k-g}. Let $X_1$ be a random
variable  with probability distribution $\mu$ and let $X_2
=\varphi(X_1)$, so that $(X_1, X_2)$ is a coupling of $(\Omega, \cf, \mu)$
with itself. Then for the kernel $k_2$ and the function $\gamma_2$
defined by:
\[
  k_2(x,y)= k_1(\varphi(x),\varphi(y))
  \quad\text{and}\quad
  \gamma_2(x)=\gamma_1 (\varphi(x)),
\]
the models $\param_1$ and~$\param_2=[(\Omega, \cf, \mu), k_2, \gamma_2]$ are coupled.
Roughly speaking, we can give different labels to the features of the population without
altering the Pareto and anti-Pareto frontiers.

\subsubsection{Model reduction using deterministic coupling}
\label{sec:exple-esp-cond}
This example is in the spirit of  Section~\ref{sec:dis-cont}, where one
merges individual with identical behavior. We consider a SIS model 
$\param_1=[(\Omega_1,\cf_1,\mu_1),k_1,\gamma_1]$. Let $\varphi$ be a
  measurable function from $(\Omega_1, \cf_1)$ to $(\Omega_2, \cf_2)$. Assume that:
  \[
    \sigma(\gamma_1)\subset \sigma(\varphi) \quad\text{and}\quad \sigma(k_1)\subset
    \sigma(\varphi)\otimes \sigma(\varphi).
  \]

  We can then build an elementary coupling. Let $X_1$ and $Y_1$ be independent $\mu_1$
  distributed random elements of $\Omega_1$, and set $(X_2,Y_2) = (\varphi(X_1),
  \varphi(Y_1))$. Since $\sigma(\gamma_1)\subset \sigma(\varphi)$ and $ \sigma(k_1)\subset
  \sigma(\varphi)\otimes \sigma(\varphi)$, we get that $\gamma_1(X_1)$ is
  $\sigma(X_2)$-measurable and $k_1(X_1,Y_1)$ is $\sigma(X_2,Y_2)$-measurable. According
  to \eqref{eq:def-f-phi0}, there exists two measurable functions $\gamma_2:\Omega_2\to\R$
  and $k_2:\Omega_2\times\Omega_2\to\R$ such that $\gamma_1=\gamma_2 \circ \varphi$ and
  $k_1=k_2(\varphi\otimes \varphi)$ that is almost surely:
  \[
    \gamma_1(X_1)=\gamma_2(X_2) \quad\text{and}\quad k_1(X_1,Y_1) = k_2(X_2,Y_2).
  \]
  Let $\mu_2=\varphi_\# \mu_1$ be the push-forward measure of $\mu_1$ by $\varphi$. Using
  \eqref{eq:push} it is easy to check that the integrability condition from Assumption
  \ref{hyp:k-g} is fulfilled, so we can consider the reduced model
  $\param_2=[(\Omega_2,\cf_2,\mu_2),k_2,\gamma_2]$. By Definition \ref{def:couple},
  $\param_1$ is coupled with $\param_2 $ through the (deterministic) coupling $\pi$ given
  by the distribution of $(X_1, \varphi(X_1))$. \medskip

  
  Eventually, we get from Corollary \ref{cor:couplage}
  with $\cg=\sigma(\varphi)$, that $\eta_1\in \Delta_1$ is Pareto optimal if and only if
  $\E_1[\eta_1|\,\varphi]$ is Pareto optimal (for the model $\param_1$), where $\E_1$
  correspond to the expectation with respect to the probability measure
  $\mu_1$ on $(\Omega_1, \cf_1)$.

  \section{Technical proofs}
  \label{sec:technical_proofs}
\subsection{The SIS model: properties of~$\I$ and of the maximal
  equilibrium}
\label{sec:proof-I}
We prove here Theorem~\ref{th:continuity-I} and Proposition~\ref{prop:I-stab}, and
properties of the maximal equilibrium.
For the convenience of the reader, we only use references to the results recalled in
\cite{delmas_infinite-dimensional_2020} for positive operators on Banach spaces. For an
operator $A$, we denote by $A^\top$ its adjoint. We first give a preliminary
lemma.

\begin{lemma}\label{lem:prelim-result}
  Suppose Assumption~\ref{hyp:k-g} holds, and consider the positive bounded linear
  integral operator $\Tinf_{k/\gamma}$ on $\cl$. If there exists $g\in
  \mathscr{L}^\infty_+ $, with $\int_\Omega g \, \mathrm{d} \mu >0$ and $\lambda > 0$
  satisfying:
  \[
    \Tinf_{k/\gamma}(g)(x) > \lambda g(x), \quad \text{for all} \quad x \text{ such that }
  g(x)>0, \]
  then we have $\rho(\Tinf_{k/\gamma})>\lambda$.
\end{lemma}

\begin{proof}
  Set $\Tinf=\Tinf_{k/\gamma}$. Let $A =\set{ g>0}$ be the support of the function $g$.
  Let $\Tinf'$ be the bounded operator defined by $\Tinf'(f) =\mathds{1}_A
  \Tinf(\mathds{1}_A f)$. Since $\Tinf'(g)= \mathds{1}_A
  \Tinf(\mathds{1}_A g) = \mathds{1}_A \Tinf(g) > \lambda g$, we
  deduce from the Collatz-Wielandt formula, see
  \cite[Proposition~3.6]{delmas_infinite-dimensional_2020}, that $\rho(\Tinf')
  \geq \lambda>0$. According to \cite[Lemma~3.7~(v)]{delmas_infinite-dimensional_2020},
  there exists $v\in L^q_+ \setminus \{ 0 \}$, seen as an element of the topological dual
  of $\cl$, a left Perron eigenfunction of $\Tinf'$, that is such that $
  (\Tinf')^\top(v) =\rho(\Tinf') v$. In particular, we have $v= \mathds{1}_A
  \, v$ and thus $\int_A v \, \mathrm{d}\mu > 0$ and $\int_\Omega vg\, \mathrm{d}\mu >
  0$. We obtain:
  \[
    (\rho(\Tinf') - \lambda) \braket{v,g} = \braket{v, \Tinf'(g) -
    \lambda g} > 0.
  \]
  This implies that $\rho(\Tinf') > \lambda$. Since $\Tinf- \Tinf'$ is a positive
  operator, we deduce from \eqref{eq:r(A)r(B)} that $\rho(\Tinf) \geq \rho(\Tinf') >
  \lambda$.
\end{proof}

We now state an interesting result on the characterization of the maximal
equilibrium~$\mathfrak{g}$. We keep notations from Sections \ref{sec:dyn} and
\ref{sec:vacc} and write $R_e$ for $R_e[k/\gamma]$. Recall that $R_0=R_e(\un)$ and $F$
defined by \eqref{eq:vec-field}. Let $DF[h]$ denote the bounded linear operator on $\cl$
of the derivative of the map $f \mapsto F(f)$ defined on $\cl$ at point $h$:
\[
  DF[h](g) = (1-h) \Tinf_k(g) - (\gamma + \Tinf_{k}(h))g
  \quad \text{for $h,g\in \cl$.}
\]

Let $s(A)$ denote the spectral bound of the bounded operator $A$, see (33) in
\cite{delmas_infinite-dimensional_2020}.

\begin{proposition}\label{prop:caract-g}
  Suppose Assumption~\ref{hyp:k-g} holds and write $R_e$ for $R_e[k/\gamma]$. Let $h$ in
  $\Delta$ be an equilibrium, that is $F(h) = 0$. The following properties are equivalent:
  \begin{propenum}
  \item\label{h=g} $h=\mathfrak{g}$,
  \item\label{s(g)<0} $s(DF[h]) \leq 0$,
  \item\label{Rh2<1} $R_e((1-h)^2)\leq 1$.
  \item\label{Rh<1} $R_e(1-h)\leq 1$.
  \end{propenum}
  We also have: $\mathfrak{g}=0 \Longleftrightarrow R_0\leq 1$; as well as:
  $\mathfrak{g}\neq 0\Longrightarrow R_e(1-\mathfrak{g})=1$.
\end{proposition}

\begin{proof}
  Let $h \in \Delta$ be an equilibrium, that is $F(h)=0$.

  \medskip

  Let us show the equivalence between \ref{s(g)<0} and \ref{Rh2<1}. According to
  \cite[Proposition~4.2]{delmas_infinite-dimensional_2020}, $s(DF[ h]) \leq 0$ if and only
  if:
  \[
    \rho\left(\Tinf_\kk\right) \leq 1 \quad \text{with} \quad
    \kk(x,y) = (1 - h(x))\frac{ k(x,y) }{\gamma(y) + \Tinf_k(h)(y)}\cdot
  \]
  Since $F(h)=0$, we have $(1-h)/\gamma= 1/(\gamma+ \Tinf_k(h))$. This gives:
  \begin{equation}\label{eq:Tk-Fh}
    \mathsf{k}(x,y) = (1 - h(x))\frac{ k(x,y) (1-h(y))}{\gamma(y)}
  \end{equation}
  and thus $\Tinf_\kk = M_{1-h}\, \Tinf_{k/\gamma} \, M_{1-h}$, where $M_f$ is the
  multiplication operator by $f$. Recall the definition~\eqref{eq:def-R_e} of $R_e$.
  According to~\eqref{eq:r(AB)}, we have:
  \begin{equation}\label{eq:commute} \rho\left(\Tinf_\kk\right) =
    \rho\left(\Tinf_{k/\gamma} M_{(1-h)^2} \right) = R_e((1-h)^2).
  \end{equation}
  This gives the equivalence between \ref{s(g)<0} and \ref{Rh2<1}.

  \medskip

  We prove that \ref{h=g} implies \ref{Rh<1}. Suppose that $R_e(1-h)>1$. Thanks
  to~\eqref{eq:r(AB)}, we have $\rho(M_{1-h} \Tinf_{k/\gamma})=
  \rho(\Tinf_{k/\gamma}M_{1-h})= R_e(1-h)>1$. According to
  \cite[Lemma~3.7~(v)]{delmas_infinite-dimensional_2020}, there exists $v\in L^q_+
  \setminus \{ 0 \}$ a left Perron eigenfunction of $\Tinf_{(1-h)k/\gamma}$, that is $
  \Tinf_{(1-h)k/\gamma}^\top(v) = R_e(1-h) v$. Using $F(h)=0$, and thus
  $(1-h) \Tinf_k(h)=\gamma h$, for the last equality, we have:
  \begin{equation*}
    R_e(1-h) \braket{v,\gamma h} = \braket{v,
    (1-h)\Tinf_{k/\gamma}(\gamma h)} = \braket{v, \gamma h}.
  \end{equation*}
  We get $\braket{v,\gamma h}=0 $ and thus $\braket{v,\mathds{1}_A} = 0$, where
  $A=\set{h>0}$ denote the support of the function $h$. Since $
  \Tinf_{(1-h)k/\gamma}^\top(v) =
  R_e(1-h) v$ and setting $v'=(1-h)v$ (so that $v'=v$ $\mu$-almost surely on $A^c$), we
  deduce that:
  \begin{equation*}
    \Tinf_{k'/\gamma}^\top(v') = R_e(1-h) v',
  \end{equation*}
  where $k'=\mathds{1}_{A^c}\, k\, \mathds{1}_{A^c}$. This implies that $\rho(
  \Tinf_{k'/\gamma})\geq R_e(1-h)$. Since $k'=(1-h)k'$ and $\Tinf_{k/\gamma}-
  \Tinf_{k'/\gamma}$ is a positive operator as $k-k'\geq 0$, we get,
  using~\eqref{eq:r(A)r(B)} for the inequality, that $ \rho(\Tinf_{k'/\gamma})= \rho (
  M_{1-h} \Tinf_{k'/\gamma}) \leq \rho ( M_{1-h}\Tinf_{k/\gamma})= R_e(1-h)$. Thus, the
  spectral radius of $\Tinf_{k'/\gamma}$ is equal to $R_e(1-h)$. According to
  \cite[Proposition~4.2]{delmas_infinite-dimensional_2020}, since
  $\rho(\Tinf_{k'/\gamma})>1$, there exists $w \in \mathscr{L}^\infty_+\setminus \{ 0 \}$
  and $\lambda>0$ such that:
  \begin{equation*}
    \Tinf_{k'}(w) - \gamma w = \lambda w.
  \end{equation*}
  This also implies that $w=0$ on $A=\set{h>0}$, that is $wh=0$ and thus $w \Tinf_k(h)=0$
  as $\Tinf_k (h) =\gamma h/(1-h)$. Using that $F(h)=0$,
  $\Tinf_k(w)=\Tinf_{k'}(w)=(\gamma+\lambda) w$ and $h\Tinf_k(w)=0$, we obtain:
  \begin{equation*}
    F(h+w) = w(\lambda - \Tinf_k(w)) .
  \end{equation*}
  Taking $\varepsilon>0$ small enough so that $\varepsilon \Tinf_{k }(w) \leq \lambda/2$
  and $\varepsilon w \leq 1$, we get $h+\varepsilon w\in
  \Delta$ and $ F( h +\varepsilon
  w) \geq 0$. Then use
  Lemma \ref{lem:Fh>0} to deduce that $h+\varepsilon w \leq \mathfrak{g}$ and thus
  $h \neq \mathfrak{g} $.

  \medskip

  To see that \ref{Rh<1} implies \ref{Rh2<1}, notice that $(1-h)^2\leq (1-h)$, and then
  deduce from Proposition~\ref{prop:R_e}~\ref{prop:increase} that $R_e((1-h)^2) \leq
  R_e(1-h)$.

  \medskip

  We prove that \ref{Rh2<1} implies \ref{h=g}. Notice that $F(g)=0$ and $g\in \Delta$
  implies that $g<1$. Assume that $h \neq \mathfrak{g}$. Notice that $\gamma/(1-h)= \gamma
  + \Tinf_{k}(h)$, so that $\gamma (\mathfrak{g} - h)/(1 - h) \in \cl_+$. An elementary
  computation, using $F(h)=F(\mathfrak{g})=0$ and $\kk$ defined in~\eqref{eq:Tk-Fh},
  gives:
  \[
    \Tinf_\kk \left( \gamma \frac{\mathfrak{g} - h}{1 - h}\right)
    = (1-h)\Tinf_{k}\left(\mathfrak{g} - h \right)
    = \gamma
    \frac{\mathfrak{g} - h}{1 - \mathfrak{g}}
    =\frac{1-h}{1 - \mathfrak{g}} \, \gamma
    \frac{\mathfrak{g} - h}{1 - h} \cdot
  \]
  Since $h \neq \mathfrak{g}$ and $h\leq \mathfrak{g}$, we deduce that $(1-h)/ (1 -
  \mathfrak{g})\geq 1$, with strict inequality on $\set{ \mathfrak{g}-h>0}$ which is a set
  of positive measure. We deduce from Lemma~\ref{lem:prelim-result} (with $k$ replaced by
  $\kk \gamma$) that $\rho\left(\Tinf_{\mathsf{k}}\right)>1$. Then use~\eqref{eq:commute}
  to conclude.

  \medskip

  To conclude notice that $\mathfrak{g}=0 \Longleftrightarrow R_0\leq 1$ is a consequence
  of the equivalence between \ref{h=g} and \ref{Rh<1} with $h=0$ and $R_0=R_e(\un)$.

  Using that $F(\mathfrak{g})=0$, we get $\Tinf_k(\mathfrak{g})= \gamma
  \mathfrak{g}/(1-\mathfrak{g})$. We deduce that
  $\Tinf_{k(1-\mathfrak{g})/\gamma}(\Tinf_k(\mathfrak{g})) = \Tinf_k(\mathfrak{g})$. If
  $\mathfrak{g}\neq 0$, we get $\Tinf_k(\mathfrak{g})\neq 0$ (on a set of positive
  $\mu$-measure). This implies that $R_e(1-\mathfrak{g})\geq 1$. Then use \ref{Rh<1} to
  deduce that $R_e(1-\mathfrak{g})=1$ if $\mathfrak{g}\neq 0$.
\end{proof}

In the SIS model, in order to stress, if necessary, the dependence of a quantity $H$, such
as $F_\eta$, $R_e$ or $\mathfrak{g}_{\eta}$, in the parameters $k$ and $\gamma$ (which
satisfy Assumption~\ref{hyp:k-g}) of the model, we shall write $H[k, \gamma]$. Recall that
if $k$ and $\gamma$ satisfy Assumption~\ref{hyp:k-g}, then the kernel $k/\gamma$ has a
finite double norm on $L^p$ for some $p\in (1, +\infty )$. We now consider the continuity
property of the maps $\eta \mapsto \mathfrak{g}_{\eta}[k, \gamma]$ and $(k,\gamma,\eta)
\mapsto \mathfrak{g}_{\eta}[k, \gamma]$.

\begin{lemma}\label{lem:cvgn2}
  Let $((k_n, \gamma_n), n\in \N)$ and $(k,\gamma)$ be kernels and functions satisfying
  Assumption~\ref{hyp:k-g} and $(\eta_n, \, n \in \N)$ be a sequence of elements of
  $\Delta$ converging weakly to $\eta$.
  \begin{propenum}
  \item\label{lem:cvg2-h} We have $\lim_{n\rightarrow \infty }
    \mathfrak{g}_{\eta_n}[k,\gamma]=\mathfrak{g}_{\eta}[k, \gamma]$ $\mu$-almost surely.
  \item\label{lem:cvg2-kgh} Assume furthermore there exists $p'\in (1,
    +\infty )$ such that $\kk=\gamma^{-1} k$ and $(\kk_n=\gamma^{-1}_n k_n,
    n\in \N)$ have finite double norm on $L^{p'}$ and that $\lim_{n\rightarrow\infty }
    \norm{\kk_n -\kk}_{p',q'}=0$. Then, we have $\lim_{n\rightarrow
    \infty } \mathfrak{g}_{\eta_n}[k_n,\gamma_n]=\mathfrak{g}_{\eta}[k, \gamma]$
    $\mu$-almost surely.
  \end{propenum}
\end{lemma}

\begin{proof}
  The proof of \ref{lem:cvg2-h} and \ref{lem:cvg2-kgh} being rather similar, we only
  provide the latter and indicate the difference when necessary. To simplify, we write
  $g_n=\mathfrak{g}_{\eta_n}[k_n,\gamma_n]$. We set $h_n = \eta_ng_n \in \Delta$ for $n\in
  \N$. Since $\Delta$ is sequentially weakly compact, up to extracting a subsequence, we
  can assume that $h_n$ converges weakly to a limit $h\in \Delta$. Since $F_{\eta_n}[k_n,
  \gamma_n](g_n)=0$ for all $n\in\N$, see~\eqref{eq:F(g)=0}, we have:
  \begin{equation}\label{eq:continuity-I} g_n = \frac{\Tinf_{\kk_n}(\eta_n g_n)}{1+
    \Tinf_{\kk_n}(\eta_n g_n)}
    = \frac{\Tinf_{\kk_n} (h_n)}{1+ \Tinf_{\kk_n}(h_n)} \cdot
  \end{equation}
  We set $g=\Tinf_\kk(h)/(1 + \Tinf_\kk(h))$. Notice that $
  \Tinf_{\kk_n}(h_n)=(\Tinf_{\kk_n}-\Tinf_\kk)(h_n) + \Tinf_\kk(h_n)$. We have
  $\lim_{n\rightarrow\infty} \Tinf_\kk(h_n)=\Tinf_\kk(h)$ pointwise. Since
  $\norm{(\Tinf_{\kk_n}-\Tinf_\kk)(h_n)}_{p'}\leq \norm{\kk_n -\kk}_{p',q'}$, up to taking
  a sub-sequence, we deduce that $\lim_{n\rightarrow\infty}
  (\Tinf_{\kk_n}-\Tinf_\kk)(h_n)=0$ almost surely. (Notice the previous step is not used
  in the proof of \ref{lem:cvg2-h} as $\kk_n=\kk$ and $\lim_{n\rightarrow\infty}
  \Tinf_k(h_n)=\Tinf_k(h) $ pointwise.) This implies that $g_n$ converges almost surely to
  $g$. By the dominated convergence theorem, we deduce that $g_n$ converges also in $L^p$
  to $g$.
  This proves that $h=\eta g$  almost surely. We get
  $g=\Tinf_\kk(\eta g)/(1+ \Tinf_\kk(\eta g))$ and thus $F_{\eta}[k, \gamma](g)=0$: $g$ is
  an equilibrium for $F_\eta[k, \gamma]$. We recall from \cite[Section~3]{ddz-Re} the
  functional equality $R_e[k'h]=R_e[hk']$, where $k'$ is a kernel, $h$ a non-negative
  functions such that the kernels $k'h$ and $hk'$ have some finite double norm. We deduce
  from the weak-continuity and the stability of $R_e$, see Theorem \ref{th:continuity-R}
  and Proposition~\ref{prop:Re-stab}, that:
  \begin{align*}
    R_e[k/\gamma](\eta (1- g))= R_e[\kk](\eta(1-g)) &= \lim_{n\rightarrow\infty }
    R_e[\kk_n](\eta_n (1- g_n))\\
						    & = \lim_{n\rightarrow\infty }
						    R_e[k_n/\gamma_n](\eta_n (1- g_n))\\
						    &\leq 1.
  \end{align*}
  (Only the weak-continuity of $\eta'\mapsto R_e[k/\gamma](\eta')$ is used in the proof of
  \ref{lem:cvg2-h} to get $R_e[k/\gamma](\eta(1-g))\leq 1$.) We deduce that property
  \ref{Rh<1} of Proposition~\ref{prop:caract-g} holds with $k$ replaced by $k\eta$, and
  thus property~\ref{h=g} therein implies that $g =\mathfrak{g}_{\eta}[k,\gamma]$.
\end{proof}

\begin{proof}[Proofs of Theorem \ref{th:continuity-I} and Proposition \ref{prop:I-stab}]
  Under the assumptions of Lemma~\ref{lem:cvgn2}, taking the pair $(k_n,
  \gamma_n)$ equal to $(k,\gamma)$ in the case \ref{lem:cvg2-h} therein, we deduce that
  $(\eta_n\, \mathfrak{g}_{\eta_n}[k_n,\gamma_n], n\in \N)$ converges weakly to
  $\eta \,\mathfrak{g}_{\eta}[k,\gamma]$. This implies that:
  \[ \lim_{n\rightarrow \infty } \I[k_n, \gamma_n](\eta_n)
    = \lim_{n\rightarrow \infty } \int_\Omega \eta_n\,
    \mathfrak{g}_\eta[k_n, \gamma_n]\, \mathrm{d}\mu
    = \int_\Omega \eta\,
    \mathfrak{g}_\eta[k, \gamma]\, \mathrm{d}\mu
    = \I[k, \gamma](\eta).
  \]

Taking $(k_n, \gamma_n)=(k, \gamma)$ provides the continuity of $\I[k,
\gamma]$ and thus Theorem \ref{th:continuity-I}. Then, arguing as in the
end of the proof of Proposition \ref{prop:Re-stab}, we get Proposition
 \ref{prop:I-stab}.
\end{proof}

\subsection{Coupling and Pareto optimality}
\label{sec:proof-coupling}

We prove here Proposition~\ref{prop:coupling-optim}. 
We only consider the SIS model $\param=[(\Omega, \cf, \mu), k, \gamma]$,
as the kernel model can be handled similarly. We suppose throughout this
section that Assumption~\ref{hyp:k-g} holds.

The random variables we consider, are defined on a probability space, say $(\Omega_0,
\cf_0, \P)$. We recall an elementary result on conditional independence. Let $\ca$, $\cb$
and $\ch$ be $\sigma$-fields subsets of $\cf_0$, such that $\ch\subset \ca \cap \cb$.
Then, according to \cite[Theorem 8.9]{Kal21}, we have that for any integrable real-valued
random variable $X$ which is $\cb$-measurable:
\begin{equation}\label{eq:kall}
  \text{$\ca$ and $\cb$ are conditionally independent given $\ch$}
  \, \Longrightarrow\, \esp{X|\, \ca}=\esp{X|\, \ch}.
\end{equation}

We now state two technical lemmas.

\begin{lemma}[Measurability]\label{lem:measurability}
  Let $\param_1$ and $\param_2$ be coupled models with independent coupling $(X_1, X_2)$
  and $(Y_1, Y_2)$. Then the random variable $\gamma_1(X_1)$ is $\sigma(X_1) \cap
  \sigma(X_2)$-measurable. For any measurable function $v:\Omega_1\times \Omega_2\to \R$,
  such that $k_1(X_1,Y_1)v(Y_1, Y_2)$ is integrable, the random variable $
  \esp{k_1(X_1,Y_1)v(Y_1, Y_2)|X_1} $ is also $\sigma(X_1) \cap \sigma(X_2)$-measurable.
\end{lemma}

\begin{proof}
  The $\sigma(X_1) \cap \sigma(X_2)$-measurability of $\gamma_1(X_1)$ is an immediate
  consequence of the almost-sure equality $\gamma_1(X_1) = \gamma_2(X_2)$. Since
  $\esp{k(X_1,Y_1)v(Y_1, Y_2)|X_1}$ is $\sigma(X_1)$-measurable, it remains to prove that
  it is also $\sigma(X_2)$-measurable. Since $(X_1,X_2)$ is independent from $(Y_1,Y_2)$,
  the $\sigma$-fields $\ca=\sigma(X_1,X_2)$ and $\cb=\sigma(X_1,Y_1, Y_2)$ are
  conditionally independent given $\ch=\sigma(X_1)$. Using \eqref{eq:kall}, we deduce
  that:
  \[
    \esp{ k_1(X_1,Y_1)v(Y_1, Y_2) | X_1} = \esp{k_1(X_1,Y_1)v(Y_1, Y_2) | X_1,X_2}.
  \]
  Since $k_1(X_1,Y_1)=k_2(X_2,Y_2)$ $\P$-almost surely, we get:
  \begin{align*}
    \esp{ k_1(X_1,Y_1)v(Y_1, Y_2) | X_1} &= \esp{k_2(X_2,Y_2)v(Y_1, Y_2) | X_1,X_2} \\
					 &= \esp{k_2(X_2,Y_2)v(Y_1, Y_2) | X_2},
  \end{align*}
  where the last equality follows from another application of \eqref{eq:kall} with
  $\ca=\sigma(X_1, X_2)$, $\cb=\sigma(X_2,Y_1,Y_2)$ which are conditionally independent
  given $\ch=\sigma(X_2)$. The last expression is $\sigma(X_2)$ measurable, so the proof
  is complete.
\end{proof}

 In the following key lemma, we simply write
$H_i$ for $H[\param_i]$ for $H$ the loss functions $R_e$ and $\I$, the cost function
$C=\costu$ and the spectrum $\spec$.

\begin{lemma}\label{lem:equivalent_models}
  If $\param_1$ and $\param_2$ are coupled models, and if the functions $\eta_1\in
  \Delta_1$ and $\eta_2\in \Delta_2$ are coupled, then $\spec_1(\eta_1)\cup
  \{0\}=\spec_2(\eta_2)\cup \{0\}$ and for $H$ any one of the mappings $\costu$, $R_e$ or
  $\I$:
  \begin{equation}\label{eq:H1=H2}
    H_1(\eta_1) = H_2(\eta_2).
  \end{equation}
\end{lemma}

\begin{proof}
  Let $(X_1, X_2)$ and $(Y_1, Y_2)$ be two independent couplings, and assume that $\eta_1$
  and $\eta_2$ are coupled through the function $\eta$, see
  Remark~\ref{rem:couplage}~\ref{item:V-X1X2}:
  \begin{equation}
    \label{eq:h=h1=h2}
    \esp{\eta(X_1, X_2)|\, X_i}=\eta_i(X_i) \quad\text{for $i\in \{1, 2\}$}.
  \end{equation}

  \textit{Step 1:} The cost function ($H=\costu$). We directly have:
  \[
    C_1(\eta_1) = 1-\esp{\eta_1(X_1)} = 1-\esp{\eta(X_1, X_2)}
    = 1-\esp{\eta_2(X_2)} = C_2(\eta_2).
  \]

  \textit{Step 2:} The spectrum and the effective reproduction function ($H=R_e$). Set
  $\kk_i=k_i/\gamma_i$ for $i\in \{1, 2\}$. Let $\lambda$ be a non-zero eigenvalue of
  $T_{\kk_1 \eta_1}$ associated with an eigenvector $v_1$. Notice that $\kk(X_1,Y_1)
  \eta_1(Y_1) v(Y_1)$ is integrable thanks to the integrability condition from Assumption
  \ref{hyp:k-g}. By definition of eigenvectors, $v_1(X_1)$ is a version of the conditional
  expectation:
  \[
    \lambda^{-1} \esp{\kk_1(X_1,Y_1)\, \eta_1(Y_1) v_1(Y_1)|X_1}.
  \]
  By Lemma~\ref{lem:measurability} applied to the function $v(y_1,
  y_2)=(v_1\eta_1/\gamma_1)(y_1)$, the real-valued random variable $v_1(X_1)$ is
  $\sigma(X_1) \cap \sigma(X_2) $-measurable and thus $\sigma(X_2)$-measurable. Thanks to
  \eqref{eq:def-f-phi0}, there exists $v_2$ such that $v_2(X_2) = v_1(X_1)$ almost surely.
  Since $(Y_1, Y_2)$ is distributed as $(X_1, X_2)$, we deduce that \eqref{eq:h=h1=h2}
  holds also with $(X_1, X_2)$ replaced by $(Y_1, Y_2)$ and that $v_2(Y_2) =
  v_1(Y_1)$ almost surely. Recall that $\kk_i=k_i/\gamma_i$, so that $\kk_1(X_1, Y_1) =
  \kk_2(X_2, Y_2)$ almost surely. We may now compute:
  \begin{equation}\label{eq:from1to2}
    \begin{aligned}
      \lambda v_2(X_2) &= \lambda v_1(X_1) \\
		       &= \esp{ k_1(X_1,Y_1)\,\eta_1(Y_1) v_1(Y_1) | X_1} \\
		       &= \esp{ \kk_1(X_1,Y_1)\,\eta(Y_1, Y_2) v_1(Y_1) | X_1}
		       &\text{(de-conditioning on~$(Y_1, X_1)$)}\\
		       &= \esp{\kk_1(X_1,Y_1)\,\eta(Y_1, Y_2) v_1(Y_1) | X_2}
		       & \text{(Lemma \ref{lem:measurability})} \\
		       &= \esp{\kk_2(X_2,Y_2)\, \eta(Y_1, Y_2) v_2(Y_2) | X_2}
		       & \text{(a.s. equality)}\\
		       &= \esp{\kk_2(X_2,Y_2)\, \eta_2(Y_2)v_2(Y_2)|X_2}
		       &\text{(conditioning on~$(Y_2, X_2)$)} \\
		       &= T_{\kk_2\eta_2}v_2(X_2).
    \end{aligned}
  \end{equation}
  Since the distribution of~$X_2$ is~$\mu_2$, we have $\lambda v_2
  = T_{\kk_2\eta_2}v_2$ $\mu_2$-almost surely. Therefore~$\lambda$ is also an
  eigenvalue for~$T_{\kk_2\eta_2}$. By symmetry we deduce that the spectrum up to $\{0\}$
  of~$T_{\kk_1\eta_1}$ and~$T_{\kk_2\eta_2}$ coincide, that is $\spec_1(\eta_1)\cup
  \{0\}=\spec_2(\eta_2)\cup \{0\}$, and in particular the spectral radius coincide.

 \medskip

  \textit{Step 3:} The total proportion of infected population function ($H=\I$). We
  assume without loss of generality that~$\rho(\Tinf_{k_1/\gamma_1}) >1$, which is
  equivalent to~$\rho(\Tinf_{k_2/\gamma_2})>1$, thanks to~\eqref{eq:H1=H2} with~$H=R_e$
  and~$\eta_1=\eta_2=1$. Let~$g_1 = \mathfrak{g}_{\eta_1}$ be the maximal equilibrium for
  the model~$\param_1$. Since~$F_{\eta_1}(g_1)=0$, see~\eqref{eq:F(g)=0}, we have:
  \begin{equation}\label{eq:continuity-I0-g1}
    g_1 = \frac{\Tinf_{k_1}(\eta_1 g_1)}{\gamma_1 + \Tinf_{k_1}(\eta_1 g_1)}\cdot
  \end{equation}
  By Lemma~\ref{lem:measurability}, this implies that~$g_1(X_1)$ is $\sigma(X_1)\cap
  \sigma(X_2)$ measurable. Thus, there exists~$g'_2$ such that $g'_2(X_2) =
  g_1(X_1)$ $\P$-almost surely.. Therefore, by the same computation as
  in~\eqref{eq:from1to2}:
  \[
    \Tinf_{k_1}(\eta_1 g_1)(X_1) = \Tinf_{k_2}(\eta_2 g'_2)(X_2) \quad \P-\text{a.s.}
  \]
  We set:
  \begin{equation}\label{eq:continuity-I0-g2}
    g_2 = \frac{\Tinf_{k_2}(\eta_2 g'_2)}{\gamma_2 + \Tinf_{k_2}(\eta_2 g'_2)}\cdot
  \end{equation}
  Then, we deduce from \eqref{eq:continuity-I0-g1} that $g_2(X_2)=g'_2(X_2)$~$\P$-almost
  surely, that is $g_2 = g'_2$~$\mu_2$-almost surely. Thus \eqref{eq:continuity-I0-g2}
  holds with $g'_2$ replaced by $g_2$. In other words,~$g_2$ satisfies~\eqref{eq:F(g)=0}:
  it is an equilibrium for the model given by~$\param_2$.

  Let us now prove that~$g_2$ is in fact the maximal equilibrium. Since
  $g_2(X_2) = g_1(X_1)$~$\P$-almost surely and $g_1(X_1)$ is~$\sigma(X_1)\cap
  \sigma(X_2)$-measurable, we deduce from Remark~\ref{rem:couplage}~\ref{item:V1},
  that~$(1-g_1)$ and~$(1-g_2)$ are coupled, so $R_e[\param_1](1-g_1) =
  R_e[\param_2](1-g_2)$, by Property~\eqref{eq:H1=H2} applied to~$H=R_e$. Since~$R_0>1$
  and~$g_1$ is the maximal equilibrium for~$\param_1$, we deduce from Proposition
  \ref{prop:caract-g} that~$R_e[\param_1](1- g_1)=1$. Using again Proposition
  \ref{prop:caract-g}, this gives that~$g_2$ is the maximal equilibrium for~$\param_2$.

  We may now compute:
  \begin{align*}
    \I_1(\eta_1) &= \esp{\eta_1(X_1)\, g_1(X_1)} &\\
		 &= \esp{\eta(X_1, X_2)\, g_1(X_1)} & \text{(deconditioning on~$X_1$)} \\
		 &= \esp{\eta(X_1, X_2)\, g_2(X_2)} & \text{($\as$ equality)} \\
		 &= \esp{\eta_2(X_2)\,g_2(X_2)} & \text{(conditioning on~$X_2$)} \\
		 &= \I_2(\eta_2),
  \end{align*}
  thus~\eqref{eq:H1=H2} holds for~$H=\I$, and the proof is complete.
\end{proof}

We now give the proof of Proposition~\ref{prop:coupling-optim}. Its first part is an
elementary consequence of the Lemma~\ref{lem:equivalent_models}; and the second part is a
direct consequence of Remark~\ref{rem:couplage}~\ref{item:V2}.

\printbibliography

\end{document}